\documentclass[12pt]{article}

\usepackage{amsmath, amsthm, amssymb, accents,amscd}
\usepackage{enumerate}
\usepackage{pdflscape}
\usepackage{caption}

\usepackage{ifpdf}
\ifpdf
\usepackage[pdftex]{graphicx}
\else
\usepackage[dvips]{graphicx}
\fi
\usepackage{tikz}
 	 \usetikzlibrary{arrows,backgrounds}
\usepackage[all]{xy}
\usepackage{tikz-cd}
\usepackage{multicol}

\input xy
\xyoption{all} 

\usepackage[pdftex,plainpages=false,hypertexnames=false,pdfpagelabels]{hyperref}
\newcommand{\arxiv}[1]{\href{http://arxiv.org/abs/#1}{\tt arXiv:\nolinkurl{#1}}}
\newcommand{\arXiv}[1]{\href{http://arxiv.org/abs/#1}{\tt arXiv:\nolinkurl{#1}}}

\newcommand{\googlebooks}[1]{(preview at \href{http://books.google.com/books?id=#1}{google books})}

\usepackage{xcolor}
\definecolor{dark-red}{rgb}{0.7,0.25,0.25}
\definecolor{dark-blue}{rgb}{0.15,0.15,0.55}
\definecolor{medium-blue}{rgb}{0,0,.8}
\definecolor{DarkGreen}{RGB}{0,150,0}
\definecolor{rho}{named}{red}
\hypersetup{
   colorlinks, linkcolor={purple},
   citecolor={medium-blue}, urlcolor={medium-blue}
}

\usepackage{longtable}
\usepackage{fullpage}
\theoremstyle{plain}
\newtheorem{thm}{Theorem}[section]
\newtheorem*{thm*}{Theorem}
\newtheorem{thmalpha}{Theorem}

\newtheorem{cor}[thm]{Corollary}

\newtheorem*{cor*}{Corollary}

\newtheorem*{conj*}{Conjecture}
\newtheorem*{lem*}{Lemma}
\newtheorem{lem}[thm]{Lemma}
\newtheorem{prop}[thm]{Proposition}

\newtheorem*{quest*}{Question}
\newtheorem*{claim*}{Claim}

\theoremstyle{definition}

\newtheorem{defn}[thm]{Definition}

\theoremstyle{remark}

\newtheorem{ex}[thm]{Example}
\newtheorem{rem}[thm]{Remark}
\newtheorem{remark}[thm]{Remark}

\DeclareMathOperator{\Aut}{Aut}

\DeclareMathOperator{\Conf}{Conf}
\DeclareMathOperator{\tDiff}
{\mathrm{D}\!\widetilde{\,i\hspace{1.5mm}}\hspace{-1.5mm}\mathrm{ff}} 
\DeclareMathOperator{\Diff}{Diff}
\DeclareMathOperator{\End}{End}

\DeclareMathOperator{\Ann}{Ann}

\DeclareMathOperator{\tAnn}{A\widetilde{\!\!\phantom{\imath}n\phantom{\imath}\!\!}n}
\DeclareMathOperator{\tRot}{R\!\!\widetilde{\phantom{i}o\phantom{i}}\!\!t}
\DeclareMathOperator{\tMob}{M\!\!\widetilde{\phantom{i}\text{\"o}\phantom{i}\!}\!b}

\DeclareMathOperator{\spann}{span}
\DeclareMathOperator{\id}{id}

\DeclareMathOperator{\Mob}{M\ddot{o}b}
\DeclareMathOperator{\Univ}{Univ}
\DeclareMathOperator{\tUniv}{U\widetilde{ni}v}
\DeclareMathOperator{\Exp}{Exp}
\DeclareMathOperator{\Vir}{Vir}
\DeclareMathOperator{\supp}{supp}

\newcommand{\comment}[1]{}

\newcommand{\be}{\begin{enumerate}[label=(\arabic*)]}
\newcommand{\ee}{\end{enumerate}}

\newcommand{\cD}{\mathcal{D}}

\newcommand{\cC}{\mathcal{C}}

\newcommand{\norm}[1]{\left\| #1 \right\|}
\newcommand{\abs}[1]{\left| #1 \right|}

\def\acts{\hspace{.1cm}{\setlength{\unitlength}{.30mm}\linethickness{.09mm}
                        \begin{picture}(8,8)(0,0)\qbezier(7,6)(4.5,8.3)(2,7)\qbezier(2,7)(-1.5,4)(2,1)\qbezier(2,1)(4.5,-.3)(7,2)
                                                 \qbezier(7,6)(6.1,7.5)(6.8,9)\qbezier(7,6)(5,6.1)(4.2,4.4)
                        \end{picture}\hspace{.1cm}}}

\def\semicolon{;}
\def\applytolist#1{
    \expandafter\def\csname multi#1\endcsname##1{
        \def\multiack{##1}\ifx\multiack\semicolon
            \def\next{\relax}
        \else
            \csname #1\endcsname{##1}
            \def\next{\csname multi#1\endcsname}
        \fi
        \next}
    \csname multi#1\endcsname}

\def\calc#1{\expandafter\def\csname c#1\endcsname{{\mathcal #1}}}
\applytolist{calc}QWERTYUIOPLKJHGFDSAZXCVBNM;
\def\bbc#1{\expandafter\def\csname bb#1\endcsname{{\mathbb #1}}}
\applytolist{bbc}QWERTYUIOPLKJHGFDSAZXCVBNM;
\def\bfc#1{\expandafter\def\csname bf#1\endcsname{{\mathbf #1}}}
\applytolist{bfc}QWERTYUIOPLKJHGFDSAZXCVBNM;
\def\sfc#1{\expandafter\def\csname s#1\endcsname{{\sf #1}}}
\applytolist{sfc}QWERTYUIOPLKJHGFDSAZXCVBNM;

\newcommand{\noshow}[1]{}
\newcommand{\MR}[1]{}

\usetikzlibrary{shapes}
\usetikzlibrary{backgrounds}
\usetikzlibrary{decorations,decorations.pathreplacing,decorations.markings}
\usetikzlibrary{fit,calc,through}
\usetikzlibrary{external}
\tikzset{
	super thick/.style={line width=3pt}
}
\tikzstyle{shaded}=[fill=red!10!blue!20!gray!30!white]
\tikzstyle{unshaded}=[fill=white]
\tikzstyle{empty box}=[circle, draw, thick, fill=white, opaque, inner sep=2mm]
\tikzstyle{annular}=[scale=.7, inner sep=1mm, baseline]
\tikzstyle{rectangular}=[scale=.75, inner sep=1mm, baseline=-.1cm]
\tikzstyle{mid>}=[decoration={markings, mark=at position 0.5 with {\arrow{>}}}, postaction={decorate}]
\tikzstyle{mid<}=[decoration={markings, mark=at position 0.5 with {\arrow{<}}}, postaction={decorate}]
\tikzstyle{over}=[double, draw=white, super thick, double=]

\title{Every conformal net has an associated unitary VOA}
\author{Andr\'e G. Henriques, James E. Tener}
\date{}

\begin{document}

\maketitle

\begin{abstract}
    Unitary vertex operator algebras (VOAs) and conformal nets
    are the two most prominent mathematical axiomatizations of two-dimensional unitary chiral conformal field theories. 
    They are conjectured to be equivalent, but a rigorous comparison has proven challenging.
    We resolve one direction of the conjecture by showing that every conformal net has an associated unitary VOA.
   We also show that every representation of a conformal net in which the generator of rotation acts with discrete spectrum and finite-dimensional eigenspaces yields a unitary module of the corresponding VOA.
   A talk describing our results is available at:  \begin{verbatim}
   https://www.youtube.com/watch?v=f_LhNSeiiaE
   \end{verbatim}
\end{abstract}

\tableofcontents

\section{Introduction}
There currently exist two main, well-developed, mathematical formalisms for $2d$ unitary chiral conformal field theory:
the formalism of unitary vertex operator algebras \cite{DongLin14,CKLW18}
 (see \cite{Kac98,FrenkelBenZvi04,LepowskyLi04} for standard textbooks on VOAs),
and the formalism of chiral conformal nets \cite{BuchholzMackTodorov,BuchholzSchulz-Mirbach90,FrReSc92,GabbianiFrohlich93,BrunettiGuidoLongo93,Wa98}, building on the foundational work of Haag and Kastler \cite{HaKa64}.\footnote{In this paper, we use the words chiral conformal net and conformal net interchangeably.}

From a physical point of view, a chiral conformal field theory is not a stand-alone quantum field theory. It is instead a boundary theory for a 3-dimensional bulk TQFT.\footnote{This is in fact not a 3d TQFT in the sense usually understood in mathematics, as it suffers from a `gravitational anomaly': it is a metric dependent theory, but a change of the space-time metric only affects the value of the correlators by a universal scalar \cite[\S2]{Witten89}.}
We quickly explain, with a physicist audience in mind, the notions of vertex operator algebras and of chiral conformal net:

\begin{itemize}
\item\underline{Vertex o}p\!\!\underline{\,\,erator al}g\!\underline{\,ebras}:
The VOA distils and axiomatizes the properties of the vector space of  {\bf boundary local operators}, along with their {\bf operator product expansion}. It consists of a vector space $V$ equipped with a `state-operator correspondence'
\[
Y(-,z) : V\,\,\to\,\,\End(V)[[z,z^{-1}]]
\] 
and a specified vector $\nu\in V$ that is the stress-energy tensor (see Definition \eqref{def: definition of N-graded VOA}).
A \emph{unitary} VOA comes additionally equipped with an anti-linear involution $\Theta:V\to V$ that implements the coordinate change $z\mapsto \bar z$.
For $a,b,c\in V$, and $z\in \bbC$, the inner product $\langle \Theta(c),Y(a,z)b\rangle$ describes the correlator $\langle a(z) b(0) c(\infty)\rangle$ on $\bbC P^1$, with the three local operators $a, b, c$ inserted as the points $z, 0, \infty$, respectively, and the whole thing placed on the boundary of a $3$-ball.

\item\underline{Conformal nets}:
Chiral conformal nets axiomatize the {\bf Hilbert space associated to a Cauchy slice} of a cylindrical spacetime $\bbD^2\times \bbR$. The Hilbert space is called the `vacuum sector' of the conformal net, and denoted $H_0$. Following Haag-Kastler \cite{HaKa64}, for every open subset $\cO\subset \bbD^2$ of the disc, one has an associated {\bf von Neumann algebra of observables} localised in $\cO$.
If $\cO\subset \bbD^2$ is a regular neighbourhood of an interval $I=\cO\cap \partial\bbD^2$, then the bulk geometry of $\cO$ does not affect the algebra of observables: we get a map
\[
\cA:\{\text{intervals in $S^1$}\}\to 
\{\text{von Neumann algebras on $H_0$}\},
\]
which is the data of the conformal net (see Definition~\ref{def: definition of conformal net} for more details).
Conformal symmetry is encoded by an action of $\Diff(S^1)$ on $H_0$.
\end{itemize}

%\noindent
In the landmark paper \cite{CKLW18},
it was shown that the formalism of 
unitary vertex operator algebras
and the formalism of chiral conformal nets
are strongly interconnected in the sense that every unitary vertex operator algebra $V$ satisfying a certain technical assumption called \emph{strong locality} yields a chiral conformal net $\cA_V$, and that it is possible to recover $V$ from $\cA_V$.
We introduce a slight simplification of strong locality, which we call \emph{AQFT-locality}, that avoids a certain technical requirement (polynomial energy bounds) used in \cite{CKLW18}, but is still sufficient for generating a chiral conformal net $\cA_V$.

The main result of our paper is that every chiral conformal net $\cA$ has an associated unitary vertex operator algebra $V_\cA$. 
(Throughout this paper, we assume that the generator of rotations $L_0$ acts on the vaccum Hilbert space $H_0$ with finite-dimensional eigenspaces.)
We prove that the vertex algebra $V_\cA$ associated to a conformal net $\cA$ is AQFT-local, and that $\cA\cong \cA_{V_\cA}$.
\begin{thmalpha}
    For every conformal net $\cA$ with vacuum Hilbert space $H_0$, there is a natural structure of a unitary vertex operator algebra on the space of finite-energy vectors $V \subset H_0$.
    The unitary VOA $V$ is AQFT-local, and $\cA = \cA_{V}$.
\end{thmalpha}
This resolves a slightly weakened version of \cite[Conj. 9.4]{CKLW18}.
The vertex algebra structure on $V$ is established in Corollary~\ref{cor: n point functions from conformal nets}, conformal symmetry of the vertex algebra is Theorem~\ref{thm: unitary VOA from conformal net}, and unitarity is Theorem~\ref{thm: unitarity of VOA}.
The reconstruction of $\cA$ from $V$ is established in Theorem~\ref{thm: inverse constructions}, which also establishes that for every AQFT-local unitary VOA $V$ we have $V_{\cA_V}\cong V$:
\begin{thmalpha}
Let $V$ be a unitary VOA which is AQFT-local, and let $\cA_V$ be the associated conformal net. Then the unitary VOA associated to $\cA_V$ is canonically the same as $V$.
\end{thmalpha}
As in \cite[Conj. 8.18]{CKLW18}, we conjecture that every unitary vertex operator algebra is AQFT-local, and that the inclusion $\{$conformal nets$\}\subset\{$unitary VOAs$\}$ that we get is in fact an equality. Summarising, we have:
{\sf
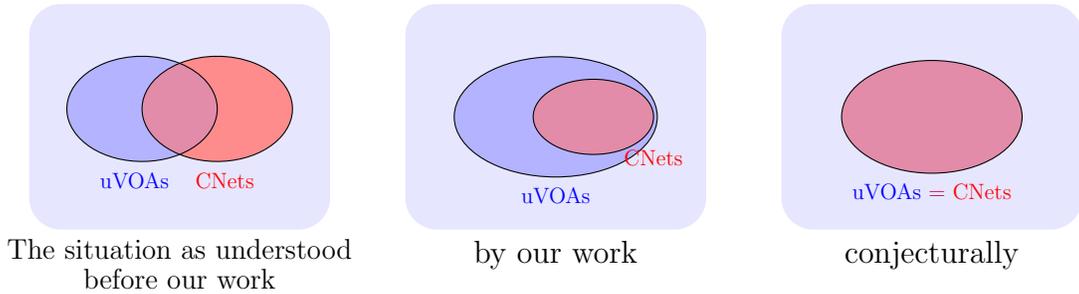
\[
\begin{tikzpicture}[black]
\fill[blue!10, rounded corners = 12] (0,0) -- +(0,3) -- +(4,3) -- +(4,0) --node[below, black
, scale=.9]{\rm The situation as understood}
node[below, yshift=-11, black
, scale=.9]{\rm before our work} cycle;
\fill[blue!10, rounded corners = 12] (5,0) -- +(0,3) -- +(4,3) -- +(4,0) --node[below, black
]{\rm by our present work} cycle;
\fill[blue!10, rounded corners = 12] (10,0) -- +(0,3) -- +(4,3) -- +(4,0) --node[below, black
]{\rm conjecturally} cycle;
\fill[red!45] (2,1) arc (-120:120:1 and .7) arc (60:-60:1 and .7);
\fill[purple!45] (2,1) arc (-120:-240:1 and .7) arc (60:-60:1 and .7);
\fill[blue!30] (2,1) arc (-120:-240:1 and .7) arc (60:300:1 and .7);
\draw (2,1) arc (-120:120:1 and .7);
\draw (2,1) arc (-120:-240:1 and .7);
\draw (2,1) arc (-60:60:1 and .7);
\draw (2,1) arc (300:60:1 and .7);
\node[blue, scale=.7] at (1.4,.65) {uVOAs};
\node[red, scale=.7] at (2.6,.65) {CNets};
\filldraw[fill=blue!30] (7,1.5) circle (1.35 and .8); 
\node[blue, scale=.7] at (7,.45) {uVOAs};
\filldraw[fill=purple!45] (7.5,1.5) circle (.8 and .5);
\node[red, scale=.7] at (8.3,.95) {CNets};
\filldraw[fill=purple!45] (12,1.5) circle (1.2 and .75);
\node[purple, scale=.7] at (12,.5) {\textcolor{blue}{uVOAs} = \textcolor{red}{CNets}};
\end{tikzpicture}
\]}
\indent
We now explain our construction of a vertex algebra from a conformal net.
Fix a conformal net $\cA$, with vacuum Hilbert space $H_0$.
The construction of the vertex algebra structure on the space $V\subset H_0$ of finite energy vectors proceeds in two main steps.
Given a family of disjoint intervals $I_j$ embedded \emph{inside} the unit disc, and bounded operators $x_j \in \cA(I_j)$\footnote{We use here the `coordinate-free' description of conformal nets, which assigns von Neumann algebras to arbitrary intervals, not just subintervals of the unit circle \cite{BartelsDouglasHenriques15}. Alternatively, one can parametrize the intervals $I_j$ by intervals of $S^1$; this parametrization should then be considered as part of the data of~$x_j$.}, we construct in Section~\ref{sec: worms}  a vector $|x_1 \ldots x_n\rangle \in H_0$
\begin{equation}\label{eqn: intro worm diagram}
\begin{tikzpicture}[scale=0.4,baseline]
  % Draw the filled circle
  \filldraw[fill=red!10!blue!20!gray!30!white, draw=black, thick] (0,0) circle (3cm);
  % Squiggle 1: top left, gentle wave
  \draw[thick] (-1.7,1) to[out=30,in=150] (-1.2,1.1) to[out=-30,in=210] (-0.8,1);
  \node at (-1.2,1.4) {${\scriptstyle x_1}$};
  % Squiggle 2: top right, tighter wave
  \draw[thick] (0.5,1.7) to[out=60,in=120] (0.9,1.7) to[out=-60,in=240] (1.3,1.6);
  \node at (0.9,2) {${\scriptstyle x_2}$};
  % Squiggle 3: center, rotated wave
  \draw[thick,rotate around={20:(0,0)}] (0,0.2) to[out=60,in=120] (0.4,0.2) to[out=-60,in=240] (0.8,0.1);
  \node at (0.4,0.6) {${\scriptstyle x_3}$};
  % Squiggle 4: bottom left, flat wiggle
  \draw[thick] (-1.5,-1.2) to[out=10,in=170] (-1,-1.3) to[out=-10,in=190] (-0.5,-1.2);
  \node at (-1,-0.9) {${\scriptstyle x_4}$};
\end{tikzpicture}
\quad\, \mapsto \,\quad |x_1 \ldots x_4 \rangle \in H_0.
\end{equation}
Informally, we call these `worm-shaped insertions' in the unit disc, in contrast to the point insertions considered below.
If each of the intervals $I_j$ lies on the boundary of the disc we have $|x_1 \ldots x_n\rangle = x_1 \ldots x_n\Omega$.
In general, the vectors $|x_1 \ldots x_n\rangle$ are determined by that property in combination with a certain compatibility condition with the action of the \emph{semigroup of annuli}, which we now describe.

An \emph{annulus} is represented by a pair of Jordan curves $C^\infty(S^1) \to \bbC$, with one curve encircling the other. The collection of all annuli, denoted $\Ann$, forms a semigroup under the operation of conformal welding.
As in \cite{HenriquesTener24ax}, we allow the Jordan curves to touch but not cross, so that the annulus might be `thin' along part of the boundary. The following is an example of a thin annulus:
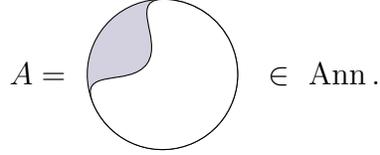
\[
A=
\begin{tikzpicture}[baseline={([yshift=-.5ex]current bounding box.center)},rotate=-35,transform shape]
	\coordinate (a) at (120:1cm);
	\coordinate (b) at (240:1cm);
	\coordinate (c) at (180:.25cm);
% BIG DISC
	\fill[fill=red!10!blue!20!gray!30!white] (0,0) circle (1cm);
	\draw (0,0) circle (1cm);
% CURVED BOUNDARY REGION
	\fill[fill=white] (a)  .. controls ++(210:.6cm) and ++(90:.4cm) .. (c) .. controls ++(270:.4cm) and ++(150:.6cm) .. (b) -- ([shift=(240:1cm)]0,0) arc (240:480:1cm);
	\draw ([shift=(240:1cm)]0,0) arc (240:480:1cm);
	\draw (a) .. controls ++(210:.6cm) and ++(90:.4cm) .. (c);
	\draw (b) .. controls ++(150:.6cm) and ++(270:.4cm) .. (c);
\end{tikzpicture}
\,\,\in\,\,\Ann.
\]
We constructed in loc.~cit central extensions $\tAnn_c \to \Ann$ for every $c\in \bbC$, and it was shown in \cite{HenriquesTenerIntegratingax} that every unitary\footnote{The unitarity condition restricts the cental charge $c$ to take its values in $\bbR_{\ge0}$.} positive energy representation of the Virasoro algebra with central charge $c$ exponentiates to a holomorphic representation of $\tAnn_c$ by bounded operators on the Hilbert space completion of the representation.
Similarly to \eqref{eqn: intro worm diagram}, given a centrally-extended annulus $\underline{A} \in \tAnn_c$ with disjoint intervals $I_j$ in the interior of $A$, and worm-shaped insertions $x_j \in \cA(I_j)$, we construct a bounded operator $\underline{A}[x_1 \ldots x_n]\in B(H_0)$:
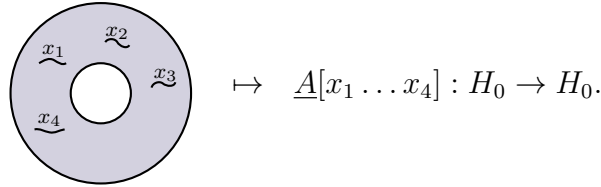
\begin{equation*}
\begin{tikzpicture}[scale=0.4,baseline]
  % Draw the filled circle
  \filldraw[fill=red!10!blue!20!gray!30!white, draw=black, thick] (0,0) circle (3cm);
  \filldraw[fill=white, draw=black, thick] (0,0) circle (1cm);
  % worm 1: top left
  \begin{scope}[xshift=-10]
  \draw[thick] (-1.7,1) to[out=30,in=150] (-1.2,1.1) to[out=-30,in=210] (-0.8,1);
  \node at (-1.2,1.4) {${\scriptstyle x_1}$};
  \end{scope}
  % worm 2: top right
  \begin{scope}[xshift=-10]
  \draw[thick] (0.5,1.7) to[out=60,in=120] (0.9,1.7) to[out=-60,in=240] (1.3,1.6);
  \node at (0.9,2) {${\scriptstyle x_2}$};
  \end{scope}
  % worm 3: center
  \begin{scope}[xshift=50]
  \draw[thick,rotate around={20:(0,0)}] (0,0.2) to[out=60,in=120] (0.4,0.2) to[out=-60,in=240] (0.8,0.1);
  \node at (0.4,0.6) {${\scriptstyle x_3}$};
  \end{scope}
  % worm 4: bottom left
  \begin{scope}[xshift=-20]
  \draw[thick] (-1.5,-1.2) to[out=10,in=170] (-1,-1.3) to[out=-10,in=190] (-0.5,-1.2);
  \node at (-1,-0.9) {${\scriptstyle x_4}$};
  \end{scope}
\end{tikzpicture}
\quad \mapsto \,\,\,\,
\underline{A}[x_1 \ldots x_4]: H_0 \to H_0.
\end{equation*}
Moreover, if $\underline{A}$ is supported in an interval $I \subset S^1$ (the boundary parametrizations agree outside of $I$ \footnote{The property of $\underline{A}$ being supported in $I$ is in fact slightly stronger than simply requiring that the two boundary parametrizations to agree outside of $I$ (see Definition~\ref{def: 4.12}).}), then $\underline{A}[x_1 \ldots x_n] \in \cA(I)$.

The next step of our construction, in Section~\ref{sec:Point insertions}, assigns vectors (resp. operators) to discs (resp. annuli) with \emph{point} insertions\footnote{These are called \emph{local operators} in the physics terminology.}.
More precisely, given finite-energy vectors $v_1, \ldots, v_n \in V \subset H_0$ and distinct points $z_1, \ldots, z_n$ lying in the interior of the unit disc, we construct vectors $| v_1(z_1) \ldots v_n(z_n) \rangle \in H_0$.
Similarly, if $\underline{A} \in \tAnn_c$ and the points $z_j$ lie in the interior of the annulus, then we define a bounded operator $\underline{A}[v_1(z_1) \ldots v_n(z_n)]\in B(H_0)$, which lies in $\cA(I)$ when $\underline{A}$ is localized in $I$.
A key insight of this paper is that point insertions can be realized as finite linear combinations of worm-shaped insertions:
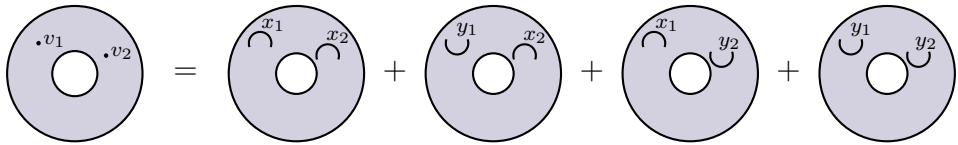
\begin{equation*}
\begin{tikzpicture}[scale=0.3,baseline=-2]
  % Draw the filled circle
  \filldraw[fill=red!10!blue!20!gray!30!white, draw=black, thick] (0,0) circle (3);
  \filldraw[fill=white, draw=black, thick] (0,0) circle (1);
\fill (140:2.1) circle (.1);
    \node[xshift=6,yshift=1] at (140:2.1) {$\scriptstyle v_1$};
\fill (30:1.6) circle (.1);
    \node[xshift=6,yshift=1] at (30:1.6) {$\scriptstyle v_2$};
\end{tikzpicture}
\,\,\,\,=\,\,\,\,
\begin{tikzpicture}[scale=0.3,baseline=-2]
  % Draw the filled circle
  \filldraw[fill=red!10!blue!20!gray!30!white, draw=black, thick] (0,0) circle (3);
  \filldraw[fill=white, draw=black, thick] (0,0) circle (.9);
\draw[thick] (140:2.1) + (-20:.5) arc (-20:200:.5);
    \node[xshift=5,yshift=7] at (140:2.1) {$\scriptstyle x_1$};
\draw[thick](30:1.6) + (-20:.5) arc (-20:200:.5);
    \node[xshift=4,yshift=7] at (30:1.6) {$\scriptstyle x_2$};
\end{tikzpicture}
\,\,+\,\,
\begin{tikzpicture}[scale=0.3,baseline=-2]
  % Draw the filled circle
  \filldraw[fill=red!10!blue!20!gray!30!white, draw=black, thick] (0,0) circle (3);
  \filldraw[fill=white, draw=black, thick] (0,0) circle (.9);
\draw[thick] (140:2.1) + (20:.5) arc (20:-200:.5);
    \node[xshift=4,yshift=5] at (140:2.1) {$\scriptstyle y_1$};
\draw[thick](30:1.6) + (-20:.5) arc (-20:200:.5);
    \node[xshift=4,yshift=7] at (30:1.6) {$\scriptstyle x_2$};
\end{tikzpicture}
\,\,+\,\,
\begin{tikzpicture}[scale=0.3,baseline=-2]
  % Draw the filled circle
  \filldraw[fill=red!10!blue!20!gray!30!white, draw=black, thick] (0,0) circle (3);
  \filldraw[fill=white, draw=black, thick] (0,0) circle (.9);
\draw[thick] (140:2.1) + (-20:.5) arc (-20:200:.5);
    \node[xshift=5,yshift=7] at (140:2.1) {$\scriptstyle x_1$};
\draw[thick](30:1.6) + (20:.5) arc (20:-200:.5);
    \node[xshift=3,yshift=4.5] at (30:1.6) {$\scriptstyle y_2$};
\end{tikzpicture}
\,\,+\,\,
\begin{tikzpicture}[scale=0.3,baseline=-2]
  % Draw the filled circle
  \filldraw[fill=red!10!blue!20!gray!30!white, draw=black, thick] (0,0) circle (3);
  \filldraw[fill=white, draw=black, thick] (0,0) circle (.9);
\draw[thick] (140:2.1) + (20:.5) arc (20:-200:.5);
    \node[xshift=4,yshift=5] at (140:2.1) {$\scriptstyle y_1$};\draw[thick](30:1.6) + (20:.5) arc (20:-200:.5);
    \node[xshift=3,yshift=4.5] at (30:1.6) {$\scriptstyle y_2$};
\end{tikzpicture}
\end{equation*}
The vectors $| v_1(z_1) \ldots v_n(z_n) \rangle$ depend holomorphically on the $z_i$, and satisfy a geometric analog of the vertex algebra axioms, which yields a vertex algebra structure on $V$; this is described in Section~\ref{sec: geometric vertex algebras}.
We then show in Section~\ref{sec: conformal vector} that there is a natural conformal vector for $V$, and that the associated representation of Virasoro algebra exponentiates to the diffeomorphism symmetry of the original net $\cA$. In Section~\ref{sec: unitarity}, we show that the resulting VOA satisfies all the axioms of a \emph{unitary} vertex operator algebra.
Finally, in Section~\ref{sec: smeared fields}, we show that the constructions $\cA\mapsto\cA_V$ and $V\mapsto \cA_V$ are inverse to each other.

Let us call a map $f:\bbD\to \bbD$ \emph{univalent} if it injective, holomorphic in the interior, smooth all the way to the boundary, and has everywhere non-zero derivative.
In future work, we will establish an equivalent condition to AQFT-locality which is expressed purely in terms of the VOA $V$ and its Hilbert space completion $H_0$:
a unitary $V$ VOA is \emph{integrable} if for every $v\in V$, every univalent map $f:\bbD\to \bbD$ sending $0$ to $0$, and every point $z\in \mathring \bbD \setminus f(\bbD)$ the map $V \to H_0$
\[
u \mapsto Y(v,z)\;\! f {\cdot} u
\]
extends to a bounded linear map $H_0 \to H_0$.
We will show that a unitary VOA is AQFT-local if and only if it is integrable, providing an alternative description of the class of VOAs that correspond to conformal nets.

Consider now a representation $(K,\pi)$ of the conformal net $\cA$
in which the generator $L_0$ of rotation acts with discrete spectrum and finite-dimensional eigenspaces,
and let $M \subset K$ be its subspace of finite-energy vectors.
In Section~\ref{sec: modules},
we show that $M$ naturally acquires the structure of a unitary $V$-module.
The compatibility between the representation $\pi$ of $\cA$ and the action $Y^M$ of $V$ on $M$ is expressed in terms of the smeared fields $Y(v,f)$ and $Y^M(v,f)$:

\begin{thmalpha}
   Let $\cA$ be a conformal net, and $V$ the associated unitary VOA. Let $(K,\pi)$ be a representation of $\cA$ in which $L_0$ acts with discrete spectrum and finite-dimensional eigenspaces, and let
   $M\subset K$ be its subspace of finite energy vectors. 
   Then there is a unique unitary $V$-module structure on $M$
   \[
   Y^M:V\to \End(M)[[z,z^{-1}]],
   \]
   characterised by the requirement that
   \[
   \pi_I(Y(v,f)) = Y^M(v,f)
   \]
   for every function $f \in C^\infty(S^1)$ with support in an interval $I\subset S^1$.
\end{thmalpha}

\noindent
This is proven in Theorem~\ref{thm: CN rep applied to smeared field is smeared field} and Corollary~\ref{cor: Kfe strong module}.

The article is organised as follows.
In Section~\ref{sec: conformal nets} we give some background information on conformal nets, including the coordinate-free version of conformal nets introduced in \cite{BartelsDouglasHenriques15}, and in Section~\ref{sec: VOAs} we provide some background on vertex operator algebras.
Section~\ref{sec: semigroup of annuli} introduces the semigroup of annuli, along with its central extensions and representation theory \cite{HenriquesTener24ax,HenriquesTenerIntegratingax}.
In Section~\ref{sec: worms}, we construct from an arbitrary conformal net vectors (and operators) associated to worm-shaped insertions in discs (and annuli).
In Section~\ref{sec:Point insertions}, we construct the analogous vectors and operators associated to point insertions in discs and annuli.
In Section~\ref{sec: The vertex algebra associated to a conformal net}, we construct a unitary VOA using the point insertions considered in the previous section.
In Section~\ref{sec: smeared fields}, we show that the VOA associated to a conformal net is AQFT-local, and that the smeared fields generate the local algebras of the original net.
Finally, in Section~\ref{sec: modules} we consider representations of conformal nets, and show that they produce modules for the associated VOA.

\subsection*{Acknowledgements}

The second author was supported by ARC Discovery Project DP200100067.
For the purpose of Open Access, the authors have applied a CC BY public copyright licence
to any Author Accepted Manuscript (AAM) version arising from this submission.

\section{Conformal nets}\label{sec: conformal nets}

Let $\bbD=\{z\in\bbC:|z|\le 1\}$ be the \emph{standard unit disc}, and
$S^1=\partial\bbD=\{z\in\bbC:|z|=1\}$ the \emph{standard circle}.
Let $\Diff(S^1)$ be the group of orientation preserving diffeomorphisms of $S^1$, and $\Mob:=\Aut(\bbD)\cong \mathit{PSL}(2,\bbR)$ its subgroup of M\"obius transformations.

\begin{defn}\label{def: definition of conformal net}
A \emph{conformal net} $\cA$ consists of:
\begin{itemize}
\item
A Hilbert space $H_0$ called the \emph{vacuum sector}.
\item
A unit vector $\Omega \in H_0$ called the \emph{vacuum vector}.
\item
A strongly continuous representation $\Mob\to U(H_0)$ with one-dimensional fix-point space spanned by $\Omega$.
The infinitesimal generator of rotations $L_0$ is required to have positive spectrum, and finite dimensional eigenspaces\footnote{\label{foot: finite dim L0 eigenspaces}It is plausible that the finite-dimensionality of the $L_0$-eigenspaces follows from the other axioms of conformal nets (even though, in the absence of $\Diff(S^1)$-covariance, this is known to be false).}.
\item
A strongly continuous projective representation $\Diff(S^1)\to PU(H_0)$ extending the action of $\Mob$.
\item
An assignment
\[
(I \subset S^1)\,\,\,\mapsto\,\,\, (\cA(I)\subset B(H_0))
\]
of a von Neumann algebra to every interval $I\subset S^1$.
This assignment should be $\Diff(S^1)$-equivariant,
and satisfy
\[
(I \subset J)\,\,\,\,\Rightarrow\,\,\,\,\,(\cA(I)\subset \cA(J)),\vspace{-2mm}
\]
and \hspace{3.6cm} $(\mathring I\cap \mathring J=\emptyset)\,\Rightarrow\,([\cA(I), \cA(J)]=0)$.\\[4mm]
Moreover, if $\varphi\in \Diff(S^1)$ fixes $I$ pointwise, then $\varphi$ should also fix $\cA(I)$ pointwise.\newline
Finally, the vacuum vector should be cyclic for the joint actions of the algebras $\cA(I)$.
\end{itemize}
\end{defn}

As a consequence of the above axioms we have the Reeh-Schlieder theorem, which states that $\Omega$ is cyclic and separating for each algebra $\cA(I)$ \cite[Cor. 2.8]{GabbianiFrohlich93} (see also \cite{ReehSchlieder61,Borchers65}).
We also have the celebrated Haag duality theorem \cite[Thm. 2.19(ii)]{GabbianiFrohlich93} (see also \cite{BuchholzSchulz-Mirbach90}, \cite[Thm. 6.2.3]{LongoLectureNotesII}):
\[
\cA(I')=\cA(I)',
\]
where $A':=\{b\in B(H_0):ab=ba\,\,\forall a\in A\}$ denotes the commutant of a von Neumann algebra $A\subset B(H_0)$, and $I':=S^1\setminus\mathring I$ denotes the closed complement of an interval $I\subset S^1$.

We write $H_0^{f.e.}$ for the algebraic direct sum of the $L_0$-eigenspaces of $H_0$, and call it the space of \emph{finite energy vectors}. It is a dense subspace of $H_0$ of countable dimension.

\begin{defn}
A \emph{representation} of a conformal net $\cA$ is a Hilbert space $K$ equipped with actions $\rho_I:\cA(I)\to B(K)$ for every interval $I\subset S^1$, which are 
compatible in the sense that $\rho_J|_{\cA(I)}=\rho_I$ whenever $I \subset J$.

A representation is called \emph{irreducible} if it cannot be written as a direct sum of other representations.
\end{defn}

Let $\Diff_I(S^1)\subset \Diff(S^1)$ be the group of diffeomorphisms supported in $I\subset S^1$.
By the axioms of conformal nets, the operator $U_\varphi\in PU(H_0)$ associated to a diffeomorphism 
$\varphi\in \Diff_I(S^1)$ commutes with $\cA(I')$;
by Haag duality, it is thus in $PU(\cA(I))$.
For every interval $I\subset S^1$, the action $\Diff(S^1)\to PU(H_0)$ therefore restricts to a homomorphism
\[
\Diff_I(S^1)\to PU(\cA(I)).
\]
Given a representation $K$ of the conformal net, we thus get a compatible collection of maps
\begin{equation}\label{eq:Diff_I -> PU(H)}
\Diff_I(S^1)\to PU(K)
\end{equation}
hence a homomorphism
\[
\mathrm{colim}_{I\subset S^1}\Diff_I(S^1)\to PU(K).
\]
The above colimit was computed in \cite[Thm. 11]{Henriques19Colimits},
and shown to be canonically isomorphic to the universal cover $\tDiff(S^1)$ of $\Diff(S^1)$.
So we get a homomorphism
\begin{equation}\label{eq:universal cover Diff -> PU(H)}
\tDiff(S^1)\to PU(K)
\end{equation}
for every representation $K$ of our conformal net.

Let $\tRot(S^1)\subset \tDiff(S^1)$ be the universal cover of the subgroup of rigid rotations, %of $S^1$,
and let $L_0$ be
the infinitesimal generator of the action $\tRot(S^1)\to PU(K)$.
This is a self-adjoint unbounded operator on $K$, well defined up to an additive constant (one may fix the constant by considering the action of the universal cover of $\Mob$, but we shall not need this here\footnote{In \cite{Weiner06}, it is proven that this operator is always positive.
}). 
In this paper, we shall always assume that our conformal net representations have the property that $L_0$ has discrete spectrum.
(A general representation can be written as direct integral of irreducible ones \cite[App.~C]{KaLoMu01}, and every irreducible representation has the property that $L_0$ has discrete spectrum.)
%As in the case of the vacuum sector, we 
\begin{defn}\label{def: H^f.e.}
Given a representation $K$ of a conformal net,
we write $K^{f.e.}$ for the algebraic direct sum of the $L_0$-eigenspaces of $K$, and call it the space of \emph{finite energy vectors} of $K$.
\end{defn}

Given a conformal net $\cA$, there is an associated ``coordinate-free conformal net'' \cite[Prop. 4.3]{BartelsDouglasHenriques15}, which assigns von Neumann algebras to abstract intervals, and Hilbert spaces to abstract discs.

\begin{defn}\label{def: disc and circle}
A \emph{disc} is a complex manifold with boundary which is holomorphically equivalent to $\bbD$.
A \emph{circle} is a smooth oriented $1$-manifold diffeomorphic to $S^1$.
An \emph{interval} is a smooth oriented $1$-manifold with boundary diffeomorphic to $[0,1]$.
\end{defn}

A coordinate-free conformal net is a functor from the category of intervals (no longer subsets of $S^1$) and embeddings to the category of von Neumann algebras (no longer subalgebras of $B(H_0)$) and injective $*$-algebra homomorphisms.
Orientation preserving embeddings go to $\bbC$-linear homomorphisms, and orientation reversing embeddings go to $\bbC$-linear anti-homomorphisms.
Given a conformal net, we shall denote by the same letter $\cA$ the corresponding coordinate-free conformal net.

Furthermore, for every circle $S$, there is the notion of an $S$-sector.
This is a Hilbert space $K$ equipped with compatible actions $\cA(I)\to B(K)$ of all the von Neumann algebras associated to intervals %for 
$I\subset S$
(for $S^1$ the standard circle, this recovers the notion of a representation of $\cA$).
Finally, every diffeomorphism $\varphi:S_1\to S_2$ induces a pullback functor $\varphi^*$ from the category of $S_2$-sectors to the category of $S_1$-sectors.

\begin{defn}\label{def: unitary implements}
Let $K_1$ be an $S_1$-sector, $K_2$ an $S_2$-sector, and $\varphi:S_1\to S_2$ a diffeomorphism.
A unitary $U:K_1\to K_2$ is said to \emph{implement} $\varphi$ if it is a morphism of $S_1$-sectors from $K_1$ to $\varphi^*(K_2)$.
Equivalently, $U$ implements $\varphi$ if it satisfies
\[
U x U^* = \varphi(x)
\]
for every $x\in\cA(I)$ and $I\subset S_1$.
\end{defn}
\begin{rem}
The homomorphism $U:\Diff(S^1)\to PU(H_0)$ in the definition of conformal net sends a diffeomorphism $\varphi\in\Diff(S^1)$ to the unique-up-to-phase
unitary $U(\varphi):H_0\to H_0$ that implements it.
\end{rem}
If $D$ is a disc with boundary $S=\partial D$, then we can associate to it %this setup 
a canonical $S$-sector
\begin{equation}\label{eq:vacuum sector def}
H_0(D) := \Conf(\bbD,D) \times_{\Mob} H_0 
\end{equation}
called the \emph{vacuum sector of $\cA$ associated to $D$}
(see \cite[Thm. 2.13]{BartelsDouglasHenriques15} for a version of this construction).
It admits actions $\cA(I)\to B(H_0(D))$,
for every $I\subset \partial D$,
given by
\begin{equation}\label{eq: action of A(I) on H_0(D)}
x[(\varphi,\xi)] := [(\varphi, {\varphi}^{-1}(x) \xi)],
\end{equation}
where ${\varphi}^{-1}(x)$ denotes the image of $x\in\cA(I)$ under the map $\cA(I)\to \cA(\varphi^{-1}(I))$ induced by $\varphi^{-1}$,
and $[(\varphi,\xi)]\in H_0(D)$ denotes the equivalence class of $(\varphi,\xi)\in \Conf(\bbD,D) \times H_0$.
The vacuum sector $H_0(D)$ also admits a canonical vector
$\Omega_D\in H_0(D)$,
called the \emph{vacuum vector associated to $D$},
defined as the equivalence class of $(\varphi,\Omega)$ for any $\varphi \in \Conf(\bbD,D)$.

Finally, when $D$ is the standard disc, the vacuum sector $H_0(\bbD)$  is canonically identified with $H_0$,
and the vacuum vector $\Omega_\bbD\in H_0(\bbD)$ maps to $\Omega\in H_0$ under that identification.

\section{Vertex operator algebras}\label{sec: VOAs}

\begin{defn}\label{def: definition of N-graded VOA}
An $\mathbb N$-graded \emph{vertex operator algebra} (VOA) consists of a vector space $V = \bigoplus_{n=0}^\infty V(n)$ with $\dim(V(0))=1$, $\dim(V(n))<\infty$, two distinguished vectors $\Omega\in V(0)$ and $\nu\in V(2)$, and a \emph{state-field correspondence}
\[
v\,\,\mapsto\,\, Y(v,z)\in\End(V)[[z,z^{-1}]]
\] 
satisfying:
\begin{itemize}
\item For any vector $v\in V$, we have
$Y(\Omega,z)v=v$ and $Y(v,z)\Omega=v+\mathcal O(z)$.
\item (\emph{locality axiom}) $\forall u,v\in V$, the commutator $[Y(u,z),Y(v,w)]$ is killed by $(z-w)^N$ for $N>\!\!>0$.
\item The operators $L_n$ defined by
$\sum_{n\in \bbZ}L_nz^{-n-2} =Y(\nu,z)$ satisfy the Virasoro algebra relations: $[L_m,L_n]=(m-n)L_{m+n}+\tfrac c{12}(m^3-m)\delta_{m+n,0}$ for some scalar $c$.
\item $L_0$ is the grading operator
\item $Y(L_{-1}v,z)=\tfrac{d}{dz}Y(v,z)$.
\end{itemize}
The number $c$ is called the \emph{central charge} of the VOA.
\end{defn}

\begin{remark}Writing $Y(v,z) = \sum v_{(n)} z^{-n-1}$, the locality axiom is known to be equivalent to the Borcherds identities (see e.g. \cite[\S4.8]{Kac98}):
\begin{align*}
\sum_{j = 0}^\infty \binom{m}{j} \big(u_{(n+j)}v\big)_{(m+k-j)}& = \sum_{j=0}^\infty (-1)^j \binom{n}{j} u_{(m+n-j)} v_{(k+j)}\\
& -\sum_{j=0}^\infty(-1)^{j+n} \binom{n}{j} v_{(n+k-j)} u_{(m+j)}. 
\end{align*}
\end{remark}

Given a vertex algebra $V = \bigoplus_{n=0}^\infty V(n)$, we define $\widehat V:= \prod_{n=0}^\infty V(n)$ to be the product of its weight spaces.
For any vectors $v_1,\ldots,v_n\in V$ and complex numbers $|z_1|>\ldots>|z_n|\ge 0$, the expression
\[
Y(v_1,z_1)\ldots Y(v_n,z_n)\Omega
\]
converges in $\widehat V$.
This is a holomorphic function of the $z_i$ in the domain $|z_1|>\ldots>|z_n|\ge 0$ that extends (uniquely) to a holomorphic function on $\bbC^n\setminus\Delta:=\{(z_i)\in\bbC^n: z_i\neq z_j\}$.
Moreover, for any permutation $\sigma\in S_n$, the functions defined by the above procedure starting from 
$Y(v_1,z_1)\ldots Y(v_n,z_n)\Omega$ 
and from
$Y(v_{\sigma(1)},z_{\sigma(1)})\ldots Y(v_{\sigma(n)},z_{\sigma(n)})\Omega$
agree as $\widehat V$-valued functions on $\bbC^n\setminus\Delta$.

\begin{defn}\label{def: unitary vertex operator algebra}
A \emph{unitary vertex operator algebra} is an $\mathbb N$-graded vertex operator algebra equipped with an inner product $\langle\,,\,\rangle:V\times V\to \bbC$ and an involution $\Theta:V\to V$, called the \emph{PCT involution}.
The involution $\Theta$ is an anti-linear VOA automorphism, and the inner product\footnote{Here, either $z$ is a formal variable, or $z \in \bbC^\times$ is a number.
If $z$ is treated as a number, then the inner product in the LHS is a map $V\times\widehat V\to \bbC$, while the one in the RHS is a map $\widehat V\times V\to\bbC$.} satisfies 
\[
\big\langle v,Y(u,z)v'\big\rangle = \big\langle Y(u,z)^\dagger v,v'\big\rangle
\]
for all $u,v,v'\in V$, where
\begin{align}\label{eq: not: Y(v,z)^dagger}
Y(u,z)^\dagger\;\!:=&\;Y\big(\Theta\;\! e^{zL_1}(-z^{-2})^{L_0}u,\bar z^{-1}\big)\\
                        =&\;Y(e^{\bar zL_1}(-\bar z^{-2})^{L_0}\Theta u,\bar z^{-1})\notag
\end{align}
(see Remark~\ref{rem: mysterious formula for adjoint of field} for an explanation of this formula).
It is customary to also include the normalisation condition $\|\Omega\|=1$ in the definition of unitary vertex operator algebra.
\end{defn}

Let $\Aut(\bbC[[t]])$ be the group of formal power series $a_1z+a_2z^2+\ldots$ with $a_1\neq 0$, under the operation of composition,
and let $\Aut_0(\bbC[[t]])\subset \Aut(\bbC[[t]])$ be the subgroup where the coefficient of $z$ is equal to $1$.
For any vertex operator algebra $V$, the subalgebra $\Vir_{>0}:=\mathrm{Span}\{L_n\}_{n>0}$ of the Virasoro algebra acts locally nilpotently on $V$, and thus exponentiates to an action of $\Aut_0(\bbC[[t]])$.
Furthermore, since $V$ has integer grading, the operator $L_0$ exponentiates to an action of $\bbC^\times$. These assemble to an action\footnote{This is not an action by VOA automorphisms.}
\begin{equation}\label{eq: action of Aut Ct on V}
\Aut_0(\bbC[[t]])\rtimes \bbC^\times = \Aut(\bbC[[t]])\to \End(V).
\end{equation}
When $V$ is a unitary vertex operator algebra, the above action extends to a bigger group that also contains the orientation reversing changes of coordinates:

\begin{lem}\label{lem: action of Aut(C[[t]]) rtimes Z/2 -- VOA}
Let $V$ be a unitary VOA. Then the Virasoro operators $L_n$ for $n\ge 0$ and the anti-linear PCT operator $\Theta:V\to V$ together induce an action of 
\[
%\Aut_{\bbR}(\bbC[[t]]) = 
\Aut(\bbC[[t]])\rtimes \bbZ/2
\]
on $V$ (where the generator of $\bbZ/2$ is the map $t\mapsto \bar t$).
\end{lem}
\begin{proof} As explained above, $\Vir_{\ge 0}$ exponentiates to an action of $\Aut(\bbC[[t]])$ on $V$.
The PCT operator commutes with the Virasoro operators $L_n$ because it is an automorphism of the VOA.
It satisfies $\Theta (\lambda L_n) \Theta = \bar \lambda L_n$ for $\lambda\in\bbC$, and therefore $\Theta g \Theta = \bar g$ for any $g\in \Aut(\bbC[[t]])$.
The latter is equivalent an action of $\Aut(\bbC[[t]])\rtimes \bbZ/2$ on $V$.
\end{proof}

\begin{remark}\label{rem: mysterious formula for adjoint of field}
The mysterious $\Theta e^{zL_1}(-z^{-2})^{L_0}u$ which appears in \eqref{eq: not: Y(v,z)^dagger} is just the action on $u$
of the coordinate transformation $t \mapsto \overline{(t+z)}^{-1}-\bar z^{-1}$
(the germ of $\zeta\mapsto \bar \zeta^{-1}$ around $\zeta = z$). %\footnote{For $z\in\bbC^\times$, this element of $V$ can also be written as $e^{\bar zL_1}(-\bar z^{-2})^{L_0}\Theta u$.}
To see this, note that 
\[
\big(t \mapsto \overline{(t+z)}^{-1}-\bar z^{-1}\big) = h\circ f \circ g %\circ h,
\]
where $h(t)=\bar t$, $f(t) = (t^{-1}-z)^{-1}$, and $g(t) = -z^{-2} t$.
The Mobius transformations $f$ and $g$ correspond to the operators $e^{z L_1}$ and $(-z^{-2})^{L_0}$, respectively.

%VERIFICATION:
%\[
%\overline{((-z^{-2} t)^{-1}-z)^{-1}} = \overline{(-z^2 t^{-1}-z)^{-1}}
%= \overline{\frac{-1}{z^2 t^{-1}+z}} = \overline{\frac{-z^{-1}t}{z+t}}
%= \overline{\frac{1}{z+t}-\frac{1+z^{-1}t}{z+t}} = \overline{(z+t)^{-1}-z^{-1}}
%\]

\end{remark}

\subsection*{Modules for vertex operator algebras}

Let $V$ be a vertex operator algebra, and let $M=\bigoplus_{n}M(n)$ be a vector space  graded by (some subset of) $\bbR_{\ge 0}$, or, more generally, graded by some subset of $\bbC$ whose real part is bounded below.
\begin{defn}\label{def: VOA module}
We say that $M$ is a \emph{module} for $V$ (also called a \emph{strong module}) if it satisfies $\dim(M(n))<\infty$, and if it comes equipped with a state-field correspondence
\[
v\,\,\mapsto\,\, Y^M(v,z)\in\End(M)[[z,z^{-1}]]
\]
which is degree preserving when the formal variable $z$ is given degree $-1$, and satisfies: 
\begin{itemize}
\item $Y^M(\Omega,z) = \id_M$
\item For any $u,v\in V$, the commutator $[Y^M(u,z),Y^M(v,w)]$ is killed by $(z-w)^N$ for $N>\!\!>0$.
\item For any $u,v\in V$, the difference between $Y^M(Y(u,z)v,w)$ and $Y^M(u,z+w)Y^M(v,w)$ is killed by $(z+w)^N$ for $N>\!\!>0$.
\end{itemize}
\end{defn}
Writing $Y^M(v,z) = \sum v^M_{(n)} z^{-n-1}$,
the last two axioms are known to be equivalent to the Borcherds identities (see \cite[\S4.2-4.4]{LepowskyLi04}, in which the Borcherds identities are written in the notation of formal delta functions):
\begin{align*}
\sum_{j = 0}^\infty \binom{m}{j} \big(u_{(n+j)}v\big)^M_{(m+k-j)}& = \sum_{j=0}^\infty (-1)^j \binom{n}{j} u_{(m+n-j)}^M v_{(k+j)}^M\\
& -\sum_{j=0}^\infty(-1)^{j+n} \binom{n}{j} v_{(n+k-j)}^M u_{(m+j)}^M. 
\end{align*}

A \emph{weak module} is a vector space $M$ as above, but without any grading, equipped with a state-field correspondence subject to the lower truncation condition 
$Y^M(v,z)a\in M[[z]][z^{-1}]$ $\forall a\in M$, and subject to all the axioms above, except for the one requiring that the state-field correspondence is degree preserving.

\begin{defn}
A unitary (weak) module for a unitary VOA $V$ is a (weak) $V$-module $M$ equipped with an inner product $\langle\,\,,\,\rangle_M:M\times M\to\bbC$ that satisfies
\[
\big\langle a,Y^M(u,z)b\big\rangle_M = \big\langle Y^M(u,z)^\dagger a, b\big\rangle_M.
\]
for all $u\in V$ and $a,b\in M$ (where $Y^M(u,z)^\dagger$ is defined as in \eqref{eq: not: Y(v,z)^dagger}).
\end{defn}

\section{The semigroup of annuli and its representations}\label{sec: semigroup of annuli}

\subsection{The semigroup of annuli}

The Lie algebra of $\Diff(S^1)$ is the set $\cX:=\cX(S^1)$ of vector fields on $S^1$, equipped with the Lie bracket
\begin{equation}\label{eq: ie bracket of vector fields - OPP}
\big[f(\theta)\tfrac{\partial}{\partial \theta},g(\theta)\tfrac{\partial}{\partial \theta}\big]:=
(gf'-fg')\tfrac{\partial}{\partial \theta}
\end{equation}
(the opposite of the usual Lie bracket of vector fields).
Its complexification $\cX_\bbC:=\cX\otimes_\bbR\bbC$ is known to have no associated Lie group.
But it has an associated semigroup, the \emph{semigroup of annuli} $\Ann$ \cite{SegalDef,Neretin90,HenriquesTener24ax}, which we now describe.

An \emph{embedded annulus} is a subset of $\bbC$ of the form
\begin{equation*} %\label{eq: Dout min Din}
A=D_{out}\setminus \mathring D_{in}\subset \bbC,
\end{equation*}
where $D_{in}\subset D_{out}\subset \bbC$ are discs (closed subsets bounded by smooth Jordan curves).
Its inner and outer boundaries $\partial_{in}A:=\partial D_{in}$ and $\partial_{out}A:=\partial D_{out}$ are equipped with the orientations inherited from $D_{in/out}$.
Note that $\partial_{in}A\cap\partial_{out}A$ is allowed to be non-empty, and that $\mathring A$ is allowed to be disconnected, or even empty.
We equip $A$ with the sheaf $\cO_A$ of functions that are continuous on $A$, holomorphic on $\mathring A$, and smooth on $\partial_{in}A$ and on $\partial_{out}A$.

\begin{defn}\label{def: annulus}
An \emph{annulus} is a locally ringed space $A=(A,\cO_A)$ which is isomorphic to an embedded annulus.
We also record the orientations of $\partial_{in}A$ and $\partial_{out}A$ as part of the data.\footnote{When $\mathring A\not = \emptyset$, %is not empty, 
the orientations of $\partial_{in}A$ and $\partial_{out}A$ can be deduced from the orientation of $\mathring A$.}
\end{defn}

An annulus $A$ is called \emph{thick} if $\partial_{in}A\cap\partial_{out}A=\emptyset$ (in which case it is diffeomorphic to $S^1\times [0,1]$).
It is called \emph{thin} if $\partial_{in}A\cap\partial_{out}A$ contains an interval, and
\emph{completely thin} if $\partial_{in}A=\partial_{out}A$ (in which case it is diffeomorphic to $S^1$).

\begin{lem}\label{lem: A iso to A with d_out = S^1}
Every annulus is isomorphic to an embedded annulus $A\subset \bbC$ that satisfies $\partial_{out}A=S^1$ and $0\not\in A$. %the standard circle.
\end{lem}

\begin{proof}
By definition, every annulus is isomorphic to one of the form $D_{out}\setminus \mathring D_{in}\subset \bbC$.
Now apply the Riemann mapping theorem for domains with smooth boundary \cite[Thm. 8.2]{Bell92} 
to $D_{out}$.
\end{proof}

An \emph{annulus with parametrised boundary} is an annulus equipped with an orientation reversing diffeomorphisms $\varphi_{in}:S^1\to \partial_{in}A$, and an orientation preserving diffeomorphism $\varphi_{out}:S^1\to \partial_{out}A$.

\begin{defn}
The \emph{semigroup of annuli} $\Ann$ is the set of isomorphisms classes of annuli with parametrised boundary.
The semigroup operation is \emph{conformal welding}, which
is the gluing of annuli along their boundaries:
\[
\begin{split}
\cup:\Ann &\times \Ann \,\longrightarrow\, \Ann\\
A_1 &\cup A_2\, :=\, \mathrm{pushout}\big(A_1 \xleftarrow{\,\varphi_{in}^{A_1}\,} S_1 \xrightarrow{\,\varphi_{out}^{A_2}\,} A_2\big)\\
&\phantom{\cup A_2\,\, :}=\, A_1 \sqcup A_2 \big/ \varphi_{in}^{A_1}(x)\sim \varphi_{out}^{A_2}(x).
\end{split}
\]
The annulus $A_1 \cup A_2$ is equipped with the sheaf $\cO_{A_{1} \cup A_{2}}$ specified by
$f\in \cO_{A_1 \cup A_2}(U) \Leftrightarrow f|_{U\cap A_i}\in \cO_{A_i}(U\cap A_i)$ for $i=1,2$.
We refer the reader to \cite[Prop. 3.5]{HenriquesTener24ax} for a proof that $(A_1 \cup A_2,\cO_{A_1\cup A_2})$ is again an element of $\Ann$.
\end{defn}

The semigroup of annuli is equipped with an involution $\dagger: \Ann \to \Ann$
which sends an annulus $A$ to the same annulus equipped with the opposite complex structure: $f\in \cO_{A^\dagger}(U)\Leftrightarrow \bar f\in \cO_{A}(U)$. The incoming and outgoing circles are exchanged, but their parametrisations remain the same.
This involution satisfies
\begin{equation}\label{eq: dag of product}
(A\cup B)^\dagger=B^\dagger\cup A^\dagger.
\end{equation}

\begin{ex}\label{ex: r^ell_0}
For $r\in(0,1]$, the round annulus $A_r:=\{z\in\bbC:r\le|z|\le1\}$, with boundary parametrizations $\varphi_{in}(z)=rz$ and $\varphi_{out}(z)=z$ is denoted $r^{\ell_0}\in\Ann$.
\end{ex}

There is an obvious embedding
\[
\Diff(S^1)\hookrightarrow \Ann
\]
which sends a diffeomorphism $\varphi$ to the completely thin annulus $A_\varphi:=(S^1,\varphi_{in}{=}\varphi,\varphi_{out}{=}\id)$.
The semigroup of annuli is, in some sense, a complexification of this subgroup. 

\begin{lem}\label{lem: A_varphi  A B=r^ell_0}
For every annulus $A\in\Ann$ there exist $0<r\le 1$, 
$B\in\Ann$, and $\varphi\in \Diff(S^1)$ such that $A_\varphi\cup A\cup B=r^{\ell_0}$. 
\end{lem}
\begin{proof}
By Lemma~\ref{lem: A iso to A with d_out = S^1}, we may assume that $A=\bbD\setminus \mathring D_{in}\subset \bbC$, with $0\not\in A$.
Let $\varphi^A_{in}:S^1\to \partial D_{in}$ and $\varphi^A_{out}\in \Diff(S^1)$  be the boundary parametrizations of $A$.
Pick $r$ such that $D_r:=\{z\in\bbC:|z|\le r\}\subset D_{in}$, and let $B:=D_{in}\setminus \mathring D_r$ with boundary parametrizations 
$\varphi^B_{in}(z)=rz$ and $\varphi^B_{out}=\varphi^A_{in}$.
Then setting $\varphi=(\varphi^A_{out})^{-1}$, we have $A_\varphi\cup A\cup B=r^{\ell_0}$.
\end{proof}

\begin{defn}\label{def: localized annuli}
An annulus with parametrised boundary is said to be \emph{localized in some interval} $I\subset S^1$ if its incoming and outgoing boundary parametrizations agree on the closed complement $I':=S^1\setminus \mathring I$ of $I$: 
\[
\varphi_{in}|_{I'} = \varphi_{out}|_{I'}: I'\to A.
\]
(Such annuli are never thick.)
We denote by $\Ann(I)\subset \Ann$ the sub-semigroup of annuli localized in $I$.
\end{defn}

The semigroup of annuli contains another important sub-semigroup $\Univ\subset \Ann$, which behaves in many respects like the Borel subgroup of an algebraic group:

\begin{defn}
The \emph{semigroup of univalent maps} $\Univ$ is the set of holomorphic embeddings $f:\bbD\to\bbD$ (with $f'$ everywhere non-zero, including on $\partial\bbD$).
The inclusion $\Univ\hookrightarrow \Ann$ sends a univalent map $f\in \Univ$ to the annulus
\[
A_f := \bbD \setminus f(\mathring \bbD),
\]
with $\varphi_{in}=f$ and $\varphi_{out}=\id$.

Within $\Univ$, we have a further sub-semigroup $\Univ_0$, consisting of the univalent maps which fix the point $0 \in \bbD$.
\end{defn}

\begin{defn}
A \emph{disc with parametrised boundary} is a disc $D$ (Definition~\ref{def: disc and circle}) equipped with an orientation preserving diffeomorphism $S^1\to\partial D$. 

The ($C^\infty$ version of the) \emph{universal Teichm\"uller space} $\cT$ is the set of isomorphism classes of disc with parametrised boundary.
\end{defn}

The semigroup of annuli acts on the universal Teichm\"uller space by conformal welding, and the map
\begin{equation}\label{eq: action of Ann on T}
\Ann\to\cT:A\mapsto A\cup\bbD
\end{equation}
identifies $\cT$ with the quotient $\Ann/\Univ$ (the quotient of $\Ann$ by the equivalence relation generated by $A\cup B\sim A$, for all $B\in \Univ$).
The universal Teichm\"uller space is thus an analog of the flag manifold $G/B$ of an algebraic group $G$.

\subsection{Central extensions of the semigroup of annuli}
\label{sec: Central extensions of the semigroup of annuli}

For $c\in\bbC$, let
\begin{equation}\label{eq: Vir SES}
0 \to \bbC\to \widetilde{\cX}_c \to \cX_\bbC \to 0
\end{equation}
be the Lie algebra central extension determined by the Virasoro cocycle
\begin{equation}\label{eq: Vir cocycle}
\omega_{\Vir}(f,g)=\omega_{\Vir}\big(f(\theta)\tfrac{\partial}{\partial \theta},g(\theta)\tfrac{\partial}{\partial \theta}\big):=\tfrac c{12}\int_{S^1}(f'''(\theta)+f'(\theta))\,g(\theta)\,\tfrac{d\theta}{2\pi i}.
\end{equation}
%\begin{align*}    KEEP THIS VERIFICATION
%\omega_{\Vir}(\ell_m,\ell_n)&=
%\omega_{\Vir}(-ie^{im\theta},-ie^{in\theta})\\
%&=-\omega_{\Vir}(e^{im\theta},e^{in\theta})\\
%&=-\tfrac c{12}\int_{S^1}((e^{im\theta})'''+(e^{im\theta})')\,e^{in\theta}\,\tfrac{d\theta}{2\pi i}\\
%&=-\tfrac c{12}\int_{S^1}((im)^3+(im))e^{im\theta}\,e^{in\theta}\,\tfrac{d\theta}{2\pi i}\\
%&=i\tfrac c{12}\int_{S^1}(m^3-m)e^{i(n+m)\theta}\,\tfrac{d\theta}{2\pi i}\\
%&=\tfrac c{12}(m^3-m)\int_{S^1}e^{i(n+m)\theta}\,\tfrac{d\theta}{2\pi}
%\end{align*}
Let $\ell_n:=-ie^{in\theta}\tfrac{\partial}{\partial \theta}\in \cX_\bbC$,\footnote{In terms of the coordinate $z=e^{i\theta}$ on $S^1$, this is $\ell_n=z^{n+1}\tfrac{\partial}{\partial z}$,
and the Virasoro cocycle is given by $\omega_{\Vir}\big(f(z)\tfrac{\partial}{\partial z},g(z)\tfrac{\partial}{\partial z}\big)=\tfrac c{12}\int_{S^1}\tfrac{\partial^3 f}{\partial z^3}(z)\,g(z)\,\tfrac{dz}{2\pi i}$.}
and let $L_n\in \widetilde{\cX}_c$ be its standard lift so that, using \eqref{eq: ie bracket of vector fields - OPP}, we have the familiar
$\big[L_m,L_n\big]\,=\,(m-n)L_{m+n}+\tfrac c{12}(m^3-m)\delta_{m+n,0}$.
In particular, the Virasoro Lie algebra $\Vir_c:=\mathrm{Span}\{L_n\}_{n\in\bbZ}\oplus\bbC$ is a dense subalgebra in $\widetilde{\cX}_c$.

Associated to the central extension \eqref{eq: Vir SES}, there exists a central extension of semigroups
\begin{equation}\label{eq: tAnnc SES}
0 \to \bbC^\times{\times}\,\bbZ\to \tAnn_c \to \Ann \to 0
\end{equation}
that we describe shortly.
We first list some related central extensions that follow from the one above.
First, if we mod out $\tAnn_c$ by the central $\bbC^\times$, we get the universal cover of $\Ann$:
\[
0 \to \bbZ\to \tAnn \to \Ann \to 0.
\]
Let $\tMob$ be the universal cover of the M\"obius group.
Since $\tMob$ is also the universal central extension of $\Mob$,
the map $\tMob\to\Mob\to \Diff(S^1)\to\Ann$ lifts uniquely to a map $\tMob\to \tAnn_c$,
yielding a canonical copy of $\bbZ$ in the center of $\tAnn_c$.
(I.e., the center of $\tAnn_c$ is \emph{canonically} isomorphic to $\bbC^\times{\times}\bbZ$.)
\begin{defn}
We let $\Ann_c$ be the quotient of $\tAnn_c$ by the central subgroup $\bbZ$ described above.
This is a central extension of the semigroup of annuli by $\bbC^\times$:
\begin{equation}\label{eq: Annc SES}
0 \to \bbC^\times\to \Ann_c \to \Ann \to 0.
\end{equation}
If $I \subset S^1$ is an interval, we also let $\Ann_c(I)$ be the pullback of the central extension $\Ann_c$ to $\Ann(I)$:
\[
0 \to \bbC^\times\to \Ann_c(I) \to \Ann(I) \to 0.
\]
We call the elements of $\Ann_c$ \emph{lifted annuli}, and the elements of $\tAnn_c$ \emph{$z$-lifted annuli}.
\end{defn}

Our construction of $\tAnn_c$ relies of the notion of \emph{framing} of an annulus:

\begin{defn}\label{def: Framings and paths}
Let $A$ be an annulus with parametrised boundary.
A \emph{framing} of $A$ is a smooth surjective map
\[
h:S^1\times[0,1]\to A
\]
whose restriction to each $S^1\times\{t\}$ is an embedding, 
satisfies $h|_{S^1\times\{0\}}=\varphi_{in}$ and $h|_{S^1\times\{1\}}=\varphi_{out}$,
and for which $-\frac{\partial h/\partial t}{\partial h/\partial \theta}$ has everywhere non-negative imaginary part.

The \emph{path} associated to a framing $h$ is the map $X:[0,1]\to \cX_\bbC$ given by:
\[
X(t):=\Big({-}\tfrac{\partial h/\partial t}{\partial h/\partial \theta}\Big) \tfrac{\partial}{\partial \theta}.
\]
\end{defn}

By \cite[Cor. 4.17]{HenriquesTener24ax} every annulus with parametrised boundary admits a framing.
Given a framing $h$ with associated path $\{X(t)\}_{t\in[0,1]}$ then, following \cite[\S4]{HenriquesTener24ax}, we can recover $A$ from the path by a certain geometric procedure, which we think of,
and formally denote by a time-ordered exponential:
\[
A = \prod_{1\ge \tau\ge 0} \Exp\big(X(\tau)d\tau\big).
\]

Given an $\cX_\bbC$-valued path $X$ as above,
let $\{\widetilde X(t)\}_{t\in[0,1]}$ be its image under the standard splitting $\cX_\bbC \to \cX_\bbC\oplus\bbC=\widetilde{\cX}_c$ of \eqref{eq: Vir SES}.
In \cite[\S5]{HenriquesTener24ax}, we constructed $\tAnn_c$ as a certain set of equivalence classes of formal expressions
\begin{equation}\label{eq: under A (formal)}
\underline A = z\,{\cdot} \prod_{1\ge \tau\ge 0} \Exp\big(\widetilde X(\tau)d\tau\big),
\end{equation}
where $\widetilde X$ is as above, and $z\in\bbC^\times$.

We now describe the equivalence relation on the set of formal expressions \eqref{eq: under A (formal)}.
If $A\in\Ann$ is an annulus with parametrised boundary, and if $\{\widetilde X_0(t)\}_{t\in[0,1]}$ and $\{\widetilde X_1(t)\}_{t\in[0,1]}$
come from two framings $h_0$ and $h_1$ of $A$, then we declare
\begin{equation}\label{eq: under A (formal)  tild  }
z_0\,{\cdot}\! \prod_{1\ge \tau\ge 0} \Exp\big(\widetilde X_0(\tau)d\tau\big)
\,\,\,\sim\,\,\,
z_1\,{\cdot}\! \prod_{1\ge \tau\ge 0} \Exp\big(\widetilde X_1(\tau)d\tau\big)
\end{equation}
in $\tAnn_c$ if there exists an (equivalently for every) isotopy $h(\theta,t,s)$ between $h_0(\theta,t)$ and $h_1(\theta,t)$ such that
\begin{equation}\label{eq: zz = int}
z_1\;\!z_0^{-1}
\,=\,\exp
\int \omega_{\Vir}\Big(\tfrac{\partial h/\partial t}{\partial h/\partial \theta},\tfrac{\partial h/\partial s}{\partial h/\partial \theta}\Big) \,dtds.
\end{equation}

Let $\ast:\widetilde{\cX}_c\to \widetilde{\cX}_c$ be the continuous conjugate-linear extension of the map $L_n\mapsto L_{-n}$. 
The dagger structure on $\Ann$ lifts to a dagger structure on $\tAnn_c$ given by
\[
\Big(z\,{\cdot}\! \prod_{1\ge \tau\ge 0} \Exp\big(\widetilde X(\tau)d\tau\big)\Big)^\dagger
\,:=\, \bar z\,{\cdot}\! \prod_{1\ge \tau\ge 0} \Exp\big(\widetilde X(1-\tau)^*d\tau\big)
\]
and satisfying
\begin{equation}\label{eq: dag of product $z$-lifted}
(\underline A\cup \underline B)^\dagger=\underline B^\dagger\cup \underline A^\dagger.
\end{equation}

\begin{ex}\label{ex: r^L_0}
The annulus $r^{\ell_0}\in\Ann$ from Example~\ref{ex: r^ell_0} admits a canonical $z$-lift $r^{L_0}\in\tAnn_c$ 
coming from the framing $h(\theta,t)=e^{i\theta} r^{1-t}$ of $A_r$.
In that case, $X(t)=i|\log(r)|$ is independent of $t$, and corresponds to the vector field $\log(r)\ell_0\in\cX_\bbC$.
This lifts to $\log(r)L_0\in\Vir_c\subset \widetilde\cX_c$, and our notation $r^{L_0}$ is justified:
\[
\prod_{1\ge \tau\ge 0} \Exp\big(\widetilde X(\tau)d\tau\big) = \Exp\big(\log(r)L_0\big) = r^{L_0}.
\]
\end{ex}

When $c\in\bbR$ and $f$, $g$ in \eqref{eq: Vir cocycle} are real-valued,
%$f(\theta)\tfrac{\partial}{\partial \theta},g(\theta)\tfrac{\partial}{\partial \theta}\in\cX$ (as opposed to $\cX_\bbC$), 
the Virasoro cocycle takes its values in $i\bbR$.
The same construction (\ref{eq: under A (formal)},\,\ref{eq: under A (formal)  tild  }) as above
restricted to the case of completely thin annuli, and with $z\in U(1)$ as opposed to $\bbC^\times$,
produces a central extension
\begin{equation}\label{eq: tAnnc SES - Diff}
0 \to U(1){\times}\,\bbZ\to \tDiff_c(S^1) \to \Diff(S^1) \to 0.
\end{equation}
Modding out by the canonical central $\bbZ$, we then also get an analog of \eqref{eq: tAnnc SES}
\begin{equation}\label{eq: Annc SES - Diff}
0 \to U(1)\to \Diff_c(S^1) \to \Diff(S^1) \to 0.
\end{equation}

The central extensions \eqref{eq: tAnnc SES - Diff}, \eqref{eq: Annc SES - Diff}, \eqref{eq: tAnnc SES}, \eqref{eq: Annc SES} fit into a commutative diagram
\[
\begin{tikzcd}
0 \arrow[r] &\bbZ \arrow[r] \arrow[d, equal] &\tDiff_c(S^1) \arrow[r] \arrow[d] & \Diff_c(S^1) \arrow[d]\arrow[r] & 0\\
0 \arrow[r] &              \bbZ \arrow[r] &               \tAnn_c \arrow[r]                    & \Ann_c\arrow[r] & 0
\end{tikzcd}
\]
where each vertical arrow is the inclusion of a real Lie group in its ``complexification''.
For $\varphi\in\Diff_c(S^1)$ or $\tDiff_c(S^1)$, we shall write $\underline A{}_\varphi$ for the corresponding element of $\Ann_c$ or $\tAnn_c$.

The above semigroups are further equipped with anti-linear involutions 
%(actions of $\bbZ/2\bbZ$)
that we now describe.
To start, the continuous $\bbR$-linear extensions of the maps 
$\lambda\ell_n\mapsto\bar\lambda\ell_n$
and
$\lambda L_n\mapsto\bar\lambda L_n$ (where $\lambda\in\bbC$) yield actions of $\bbZ/2\bbZ$ on 
$\cX_\bbC$ and 
$\widetilde{\cX}_c$ by Lie algebras automorphisms. Those integrate to actions
\[
\bbZ/2\bbZ\,\acts \tAnn
\qquad\text{and}
\qquad
\bbZ/2\bbZ\,\acts \tAnn_c
\]
that preserve $\tMob$, and hence the central $\bbZ\subset \tMob$ (on which they act by $-1$).
They therefore descend to actions
\begin{equation}\label{eq: the invol}
\bbZ/2\bbZ\,\acts \Ann
\qquad\text{and}
\qquad
\bbZ/2\bbZ\,\acts \Ann_c.
\end{equation}
Note that the involution \eqref{eq: the invol} on $\Ann$ admits the following more elementary description,
as the map $\Ann\to \Ann:A\mapsto \overline A$ which sends an annulus $A\in \Ann$, with boundary parametrizations $\varphi_{in/out}$, to that same annulus with opposite complex structure, and boundary parametrizations $\varphi_{in/out}\circ (z\mapsto\bar z)$.

\begin{defn}\label{def: 4.12}
If an annulus $A\in\Ann$ is localised in $I\subset S^1$, then a $z$-lift $\underline{A}\in\tAnn_c$ is said to be \emph{localised in $I$} if it can be written as
\[
\underline A = z\,{\cdot}\! \prod_{1\ge \tau\ge 0} \Exp(\widetilde X(\tau)d\tau),
\]
where each $\widetilde X(\tau)$ covers a vector field $X(\tau)\in\cX_\bbC$ supported in $I$. 
We write $\tAnn_c(I)\subset\tAnn_c$ for the sub-semigroup of $z$-lifted annuli that are localised in $I$.
\end{defn}

The projection $\tAnn_c\to\Ann_c$ maps $\tAnn_c(I)$ isomorphically onto $\Ann_c(I)$.
We may thus drop the distinction between $\tAnn_c(I)$ and $\Ann_c(I)$, and treat 
$\Ann_c(I)$ as a sub-semigroup of both $\Ann_c$ and $\tAnn_c$.

Let $\tUniv$ denote the universal cover of $\Univ$, and let us recall that
$\cT=\Ann/\Univ=\tAnn/\tUniv$.
By \cite[Prop. 4.22]{HenriquesTener24ax}, every element $A\in \tUniv$ can be written as
\begin{equation}\label{eq: A univalent}
A = \prod_{1\ge \tau\ge 0} \Exp(X(\tau)d\tau),
\end{equation}
with each $X(\tau)$ in the \emph{univalent subalgebra}
$\overline{\mathrm{Span}}\{\ell_n\}_{n\ge -1}$ of $\cX_\bbC$.
(In terms of the coordinate $z=e^{i\theta}$, the univalent subalgebra consists of  vector fields $f(z)\tfrac{\partial}{\partial z}$ with $f$ extending holomorphically to the unit disc.)
Let us define a \emph{univalent framing} to be a framing $h$ such that the corresponding path $X$ takes values in the univalent subalgebra. 
Any pair of univalent framings of $A\in \tUniv$ are homotopic through univalent framings: they are homotopic through framings by \cite[Lem. 4.18]{HenriquesTener24ax}, and these framings may be modified to be univalent as in the proof of \cite[Prop. 4.22]{HenriquesTener24ax}.
A choice of univalent framing induces a lift 
of $A$ to $\tAnn_c$.
Moreover, since the Virasoro cocycle vanishes on the univalent subalgebra, this lift is independent of the choice of univalent framing (because the integral in the right hand side of \eqref{eq: zz = int} vanishes).
It follows that the central extension $\tAnn_c\to \tAnn$ uniquely trivialises over $\tUniv$, and we may thus view $\tUniv$ as a sub-semigroup of $\tAnn_c$.
Modding out by the central $\bbZ$, we also get an embedding $\Univ \hookrightarrow \Ann_c$.
\begin{defn}
The quotient space
\[
\cT_c\,:=\,\Ann_c/\Univ\,=\,\tAnn_c/\tUniv.
\]
is a $\bbC^\times$-torsor over the universal Teichm\"uller space $\cT$.
Its elements are called \emph{lifted discs}.
\end{defn}
Since the action of $\Ann_c$ on $\cT_c$ covers the action of $\Ann$ on $\cT$ by conformal welding, we shall denote it by the same symbol:
given $\underline A\in\Ann_c$ and $\underline{D}\in \cT_c$, we write $\underline A\cup \underline{D}\in \cT_c$ for the action of $\underline A$ on $\underline{D}$.
We shall also abuse notation and write $\bbD\in\cT_c$ for the standard lift of the the standard disc $\bbD\in\cT$ (defined as the image of $1\in\Ann_c$ under the quotient map $\Ann_c\to\cT_c$).

Given a lifted disc $\underline{D}\in \cT_c$ with underlying disc $D$, and given an embedding $f:\bbD\to D$, there is a unique lift $\underline A \in \Ann_c$ of %the annulus 
$A:=D\setminus f(\bbD)$ that satisfies $\underline A \cup \bbD = \underline{D}$. 
We write $\underline{D}\setminus f(\bbD)$ for this unique element of $\Ann_c$ so that,
by definition, we have
\[
(\underline{D}\setminus f(\bbD))\cup \bbD \,=\, \underline{D}.
\]

\subsection{Representations of the semigroup of annuli}\label{sec: reps of semigroup of annuli}

The conjugate-linear involution $*:\Vir_c\to \Vir_c$ given on generators by $L_n\mapsto L_{-n}$ equips the Virasoro algebra with the structure of a $*$-Lie algebra.

\begin{defn}
A \emph{unitary representation} of the Virasoro algebra on a Hilbert space $H$
is map $\pi$ from $\Vir_c$ to the set of closed unbounded operators on $H$ satisfying:
\begin{itemize}
\item $\pi(1)=1$
\item $\pi(X^*)=\pi(X)^*$
\item $\pi$ restricts to a Lie algebra homomorphism $\Vir_c\to\End(H^{f.e.})$,
where $H^{f.e.}$ denotes the algebraic direct sum of the eigenspaces of $\pi(L_0)$.\footnote{If we wished to also accommodate representations in which $\pi(L_0)$ has continuous spectrum, then $H^{f.e.}$ should be instead defined as the algebraic span of the ranges of the finite spectral projections of $\pi(L_0)$.}
\end{itemize}
A unitary representation $(H,\pi)$ is said to have \emph{positive energy} if $\pi(L_0)$ is a positive operator with discrete spectrum and finite dimensional eigenspaces.
\end{defn}

We have shown in \cite{HenriquesTenerIntegratingax}
that every positive energy unitary representation $\pi$ of $\Vir_c$ integrates to a holomorphic representation of $\tAnn_c$, denoted again $\pi$.
Given a $z$-lifted annulus $\underline A\in \tAnn_c$, represented %as in \eqref{eq: under A (formal)} 
as $\underline A = z\,{\cdot} \prod_{1\ge \tau\ge 0} \Exp\big(\widetilde X(\tau)d\tau\big)$, we define in loc.~cit.
\begin{equation}\label{eq: def how annuli act}
\pi\Big(z\,{\cdot}\! \prod_{1\ge \tau \ge 0} \Exp(\widetilde X(\tau)d\tau)\Big)
\,:=\,z\,{\cdot}\! \prod_{1\ge \tau \ge 0} \Exp\big(\pi(\widetilde X(\tau))d\tau\big).
\end{equation}
Here, the time-ordered exponential in the right-hand side %of \eqref{eq: def how annuli act} 
is, by definition, the value at $t=1$ of the solution of the evolution equation
$\tfrac d{dt} Y_t = \pi(\widetilde X(t))Y_t$ (see \cite[\S4]{HenriquesTenerIntegratingax} for the proof of existence of uniqueness of solutions to that evolution equation).
The representation \eqref{eq: def how annuli act} is compatible with the involutions in the sense that
\begin{equation}\label{eq: pi(Adag)=pi(A)*}
\pi(\underline A^\dagger)=\pi(\underline A)^*.
\end{equation}

When $H=H_0$, this action of $\tAnn_c$ on $H_0$ descends to an action of $\Ann_c$.
Also, as an immediate consequence of the formula \eqref{eq: A univalent} for univalent annuli, along with the fact that the operators $L_n$ for $n\ge -1$
kill the vacuum vector $\Omega\in H_0$, we have
\begin{equation}\label{eq:  A Omega=Omega}
\pi(\underline A)\Omega=\Omega
\end{equation}
for every lifted annulus $\underline A$ in the image of the embedding $\Univ \hookrightarrow \Ann_c$.

\begin{lem}\label{lem: annuli have dense image}
The operator $\pi(\underline A):H\to H$ associated to a $z$-lifted annulus $\underline A\in\tAnn_c$ is injective with dense image.
\end{lem}

\begin{proof}
An operator is injective if and only if its adjoint has dense image.
So, by \eqref{eq: pi(Adag)=pi(A)*}, it is enough to just check injectivity.
Applying Lemma~\ref{lem: A_varphi  A B=r^ell_0} to the conjugate annulus $\underline A^\dagger$,
we may write $r^{L_0}=\underline B\cup \underline A\cup \underline A{}_\varphi$
for some $\underline B\in\tAnn_c$ and $\varphi\in \tDiff_c(S^1)$.
It follows that
\[
\pi(r^{L_0})=r^{\pi(L_0)}=\pi(\underline B)\pi(\underline A)\pi(\underline A{}_\varphi).
\]
The left hand side is the exponential of a self-adjoint operator, hence injective.
The operator $\pi(\underline A{}_\varphi)$ is invertible.
It follows that $\pi(\underline A)$ is injective.
\end{proof}

Let $\cA$ be a conformal net, let $H_0$ be its vacuum sector, and let $U_0:\Diff(S^1)\to PU(H_0)$ be the projective representation of $\Diff(S^1)$ on $H_0$.
Recall that all our conformal net representations are assumed to have discrete $L_0$ spectrum.

\begin{prop}\label{prop: exist: stress energy tensor}
Let $(H,\rho)$ be a representation of $\cA$. 
Then there exists a unique continuous linear map
%\begin{equation}\label{eq: f mapsto T_H(f)}
$f \mapsto T_H(f)$
%\end{equation}
from $C^\infty(S^1)$ to the space of closed unbounded operators on $H$ with core $H^{f.e.}$\,(with the topology of pointwise convergence on $H^{f.e.}$\!) such that:
\begin{itemize}
\item
The operators $L_n:=T_H(e^{in\theta})$ satisfy the Virasoro algebra relations on $H^{f.e.}$.
\item
$\rho_I(U_0(\exp(f \tfrac{d}{d\theta}))) = \exp(iT_H(f))$ up to phase, for any real-valued function $f$ with support in $I$.
\end{itemize}
These operators satisfy $T_H(f)^*=T_H(\bar f)$, 
and $L_0:=T_H(1)$ agrees (up to constant) with the operator used in Definition~\ref{def: H^f.e.}.
Moreover, for any function $f$ with support in $I$, the operator $T_{H_0}(f)$ is affiliated with $\cA(I)$, and
\begin{equation}\label{eq: rho(T_H0(f)) = T_H(f)}
\rho_I(T_{H_0}(f)) = T_H(f).
\end{equation}
\end{prop}

The operator-valued distribution $f \mapsto T_H(f)$ is called the \emph{stress energy tensor} of $H$.

\begin{proof}
We first prove uniqueness of the stress-energy tensor.
For $f$ real-valued and supported in some interval $I$, the formula
\[
\qquad\qquad\quad
iT_H(f) = \tfrac{d}{dt}\big|_{t=0}\Big( \exp(iT_H(tf))\Big)
= \tfrac{d}{dt}\big|_{t=0}\Big(\rho_I\big(U_0(\exp(tf \tfrac{d}{d\theta}))\big)\Big)\qquad\quad(\mathrm{mod}\,\, \bbC{\cdot}1_H)
\]
determines $T_H(f)$ up to an additive constant in terms of the projective action of $\Diff(S^1)$ on $H_0$.
The linearity of $f \mapsto T_H(f)$
(along with the fact that finite-energy vectors are assumed to be a core) then also determines $T_H(f)$ up to a constant in the 
absence of any reality and support conditions.
The requirement that $L_n$ satisfy the Virasoro algebra relations fixes the constant when $f$ is a trigonometric polynomial,
and the continuity assumption fixes the constant for all $f$.

We now prove existence.
Let $c$ be the central charge of $\cA$, and let $\cA_{\Vir_c}\subset \cA$ be the Virasoro sub-net, given by
\begin{equation}\label{eq: def: Vir net}
\cA_{\Vir_c}(I)=\big\{\exp(iT_{H_0}(f)):\supp(f)\subset I \text{ and } f \text{ is real-valued} \big\}''.
\end{equation}
The characterising property
$
\rho_I\big(U_0(\exp(f \tfrac{d}{d\theta}))\big) = \exp(iT_H(f))
$
of $T_H(f)$ is stated purely in terms of the elements $U_0(\exp(f \tfrac{d}{d\theta}))$ of $\cA_{\Vir_c}(I)$ (and the definition of $H^{f.e.}\!$ also only uses the action of $\cA_{\Vir_c}$),
so we may assume without loss of generality that $\cA = \cA_{\Vir_c}$.
By decomposing $H$ into irreducible $\cA_{\Vir_c}$-reps (using our assumption that $L_0$ has discrete spectrum), we may furthermore assume that $H$ is irreducible.
Irreducible representations of $\cA_{\Vir_c}$ have been classified in \cite[Prop. 2.1]{Carpi04}, \cite[Cor. 6.4.7]{Weiner05}, and \cite[Thm. 5.6]{Weiner17}: they are all isomorphic to some $H_{c,h}$, characterised by a choice of minimal energy $h\ge 0$, which we now describe.
Let $M=M_{c,h}$ be the simple unitary module for the Virasoro algebra $\Vir_c$ with minimal energy $h$, and let $H=H_{c,h}$ be its Hilbert space completion, so that $H^{f.e.}=M$.
Then $\cA_{\Vir_c}$ acts on $H$ in the following way:

For $f=\sum a_ne^{in\theta}\in C^\infty(S^1)$, let $T_H(f)$ be the closure of $\sum a_n L_n$ acting on $M$.
These operators satisfy $T_H(f)^* = T_H(\overline{f})$ (see \cite[\S3.2.2]{Weiner05}).
Consider the action $U:\Diff(S^1)\to PU(H)$ uniquely specified by the requirement that
\begin{equation}\label{eq: def of proj action U on Vir module}
U\big(\exp(f \tfrac{d}{d\theta}))\big) = \exp\big(iT_{H}(f)\big)
\end{equation}
when $f$ is real-valued \cite{GoWa85,ToledanoLaredo99}.\footnote{This defines a projective action of $\Diff(S^1)$ because $H$ is irreducible. Otherwise, it would only give a projective action of its universal cover $\tDiff(S^1)$.}
In the same way, we have $T_{H_0}(f)$ acting on the vacuum sector $H_0=H_{c,0}$ of $\cA_{\Vir_c}$, and $U_0:\Diff(S^1)\to PU(H_0)$ given by $U_0(\exp(f \tfrac{d}{d\theta}))) = \exp(iT_{H_0}(f))$.
By \eqref{eq: def: Vir net}, the operator $T_{H_0}(f)$ is affiliated with $\cA_{\Vir_c}(I)$ when $\supp(f) \subset I$ and $f$ is real-valued.
The action %$\rho$ 
of $\cA_{\Vir_c}(I)$ on $H$ is then given by \cite[Eqn. (31)]{Carpi04}:
\begin{equation}\label{eq: Carpi (31)}
\rho_I\big(\exp(iT_{H_0}(f))\big) =
\exp(iT_H(f)).
\end{equation}
In particular,
\(
i\rho_I(T_{H_0}(f))=
\tfrac{d}{dt}\big|_{t=0}\Big(\rho_I(\exp(iT_{H_0}(tf)))\Big)=
\tfrac{d}{dt}\big|_{t=0}\Big(\exp(iT_H(tf))\Big)=iT_H(f)
\).

By \cite[Lem. 4.1]{HenriquesTenerIntegratingax}, 
$T_{H}(f+ig)$ is the closure of $T_{H}(f)+iT_{H}(g)$,
and similarly for $T_{H_0}(f+ig)$.
It follows that $T_{H_0}(f)$ is affiliated with $\cA_{\Vir_c}(I)$,
and that $\rho_I(T_{H_0}(f))=T_H(f)$ holds
for all functions $f$ with support in $I$, not just real-valued ones.

To finish the proof, we must still show that $L_0:=T_H(1)$ agrees (up to constant) with the operator used in Definition~\ref{def: H^f.e.}.
By \eqref{eq: Carpi (31)} and by continuity, the action of $\cA_{\Vir_c}(I)$ on $H$ satisfies %$\rho_I(U_0(\exp(f \tfrac{d}{d\theta}))) = U(\exp(f \tfrac{d}{d\theta})))$
$\rho_I(U_0(\varphi))=U(\varphi)$ for all $\varphi\in\Diff_I(S^1)$.
The projective actions \eqref{eq:Diff_I -> PU(H)} and \eqref{eq: def of proj action U on Vir module} of $\Diff_I(S^1)$ on $H$ therefore agree.
Taking the colimit over $I$,
the projective actions \eqref{eq:universal cover Diff -> PU(H)} and \eqref{eq: def of proj action U on Vir module} of $\tDiff(S^1)$ on $H$ also agree.
%The induced \emph{honest} actions of $\tDiff_c(S^1)$ on $H$ therefore also agree.
Letting $r_t:=\exp(it\ell_0)\in\tRot(S^1)$ be %the diffeomorphism of 
rotation by angle $t$, the desired result follows:
\[
iL_0=iT_H(1)=\tfrac{d}{dt}\big|_{t=0}\big(\exp(iT_H(t))\big)=\tfrac{d}{dt}\big|_{t=0}U\big(\exp(t \tfrac{d}{d\theta}))\big)=\tfrac{d}{dt}\big|_{t=0}U(r_t).\qedhere
\]
\end{proof}

By Proposition \ref{prop: exist: stress energy tensor}, \eqref{eq:universal cover Diff -> PU(H)} and \eqref{eq: def of proj action U on Vir module},
we have compatible actions $\tDiff(S^1)\to PU(H)$ and $\Vir_c\to \End(H^{f.e.})$ on any representation $H$ of a conformal net.
And by \eqref{eq: def how annuli act}, the latter integrates to a holomorphic action
\begin{equation}\label{eq: pi:Ann_c to B(H)}
\pi:\tAnn_c\to B(H)
\end{equation}
extending and lifting the action of $\tDiff(S^1)$.
When $H=H_0$ is the vacuum sector of our conformal net, the eigenvalues of $L_0$ are all integral
(because the homomorphism 
$\pi:\tMob\to U(H_0)$ descends to $\Mob$),
so the map $\tAnn_c\to B(H_0)$ descends to an action
\begin{equation}\label{eq: pi_0:Ann_c to B(H_0)}
\pi_0:\Ann_c\to B(H_0).
\end{equation}

Let now $I\subset S^1$ be an interval.
The formula \eqref{eq: def how annuli act} for the action of %the sub-semigroup 
$\Ann_{c}(I)\subset \Ann_c$ only involves operators %that are 
affiliated with $\cA(I)$.
So the action of $\Ann_{c}(I)$ on $H_0$ lands in the subalgebra $\cA(I)\subset B(H_0)$,
and \eqref{eq: pi_0:Ann_c to B(H_0)} restricts to a homomorphism $\pi_0:\Ann_{c}(I)\to \cA(I)$.
The following was proven in \cite[Thm. 7.4]{HenriquesTenerIntegratingax}:

\begin{prop}\label{lem: local annuli lie in local algebra}
For every representation $(H,\rho)$ of $\cA$, the following diagram commutes
\[
\begin{tikzcd}
\Ann_{c}(I) \arrow[rd, "\pi"'] \arrow[r, "\pi_0"] & \cA(I) \arrow[d, "\rho_I"] &\\
& B(H)
\end{tikzcd}
\]
where 
$\pi_0$ and $\pi$ are given by \eqref{eq: pi_0:Ann_c to B(H_0)} and \eqref{eq: pi:Ann_c to B(H)}, respectively.
\end{prop}

\subsection{Unparametrised annuli}\label{sec: unparametrised annuli}

Let $\cA$ be a conformal net. As in \eqref{eq:vacuum sector def}, we let $H_0(D) = \Conf(\bbD,D) \times_{\Mob} H_0$ be the vacuum sector associated to a disc $D$, and recall that $H_0=H_0(\bbD)$.
If $D_{in}\subset D_{out}$ is a pair of complex discs included into one another, with $A:=D_{out} \setminus \mathring D_{in}$, then there is an associated operator
\begin{equation}\label{Y_A def}
Y_A:H_0(D_{in}) \to H_0(D_{out})
\end{equation}
defined as follows.
Any choice of holomorphic parametrizations $\varphi_{in} : \bbD \to D_{in}$ and $\varphi_{out} : \bbD \to D_{out}$
provides a lift of $A$ to an element $f:=\varphi_{out}^{-1}\circ\varphi_{in}\in \Univ$.
Writing $(\varphi, \xi)$ for the element previously denoted $[(\varphi, \xi)]$,
we set
\[
Y_A(\varphi_{in}, \xi) := (\varphi_{out}, \pi(f) \xi),
\]
where $\pi:\Univ\subset \Ann_c\to B(H_0)$ is the map \eqref{eq: pi_0:Ann_c to B(H_0)}.
To see that the map $Y_A$ is well defined, we first verify that it is independent of the way one represents an element of $H_0(D_{in})$ as $(\varphi_{in}, \xi)$.
Indeed, for any $\psi\in\Mob$, the two expressions 
$Y_A(\varphi_{in}\circ \psi, \xi)=(\varphi_{out}, \pi({f\circ \psi}) \xi)$
and
$Y_A(\varphi_{in},\pi(\psi) \xi)=(\varphi_{out}, \pi(f) \pi(\psi) \xi)$
are equal in $H_0(D_{out})$.
We must also check that $Y_A$ is independent of the choice of parametrization $\varphi_{out}$.
Indeed, had we chosen $\varphi_{out} \circ \psi$ instead of $\varphi_{out}$ for some $\psi \in \Mob$, we would have obtained the same element of $H_0(D_{out})$:
\[
Y_A(\varphi_{in}, \xi) 
 = (\varphi_{out} \circ \psi, \pi({\psi^{-1}\circ f}) \xi)
 = (\varphi_{out},  \pi(f) \xi).
\]
This finishes the proof that the map \eqref{Y_A def} is well defined.
The maps $Y_A$ are compatible with composition,
\[
Y_{A_1\cup A_2} = Y_{A_1}\circ Y_{A_2},
\]
making $D\mapsto H_0(D)$ into a functor from the category of discs (Definition~\ref{def: disc and circle}) and holomorphic embeddings (a.k.a. univalent maps) to the category of Hilbert spaces and bounded linear maps.

Recall from \eqref{eq:vacuum sector def} that the vacuum vector $\Omega_D\in H_0(D)$ is the equivalence class of $(\varphi,\Omega)$, for any $\varphi \in \Conf(\bbD,D)$.
By \eqref{eq:  A Omega=Omega}, we have:
\begin{lem}\label{lem: annuli map Om to Om}
Let $D_{in}\subset D_{out}$ be discs,
and let $\Omega_{in/out}:=\Omega_{D_{in/out}}\in H_0(D_{in/out})$ be their respective vacuum vectors. 
Then
\[
\,\,Y_{D_{out} \setminus \mathring D_{in}} \Omega_{in}=\Omega_{out}.\!\!
\hspace{6.3cm}\square\hspace{-6.3cm}
\]
\end{lem}

The operators $Y_A$ associated to thin annuli will be particularly important for us:

\begin{lem}\label{lem: thin annuli intertwine action of boundary intervals}
Let $A=D_{out} \setminus \mathring D_{in}$ be a thin annulus, and let $I\subset \partial_{in}A \cap \partial_{out}A$ be an interval.
Then $Y_A$ is equivariant for the actions of $\cA(I)$ on $H_0(D_{in})$ and on $H_0(D_{out})$.
\end{lem}
\begin{proof}
Recall from \eqref{eq: action of A(I) on H_0(D)} that the action of $x \in \cA(I)$ on $H_0(D_{in})$ is given by $x(\varphi_{in},\xi) = (\varphi_{in}, {\varphi^{-1}_{in}}(x) \xi)$, and similarly for $H_0(D_{out})$.
Letting $f=\varphi_{out}^{-1}\circ\varphi_{in}$,
we need to show that the following two expressions are equal in $H_0(D_{out})$:
\begin{align*}
&x Y_A(\varphi_{in}, \xi)
= x (\varphi_{out}, \pi(f) \xi) = (\varphi_{out}, {\varphi^{-1}_{out}}(x) \pi(f) \xi)\\
\text{and}\qquad
&Y_A x (\varphi_{in}, \xi) = Y_A(\varphi_{in}, {\varphi^{-1}_{in}}(x) \xi)= (\varphi_{out}, \pi(f) {\varphi^{-1}_{in}}(x) \xi).
\end{align*}
We check that ${\varphi^{-1}_{out}}(x) \pi(f)=\pi({f}){\varphi^{-1}_{in}}(x)$.
Let $g\in\Diff_c(S^1)$ be an element whose underlying diffeomorphism $g_0\in\Diff(S^1)$ satisfies $(g_0\circ f)|_{{\varphi^{-1}_{in}}(I)} = \id$,
and let $U_g\in U(H_0)$ be the corresponding unitary.
By construction, we have
\[
A_{g}\cup A_f\in\Ann_c(\varphi^{-1}_{in}(I)'),
\]
so $U_g\pi(f)=\pi(A_{g}\cup A_f)$ commutes with $\varphi^{-1}_{in}(x)$. It follows that
\[
{\varphi^{-1}_{out}}(x) \pi(f) = {\varphi^{-1}_{out}}(x) U_g^*U_g\pi(f) = U_g^*{\varphi^{-1}_{in}}(x) U_g\pi(f) = U_g^*U_g\pi(f){\varphi^{-1}_{in}}(x) = \pi(f){\varphi^{-1}_{in}}(x).
\]
\end{proof}

\begin{rem}
Note that, by the Reeh-Schlieder theorem, Lemmas~\ref{lem: annuli map Om to Om} and~\ref{lem: thin annuli intertwine action of boundary intervals} taken together completely
characterise the operator $Y_A:H_0(D_{in}) \to H_0(D_{out})$ for thin annuli.
\end{rem}

The operators $Y_A$ assigned to unparametrised annuli are closely related to the representation $\pi_0:\Ann_c \to B(H_0)$ described in Section~\ref{sec: reps of semigroup of annuli}.

\begin{lem}\label{lem: Aunderline in terms of A}
Let $A = D_{out} \setminus \mathring{D}_{in}$ be an annulus with boundary parametrizations $\varphi_{in/out}:S^1\to\partial D_{in/out}$ (we do not require $\varphi_{in/out}: S^1\to\partial D_{in/out}$ to extend to holomorphic maps $\bbD\to D_{in/out}$),
and let
\[
U(\varphi_{in}):H_0(\bbD)  \to  H_0(D_{in})
\qquad\text{and}\qquad
U(\varphi_{out}):H_0(\bbD)  \to  H_0(D_{out})
\]
be unitaries implementing $\varphi_{in/out}$ (Definition~\ref{def: unitary implements}).
Then for any lift $\underline{A}\in\Ann_c$ of $(A,\varphi_{in},\varphi_{out})$, the equality
\[
\pi_0(\underline{A}) = U(\varphi_{out})^*Y_A U(\varphi_{in})
\]
holds up to scalar.
\end{lem}
\begin{proof}
Let $\varphi'_{in/out}: S^1\to\partial D_{in/out}$ be diffeomorphisms
that extend to holomorphic maps $\bbD\to D_{in/out}$, so that $\underline{A}':=(A,\varphi_{in}',\varphi_{out}')\in\Univ$.
The unitaries $U(\varphi_{in/out}'):H_0\to H_0(D_{in/out})$ defined by $U(\varphi_{in/out}')(\xi):=(\varphi_{in/out},\xi)$ implement $\varphi_{in/out}$.
And the map $Y_A$ satisfies
\[
\pi(\underline{A}') = U(\varphi_{out}')^*Y_A U(\varphi_{in}')        % A' = U(?'out)*A U(?'in)  
\]
by definition.
Writing
$\varphi_{in}=\varphi_{in}'\circ \varphi_{in}''$ and $\varphi_{out}=\varphi_{out}'\circ \varphi_{out}''$,
and letting $U(\varphi_{in/out}'')$ be unitaries implementing $\varphi_{in/out}''$,
we then have
\[
\pi(\underline{A})
= U(\varphi_{out}'')^*\pi(\underline{A}')U(\varphi_{in}'')
= U(\varphi_{out}'')^*U(\varphi_{out}')^* Y_A U(\varphi_{in}')U(\varphi_{in}'')
 = U(\varphi_{out})^*Y_A U(\varphi_{in}),
% A
% = U(?''out)* A' U(?''in)
% = U(?''out)*U(?'out)* A U(?'in)U(?''in)
% = U(?out)*A U(?in)  
\]
where all the equalities hold up to scalar.
\end{proof}

From now on, when no confusion can arise, we shall denote the operator $Y_A:H_0(D_{in}) \to H_0(D_{out})$ simply by
\begin{equation} \label{eq: A:H_0(D_in) --> H_0(D_out)}
A\,:\,H_0(D_{in}) \to H_0(D_{out}).
\end{equation}

Our next task will be to extend the construction \eqref{Y_A def} to assign anti-linear maps $H_0(D_1) \to H_0(D_2)$ to
anti-holomorphic embeddings $D_1\to D_2$ between discs.
Before that, we need some background on the Bisognano-Wichmann theorem
(the reader may consult \cite[Thm. 2.19]{GabbianiFrohlich93}, \cite{BuchholzSchulz-Mirbach90}, \cite[Thm. 2.1]{DantoniLongoRadulescu01}, \cite[Thm. II.9]{Borchers92}, or \cite[Thm. 6.2.3]{LongoLectureNotesII}).

Recall that $\bbD$ denotes the standard unit disc. Let $I_-\subset \partial \bbD$ be the lower semi-circle, and let $I_+:=I_-'$ be its closed complement, so that $\partial I_-=\partial I_+=\{\pm1\}$.
Let $\delta:i\bbR\to \Mob=\Aut(\bbD)$ be the unique action of $i\bbR$ on $\bbD$ that fixes $\{\pm 1\}$ and is normalised so that the derivative at $1$ of $\delta_{it}$ is $e^t$.
Finally, let $\vartheta(z)=\bar z$.
The Bisognano Wichman theorem for conformal nets says that $H_0$ is canonically isomorphic to the $L^2$-space of the von Neumann algebra $\cA(I_-)$, and that the modular flow $\Delta^{it} : L^2\cA(I_-) \to L^2\cA(I_-)$ agrees under that isomorphism with %the action on $\cA(I_+)$ induced by 
the action on $H_0=H_0(\bbD)$ induced by $\delta_{2\pi it}:\bbD \to \bbD$.
Moreover, the modular conjugation $J:L^2\cA(I_-) \to L^2\cA(I_-)$ corresponds to an operator
\begin{equation}\label{eq: first Theta}
\Theta: H_0 \to H_0
\end{equation}
that satisfies $\Theta\cA(I)\Theta=\cA(\vartheta I)$ for every interval $I\subset S^1$.
By \cite[Prop. 4.7]{BartelsDouglasHenriques15}, the actions of $\Theta$ and $\Diff_c(S^1)$ on $H_0$ assemble to an action of the semidirect product $\Diff_c(S^1) \rtimes \bbZ/2$.
By differentiating the $1$-parameter groups of unitaries
\begin{align*}
\Theta \exp(it(L_n+L_{-n})) \Theta &= \exp(-it(L_n+L_{-n}))\\
\text{and}\quad\,\,\,
\Theta \exp(t(L_n-L_{-n})) \Theta &= \exp(t(L_n-L_{-n}))
\end{align*}
and using that $\Theta$ is conjugate linear, we learn that
\begin{equation}\label{eq: Th L Th= L}
\Theta (\lambda L_n) \Theta = \bar\lambda L_n\qquad \forall \lambda\in\bbC.
\end{equation}
It follows by continuity that
\begin{equation}\label{eq: Theta T_H0(f)Theta}
\Theta T_{H_0}(f)\Theta =T_{H_0}(f^\vartheta)
\end{equation}
for all $f\in\cC^\infty(S^1)$, where $f^\vartheta(\theta):=\overline{f(-\theta)}$.
%$\ell_n:=-ie^{in\theta}\tfrac{\partial}{\partial \theta}\in \cX_\bbC$ and $\ell_n=z^{n+1}\tfrac{\partial}{\partial z}$.
Combining \eqref{eq: Theta T_H0(f)Theta} with the explicit formula \eqref{eq: def how annuli act} for the action 
of $\Ann_c$, % on $H_0$, 
we see that the actions of $\Theta$ and $\Ann_c$ on $H_0$ assemble to an action
\begin{equation}\label{eq: pi:Ann_c rtimes Z/2 -> B(H_0)}
\pi:\Ann_c \rtimes\, \bbZ/2\to B(H_0).
\end{equation}

\begin{rem}
The semigroup $\Ann \rtimes\, \bbZ/2$ is the set of isomorphism classes of \emph{unoriented annuli with parametrised boundary},
and its subgroup
${\Diff(S^1) \rtimes \bbZ/2} $ is the group of all diffeomorphisms of the circle (not necessarily orientation preserving).
Here, \emph{unoriented annuli} are locally ringed spaces $A=(A,\cO_A)$ which are isomorphic to embedded annuli, as in Definition~\ref{def: annulus},
with the only difference being that we now treat $\cO_A$ as a sheaf of $\bbR$-algebras as opposed to a sheaf of $\bbC$-algebras.
Additionally, the boundary parametrizations $\varphi_{in}:S^1\to \partial_{in}A$ and $\varphi_{out}:S^1\to \partial_{out}A$ are no longer required to satisfy any orientation conditions.

The semigroup of unoriented annuli with parametrised boundary admits an extension by~$\bbC^\times$
\[
0 \to \bbC^\times\to \Ann_c \rtimes\, \bbZ/2 \to \Ann \rtimes\, \bbZ/2 \to 0
\]
but this is not a central extension, because $\bbZ/2$ acts non-trivially on the center $\bbC^\times$ of $\Ann_c$.
\end{rem}

We can now return to the task of extending the construction \eqref{Y_A def} to assign anti-linear maps to anti-holomorphic embeddings between discs.
Given two discs $D_1$, $D_2$ and 
an embedding $f:D_1\to D_2$ which is either holomorphic or anti-holomorphic, we have an associated linear or anti-linear map $f_*:H_0(D_1) \to H_0(D_2)$. It is given by
\begin{equation}\label{eqn: pushforward by (anti)holomorphic map}
f_*(\varphi_1, \xi) := (\varphi_2, \pi(\varphi_2^{-1}\circ f\circ\varphi_1) \xi),
\end{equation}
where $\varphi_{i} : \bbD \to D_{i}$ are holomorphic parametrisations,
$\varphi_2^{-1}\circ f\circ\varphi_1\in\Univ\rtimes\,\bbZ/2\subset \Ann_c\rtimes\,\bbZ/2$, and $\pi$ is as in \eqref{eq: pi:Ann_c rtimes Z/2 -> B(H_0)}.
This construction extends \eqref{Y_A def}, and is well defined for the same reasons.
It is functorial in the sense that given $f:D_1\to D_2$ and $g:D_2\to D_3$, we have
\[
g_*\circ f_*=(g\circ f)_*:H_0(D_1)\to H_0(D_3).
\]
In summary, our construction equips the assignment $D\mapsto H_0(D)$ with the structure of a functor from the category of discs and embeddings that are either holomorphic or anti-holomorphic, to the category of Hilbert spaces and bounded maps that either complex linear or complex anti-linear.

\section{Worm-shaped insertions}\label{sec: worms}

Given a vertex algebra $V=\bigoplus_{n\in\bbN} V_n$, a collection $z_1,\ldots,z_n$ of distinct points in the complex plane, and elements $v_1,\ldots,v_n \in V$, we can form the element 
$  
Y(v_1,z_1)\ldots Y(v_n,z_n)\Omega
$ 
of the algebraic completion $\widehat V:=\prod_{n\in\bbN} V_n$.
More generally, if we're also given 
local coordinates
$g_i\in \Aut(\bbC[[t]])$ at the points $z_i$,
we can form the element
\begin{equation}\label{eq: Y(g_1(v_1),z_1) ... Y(g_n(v_n),z_n)}
Y(g_1(v_1),z_1)\ldots Y(g_n(v_n),z_n)\Omega\,\in\,\widehat V.
\end{equation}
If the VOA is moreover unitary, and the points are in $\mathring\bbD$ (i.e., satisfy $|z_i|<1$),
then the above element \eqref{eq: Y(g_1(v_1),z_1) ... Y(g_n(v_n),z_n)} lies in fact in the Hilbert space completion of $V$.
In this section, we study an analogous construction which takes place in the world of conformal nets.
The $z_i$ and $g_i$ are replaced by disjoint intervals $I_i\subset \bbD$, and the $v_i$ are replaced by elements $x_i\in \cA(I_i)$ of the conformal net.

\subsection{Construction of worm-shaped insertions}

\begin{defn}\label{def: extendable}
Let $D$ be a disc.
An embedded interval $I\subset D$ is called \emph{extendable} if there exists a larger interval $I_+\subset D$
containing $I$ in its interior, and if the orientations of $I$ and $\partial D$ agree whenever the two intersect (they are necessarily tangent at such points).
\end{defn}

Let $\cA$ be a conformal net, and let $H_0$ be its vacuum sector.
In this section, for every collection of disjoint oriented extendable intervals $I_i\subset \bbD$, and elements $x_i\in \cA(I_i)$, we will construct a vector 
$| x_1 \ldots x_n \rangle \in H_0$.
More generally, for any disc $D$, disjoint oriented extendable intervals $I_i\subset  D$, and elements $x_i\in \cA(I_i)$, we will construct a corresponding vector 
\begin{equation*} %\label{eq: | x_1 ... x_n >}
| x_1 \ldots x_n \rangle_D \in H_0(D).
\end{equation*}
These vectors are completely determined by the following two requirements:
\begin{itemize}
\item
If all the intervals $I_i$ lie in the boundary of the disc, then
\begin{equation}\label{eq: worms in a disc}
| x_1 \ldots x_n \rangle_D = x_1 \cdots x_n \Omega_D \in H_0(D),
\end{equation}
where the right hand side uses the actions \eqref{eq: action of A(I) on H_0(D)} of the algebras $\cA(I_i)$ on $H_0(D)$.
(And the disc with no insertions corresponds to the vacuum vector.) % $\Omega_D \in H_0(D)$:
\item
If $D_{in}\subset D_{out}$ are discs, and $A=D_{out}\setminus \mathring D_{in}$ is the corresponding annlus, %and the intervals $I_i$ are (extendable) in $D_{in}$, 
then
\[
| x_1 \ldots x_n \rangle_{D_{out}} = Y_A | x_1 \ldots x_n \rangle_{D_{in}}.
\]
\end{itemize}
Given a disc $D$,
disjoint oriented extendable intervals $I_i \subset D$
(the intervals $I_i$ are allowed to touch the boundary of $D$),
and elements $x_i \in \cA(I_i)$, pick a sequence of discs\footnote{If the intersections $I_i\cap \partial D$ have multiple connected components, such a sequence of discs might only exist after permuting the order of the intervals $I_i$.}
\begin{equation}\label{eq:DDD 1}
D_1\subset D_2\subset \ldots \subset D_n=D
\end{equation}
with $I_i\subset \partial D_i$ (compatibly with the orientations), and let $A_i:=D_i\setminus D_{i-1}$:
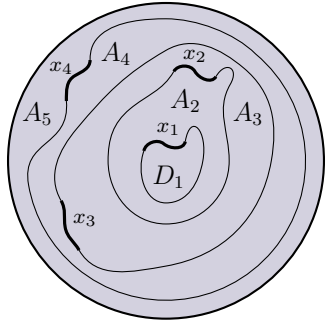
\begin{equation*}
\begin{tikzpicture}[scale=0.7,baseline]
  % Draw the filled circle
  \filldraw[fill=red!10!blue!20!gray!30!white, draw=black, thick] (0,0) circle (3cm);
%  \filldraw[fill=white, draw=black, thick] (0,0) circle (1cm);
  % Squiggle 1: center, rotated wave
  \begin{scope}[xshift=-10]
  \draw[very thick,rotate around={20:(0,0)}] (0,0.2) coordinate(a0) to[out=60,in=120] (0.4,0.2) to[out=-60,in=240] (0.8,0.1) coordinate(b0);
  \node at (0.4,0.6) {${\scriptstyle x_1}$};
  \end{scope}
  % Squiggle 2: top right, tighter wave
  \begin{scope}[xshift=-10]
  \draw[very thick] (0.5,1.7) coordinate(a1) to[out=60,in=120] (0.9,1.7)  to[out=-60,in=240] (1.3,1.6) coordinate(b1);
  \node at (0.9,2) {${\scriptstyle x_2}$};
  \end{scope}
  % Squiggle 3: bottom left, flat wiggle
  \begin{scope}[yshift=-50, xshift=-10, rotate=-70]
  \draw[very thick] (-1.5,-1.2) coordinate(a2) to[out=10,in=170] (-1,-1.3) to[out=-10,in=190] (-0.5,-1.2) coordinate(b2);
  \node at (-1,-0.9) {${\scriptstyle x_3}$};
  \end{scope}
  % Squiggle 4: top left, gentle wave
  \begin{scope}[yshift=60, xshift=-5, rotate=60]
  \draw[very thick] (-1.7,1) coordinate(a3) to[out=30,in=150] (-1.2,1.1) to[out=-30,in=210] (-0.8,1) coordinate(b3);
  \node at (-1.2,1.4) {${\scriptstyle x_4}$};
  \end{scope}
\draw (a0)
to[out=-110, in=180] (270:.8)
to[out=0, in=-100] (0:.7)
to[out=80, in=80, looseness=3] (b0);
\draw (a1)
to[out=-120, in=40] (125:1.2)
to[out=180+40, in=90] (180:1.1)
to[out=-90, in=180] (270:1.2)
to[out=0, in=-90] (0:1.2)
to[out=90, in=-120] (50:1.9)
to[out=60, in=60, looseness=3] (b1);
\draw (a2)
to[out=120, in=-130] (160:1.8)
to[out=50, in=-150](90:2.1)
to[out=30, in=140](50:2.3)
to[out=-40, in=80](0:2)
to[out=-100, in=10](270:2)
to[out=-170, in=-60](b2);
\draw (b3)
to[out=90, in=-140] (117:2.7)
to[out=40, in=180, looseness=.8](90:2.7)
to[out=0, in=135, looseness=.8](50:2.65)
arc(50:-170:2.65)
to[out=100, in=-85, looseness=1.3](a3);
\node[scale=.85] at (.05,-.3) {$D_1$};
\node[scale=.85] at (72:1.2) {$A_2$};
\node[scale=.85] at (30:1.8) {$A_3$};
\node[scale=.85] at (115:2.25) {$A_4$};
\node[scale=.85] at (160:2.6) {$A_5$};
\end{tikzpicture}
\end{equation*}
Letting $\Omega_1\in H_0(D_1)$ be the vacuum vector associated to $D_1$,
we set
\begin{equation}\label{eq: <x,...,x>_D}
| x_1 x_2 \ldots x_n\rangle_{(D_j)} \,:=\, A_n x_{n-1} \cdots A_3 x_2 A_2 x_1\Omega_1
\end{equation}
where $A_j$ denotes the operator $Y_{A_j}:H_0(D_{i-1})\to H_0(D_i)$, as in \eqref{eq: A:H_0(D_in) --> H_0(D_out)}.
In Proposition~\ref{proposition independence} below, we will show that the above vector is independent of the choice of discs $D_i$ and of the order of the insertions.
The subscript in the left hand side of \eqref{eq: <x,...,x>_D} emphasizes the potential dependence on these choices.
(After Proposition~\ref{proposition independence} has been proven we will drop that subscript and replace it by the subscript $D$.)

Suppose now that we are given an annulus $A = D_{out}\setminus \mathring D_{in}$ with disjoin oriented extendable intervals $I_i\subset A$, and elements $x_i\in \cA(I_i)$. Let
\begin{equation}\label{eq:DDD 2}
D_{in} = D_0\subset D_1\subset \ldots \subset D_{n+1}=D_{out}
\end{equation}
be a sequence of discs\footnote{Once again, if $I_i\cap \partial A$ have multiple connected components, then such a sequence of discs might only exist after permuting the order of the $I_i$'s.} with $I_i\subset \partial D_i$, and let $A_i:=D_i\setminus D_{i-1}$:
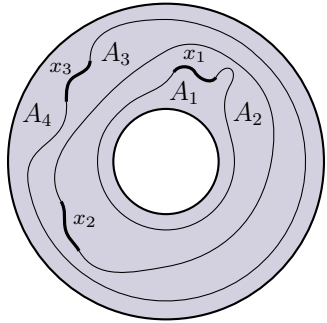
\begin{equation*}
\begin{tikzpicture}[scale=0.7,baseline]
  % Draw the filled circle
  \filldraw[fill=red!10!blue!20!gray!30!white, draw=black, thick] (0,0) circle (3cm);
  \filldraw[fill=white, draw=black, thick] (0,0) circle (1cm);
  % Squiggle 1: top right, tighter wave
  \begin{scope}[xshift=-10]
  \draw[very thick] (0.5,1.7) coordinate(a1) to[out=60,in=120] (0.9,1.7)  to[out=-60,in=240] (1.3,1.6) coordinate(b1);
  \node at (0.9,2) {${\scriptstyle x_1}$};
  \end{scope}
  % Squiggle 2: bottom left, flat wiggle
  \begin{scope}[yshift=-50, xshift=-10, rotate=-70]
  \draw[very thick] (-1.5,-1.2) coordinate(a2) to[out=10,in=170] (-1,-1.3) to[out=-10,in=190] (-0.5,-1.2) coordinate(b2);
  \node at (-1,-0.9) {${\scriptstyle x_2}$};
  \end{scope}
  % Squiggle 3: top left, gentle wave
  \begin{scope}[yshift=60, xshift=-5, rotate=60]
  \draw[very thick] (-1.7,1) coordinate(a3) to[out=30,in=150] (-1.2,1.1) to[out=-30,in=210] (-0.8,1) coordinate(b3);
  \node at (-1.2,1.4) {${\scriptstyle x_3}$};
  \end{scope}
\draw (a1)
to[out=-120, in=40] (125:1.3)
to[out=180+40, in=90] (180:1.3)
to[out=-90, in=180] (270:1.3)
to[out=0, in=-90] (0:1.3)
to[out=90, in=-120] (50:1.9)
to[out=60, in=60, looseness=3] (b1);
\draw (a2)
to[out=120, in=-130] (160:1.8)
to[out=50, in=-150](90:2.1)
to[out=30, in=140](50:2.3)
to[out=-40, in=80](0:2)
to[out=-100, in=10](270:2)
to[out=-170, in=-60](b2);
\draw (b3)
to[out=90, in=-140] (117:2.7)
to[out=40, in=180, looseness=.8](90:2.7)
to[out=0, in=135, looseness=.8](50:2.65)
arc(50:-170:2.65)
to[out=100, in=-85, looseness=1.3](a3);
\node[scale=.85] at (75:1.35) {$A_1$};
\node[scale=.85] at (30:1.8) {$A_2$};
\node[scale=.85] at (115:2.25) {$A_3$};
\node[scale=.85] at (160:2.6) {$A_4$};
\end{tikzpicture}
\end{equation*}
Then we may form the operator
\begin{equation}\label{eq: A[x,...,x]_D}
A[x_1 x_2 \ldots x_n]_{(D_j)}:=A_{n+1}x_n \cdots  x_2 A_2 x_1 A_1 : H_0(D_{in}) \to H_0(D_{out}).
\end{equation}
where $A_j=Y_{A_j}$ %:H_0(D_{i-1})\to H_0(D_i)$ 
as in \eqref{eq: A:H_0(D_in) --> H_0(D_out)}.

Given boundary parametrizations $\varphi_{in/out}$ of $A$, a $z$-lift $\underline{A}\in \tAnn_c$, and a representation $K$ of the conformal net,
we can also form the operator
\begin{equation}\label{eq: A[x,...,x]_D - with parametrisations}
\underline{A}[x_1 x_2 \ldots x_n]_{(D_j)}\,:=\,\underline{A}_{n+1}x_n \cdots  x_2 \underline{A}_2 x_1 \underline{A}_1  : K \to K.
\end{equation}
Here, $\underline{A}_j\in \tAnn_c$ are $z$-lifts of the annuli $A_j$ which appear in \eqref{eq: A[x,...,x]_D}, subject to the constraint $\prod \underline{A}_j = \underline{A}$, and
the parametrisation $\varphi_j$ of $\partial_{out}A_j$ is required to agree with that of $\partial_{in}A_{j+1}$.
The elements $x_j\in\cA(I_j)$ act on $K$ by means of the homomorphism
$\varphi_j^*:\cA(I_j)\to \cA(\varphi_j^{-1}(I_j))$.

Note that if $\underline{A} \in \tAnn_c(I)$, $K=H_0$, and the intervals $I_i$ are disjoint from $\varphi_{in}(I')=\varphi_{out}(I')$, then provided we make suitable choices of intermediate parametrizations, we have
\begin{equation}
\label{eqn: annulus with worm insertions is local}
\underline{A}[x_1 \ldots x_n]_{(D_j)} \in \cA(I)
\end{equation}
by construction.

When $K=H_0$, we may compare \eqref{eq: A[x,...,x]_D} and \eqref{eq: A[x,...,x]_D - with parametrisations}, as follows.
Let $A = D_{out} \setminus \mathring{D}_{in}$ and $\varphi_{in/out}:S^1\to\partial D_{in/out}$ be as in Lemma \ref{lem: Aunderline in terms of A}, and let $\underline A\in \Ann_c$ be a lift of $(A,\varphi_{in},\varphi_{out})$.
By Lemma \ref{lem: Aunderline in terms of A}, we may choose unitaries $U(\varphi_{in/out}):H_0 \to H_0(D_{in/out})$ implementing ${\varphi_{in/out}}$, and $z\in\bbC^\times$ such that $\underline{A}=z\cdot U(\varphi_{out})^*AU(\varphi_{in})$.
We then claim that the same relation also holds in the presence of worm insertions:
\begin{equation}\label{eq: parametrized vs unparametrized annuli}
\underline A[x_1 x_2 \ldots x_n]_{(D_j)} = z\cdot U(\varphi_{out})^* A[x_1 x_2 \ldots x_n]_{(D_j)} U(\varphi_{in}).
\end{equation}
To see that, choose unitaries $U(\varphi_j)$ implementing the intermediate parametrizations $\varphi_j$.
Since $\underline{A} = \prod\underline{A}_j$, we may choose these unitaries so as to have $\underline{A}_i = U(\varphi_i)^*A_iU(\varphi_{i-1})$ for all $i=1, \ldots, n$, and 
$\underline{A}_{n+1} = z\cdot U(\varphi_{n+1})^*A_{n+1}U(\varphi_{n})$
(where by convention $\varphi_0 = \varphi_{in}$ and $\varphi_{n+1} = \varphi_{out}$).
We can then check that the relation \eqref{eq: parametrized vs unparametrized annuli} indeed holds:
\begin{align*}
&\,\,\,\,\,\,\,\,\underline{A}[x_1 x_2 \ldots x_n]_{(D_j)}
\\&= \underline{A}_{n+1} \varphi_{n}^*(x_n)\underline{A}_n \cdots \varphi_1^*(x_1) \underline{A}_1
\\
&= [z\cdot U(\varphi_{out})^*A_{n+1}U(\varphi_{n})]\varphi_{n}^*(x_n)[U(\varphi_{n})^*A_{n}U(\varphi_{n-1})] \cdots \varphi_{1}^*(x_1)[ U(\varphi_{1})^*A_{1}U(\varphi_{in})]
\\
&= z\cdot U(\varphi_{out})^* A_{n+1}x_nA_{n-1} \cdots x_1 A_1 U(\varphi_{in})
\\
&=  z\cdot U(\varphi_{out})^* A[x_1 x_2 \ldots x_n]_{(D_j)} U(\varphi_{in}).
\end{align*}

Our next goal is to show that \eqref{eq: <x,...,x>_D}, \eqref{eq: A[x,...,x]_D}, and \eqref{eq: A[x,...,x]_D - with parametrisations} are independent of the choices of discs and ordering of the insertions.
In the proof of Proposition~\ref{proposition independence} below, we first assume that $I_i\cap \partial D$ and $I_i\cap \partial A$ are either empty or connected.
This ensures that sequences of discs as in \eqref{eq:DDD 1} or \eqref{eq:DDD 2} always exist when needed.
We then explain, at the end of the proof, how to remove that constraint.

Recall that an annulus $A=D_{out}\setminus \mathring D_{in}$ is called `thin'  if $\partial D_{out}\cap \partial D_{in}$ contains an interval.

\begin{prop} \label{proposition independence}
\begin{enumerate}[\rm (\!\!\!\!\; \it a\!\! \rm )\!\! \it]
\item   \label{item 1 independence : discs}
The vector \eqref{eq: <x,...,x>_D} corresponding to a disc with worm insertions is independent of the choice of intermediate discs~$D_i$, and of the ordering of the insertions.
\item   \label{item 2 independence : annuli}
The operator \eqref{eq: A[x,...,x]_D} corresponding to an unparametrised annulus with worm insertions is independent of the choice of intermediate discs $D_i$, and of the ordering of the insertions.
\item   \label{item 3 independence : annuli}
For every representation $K$ of $\cA$, the operator \eqref{eq: A[x,...,x]_D - with parametrisations}  on $K$ corresponding to a $z$-lifted parametrised annulus with worm insertions is independent of the choice of intermediate discs $D_i$, and of the ordering of the insertions.
\end {enumerate}
\end{prop}

In light of Proposition~\ref{proposition independence}, we will later simplify the notations
\eqref{eq: <x,...,x>_D}, \eqref{eq: A[x,...,x]_D}, \eqref{eq: A[x,...,x]_D - with parametrisations} by simply writing 
$| x_1 x_2 \ldots x_n\rangle_D$
in place of $| x_1 x_2 \ldots x_n\rangle_{(D_j)}$,
$A[x_1 x_2 \ldots x_n]$
in place of $A[x_1 x_2 \ldots x_n]_{(D_j)}$,
and
$\underline{A}[x_1 x_2 \ldots x_n]$
in place of $\underline{A}[x_1 x_2 \ldots x_n]_{(D_j)}$.
Note that, by construction, if $A=D_{out}\setminus \mathring D_{in}$ is an annulus, then we have
\begin{equation}\label{eq: compatibility of discs and annuli}
A[x_1 x_2 \ldots x_n] | y_1 y_2 \ldots y_m\rangle_{D_{in}}
= | x_1 x_2 \ldots x_n y_1 y_2 \ldots y_m\rangle_{D_{out}}.
\end{equation}

\begin{proof}
As mentioned above, we start by assuming that the intersections $I_i\cap \partial D$ and $I_i\cap \partial A$ are either empty or connected.

Our proof is by induction, broken into four steps.
Our base case is the statement $(\ref{item 2 independence : annuli})$ for $n=0$, in which case there is nothing to show.
In \emph{Step 1} of our inductive proof, we show that if $(\ref{item 2 independence : annuli})$ holds whenever there are $n$ insertions (for $n \ge 0$) then 
$(\ref{item 1 independence : discs})$ holds whenever there are $n+1$ insertions.
In \emph{Step 2}, we show that if $(\ref{item 1 independence : discs})$ holds whenever there are $n$ insertions for some $n \ge 1$, then $(\ref{item 2 independence : annuli})$ holds whenever there are $n$ insertions and the annulus $A$ is thin.
In \emph{Step 3}, we show that $(\ref{item 2 independence : annuli})$ for thin annuli with $n$ insertions implies the corresponding parametrised version $(\ref{item 3 independence : annuli})$, again just for thin annuli.
Finally, in \emph{Step 4}, we show that the statements $(\ref{item 2 independence : annuli})$ and $(\ref{item 3 independence : annuli})$ 
for thin annuli imply the corresponding statements without the restriction of the annuli being thin.

\noindent\emph{Step 1.}
In this step, assuming that $(\ref{item 2 independence : annuli})$ holds when there are $n$ insertions, we show that $(\ref{item 1 independence : discs})$ holds when there are $n+1$ insertions.
Let $D$ be a disc, let $I_0,\ldots,I_n\subset D$ be 
disjoint oriented extendable intervals, and let $x_i\in\cA(I_i)$. % be algebra elements.
Given discs $D_0\subset \ldots \subset D_n=D$ such that $I_i\subset \partial D_i$, we can form
$| x_0 x_1 \ldots x_n\rangle_{(D_j)}\in H_0(D)$ as in \eqref{eq: <x,...,x>_D}.
Let $\sigma$ be a permutation of the set $\{0,\ldots,n\}$,
let $D'_0\subset \ldots \subset D'_n=D$ be such that $I_{\sigma(i)}\subset\partial D'_i$,
and let 
$| x_{\sigma(0)} x_{\sigma(1)} \ldots x_{\sigma(n)}\rangle_{(D'_j)}\in H_0(D)$ be the corresponding element.
We wish to show that 
$| x_{\sigma(0)} x_{\sigma(1)} \ldots x_{\sigma(n)}\rangle_{(D'_j)}=
| x_0 x_1 \ldots x_n\rangle_{(D_j)}$.

If $\sigma(0)=0$, pick a disc $D''_0\subset D_0\cap D'_0$ that contains $I_0$ in its boundary, and let $\Omega_{D''_0}\in H_0(D''_0)$ be the corresponding vacuum vector. 
Using Lemmas~\ref{lem: annuli map Om to Om} and~\ref{lem: thin annuli intertwine action of boundary intervals},
we then check:
\begin{align*}
| x_0 x_1 \ldots x_n\rangle_{(D_j)}
&= (D\setminus \mathring D_0)[x_1 x_2 \ldots x_n] x_0(D_0\setminus \mathring D''_0)\Omega_{D''_0}
\\&= (D\setminus \mathring D_0)[x_1 x_2 \ldots x_n](D_0\setminus \mathring D''_0) x_0\Omega_{D''_0}
\\&= (D\setminus \mathring D''_0)[x_1 x_2 \ldots x_n] x_0\Omega_{D''_0},
\end{align*}
and similarly 
\begin{align*}
| x_{\sigma(0)} x_{\sigma(1)} \ldots x_{\sigma(n)}\rangle_{(D'_j)}
&= 
(D\setminus \mathring D''_0)[x_{\sigma(1)} x_{\sigma(2)} \ldots x_{\sigma(n)}] x_0\Omega_{D''_0}
\\&=
(D\setminus \mathring D''_0)[x_1 x_2 \ldots x_n] x_0\Omega_{D''_0},
\end{align*}
where the last equality follows from our inductive assumption about $(\ref{item 2 independence : annuli})$ holding whenever there are $n$ insertions.
It follows that 
\begin{equation}\label{eq: step 1 part 2}
| x_{\sigma(0)} x_{\sigma(1)} \ldots x_{\sigma(n)}\rangle_{(D'_j)}=
| x_0 x_1 \ldots x_n\rangle_{(D_j)}.
\end{equation}

If $\sigma(0)\not=0$, pick permutations $\pi$ and $\tau$ of $\{0,\ldots,n\}$ that satisfy $\pi(0)=\tau(1)=0$ and $\pi(1)=\tau(0)=\sigma(0)$. 
Let $D''_0\subset D$ be a disc that contains both $I_0$ and $I_{\sigma(0)}$ in its boundary, and that doesn't intersect any of the other intervals, and let $\Omega_{D''_0}\in H_0(D''_0)$ be the associated vacuum vector.
We then have:
\begin{align*}
| x_{\sigma(0)} x_{\sigma(1)} \ldots x_{\sigma(n)}\rangle_{(D'_j)}
&=(D\setminus \mathring D''_0)[x_{\sigma(1)} x_{\sigma(2)} \ldots x_{\sigma(n)}] x_{\sigma(0)}\Omega_{D''_0}
\\&=(D\setminus \mathring D''_0)[x_{\tau(1)} x_{\tau(2)} \ldots x_{\tau(n)}] x_{\tau(0)}\Omega_{D''_0}
\\&=(D\setminus \mathring D''_0)[ x_{\tau(2)} \ldots x_{\tau(n)}] x_{\tau(1)}x_{\tau(0)}\Omega_{D''_0}
\\&= (D\setminus \mathring D''_0)[x_{\pi(2)} \ldots x_{\pi(n)}] x_{\pi(1)} x_{\pi(0)}\Omega_{D''_0}
\\&= (D\setminus \mathring D''_0)[x_{\pi(1)} x_{\pi(2)} \ldots x_{\pi(n)}] x_{\pi(0)}\Omega_{D''_0}
\\&= (D\setminus \mathring D''_0)[x_1 x_2 \ldots x_n] x_0\Omega_{D''_0}
\\&= | x_0 x_1 \ldots x_n\rangle_{(D_j)},
\end{align*}
where the first and last equalities are special cases of \eqref{eq: step 1 part 2}.
The second, fourth, and sixth equalities hold by our inductive assumption about $(\ref{item 2 independence : annuli})$ holding when there are $\le n$ insertions, while the fourth equality also uses the locality axiom to argue that $x_{\tau(1)} x_{\tau(0)}=x_{\pi(1)} x_{\pi(0)}$.

\noindent\emph{Step 2.}
In this step, assuming that $(\ref{item 1 independence : discs})$ holds with $n$ insertions ($n \ge 1$), 
we show that $(\ref{item 2 independence : annuli})$ holds for thin annuli with the same number of insertions.
Let $D_{in} \subset D_{out}$ be discs such that $\partial D_{out} \cap \partial D_{in}$ contains an interval $J$, and let $A = D_{out} \setminus \mathring D_{in}$ be the corresponding thin annulus.
Let $I_1,\ldots,I_n\subset A\setminus J$ be disjoint oriented extendable intervals, and let $x_i \in \cA(I_i)$.

Given a family of discs 
$D_{in} = D_0 \subset \ldots \subset D_{n+1} = D_{out}$
such that $I_i\subset\partial D_i$,
let\[
X_{(D_j)} := A[x_1 x_2 \ldots x_n]_{(D_j)} : H_0(D_{in}) \to H_0(D_{out})
\]
be as in \eqref{eq: A[x,...,x]_D}.
Given a permutation $\sigma$ of $\{1,\ldots,n\}$ and discs $D_{in} = D'_0 \subset \ldots \subset D'_{n+1} = D_{out}$ such that $I_{\sigma(i)}\subset \partial D'_i$,
we wish to show that $X_{(D_j)}=X_{(D_j')}:=A[x_{\sigma(1)} x_{\sigma(2)} \ldots x_{\sigma(n)}]_{(D'_j)}$.
By our induction hypothesis, we have
\[
X_{(D_j)}\Omega_{D_{in}} = | x_1 x_2 \ldots x_n\rangle_{(D_j)}
= | x_{\sigma(1)} x_{\sigma(2)} \ldots x_{\sigma(n)}\rangle_{(D_j')} = X_{(D_j')}\Omega_{D_{in}}.
\]
The maps $X_{(D_j)} $ and $X_{(D_j')}$ are $\cA(J)$-linear by Lemma~\ref{lem: thin annuli intertwine action of boundary intervals}, and
the vacuum vector $\Omega_{D_{in}}\in H_0(D_{in})$ is cyclic for the action of $\cA(J)$ by the Reeh-Schlieder theorem. It follows that $X_{(D_j)}=X_{(D_j')}$.

\noindent\emph{Step 3.}
In this step we assume that $(\ref{item 2 independence : annuli})$ holds for thin annuli with $n$ insertions ($n\ge 1$), and show that $(\ref{item 3 independence : annuli})$ holds for thin annuli with the same number of insertions.

Let $\underline{A} \in \tAnn_c$ be a thin $z$-lifted annulus, with underlying annulus $A = D_{out} \setminus \mathring{D}_{in}$,
and boundary parametrizations $\varphi_{in/out}:S^1 \to \partial D_{in/out}$.
Let $J\subset \partial_{in}A\cap\partial_{out}A$ be an interval,
$I_1,\ldots,I_n\subset A\setminus J$ be disjoint oriented extendable intervals, and $x_i \in \cA(I_i)$.
Choose discs $D_{in} = D_0 \subset D_1 \subset \ldots \subset D_{n+1} = D_{out}$
such that $I_i \subset \partial D_i$, as well as parametrizations $\varphi_i:S^1 \to \partial D_i$ for $i\in\{1,\ldots,n\}$.
Equip $A_i := D_{i} \setminus \mathring{D}_{i-1}$ with the boundary parametrizations $\varphi_{i-1}$ and $\varphi_{i}$ (where by convention $\varphi_0=\varphi_{in}$ and $\varphi_{n+1}=\varphi_{out}$), and choose $z$-lifts $\underline{A}_i \in \tAnn_c$ of $(A_i,\varphi_{i-1}, \varphi_i)$ satisfying $\prod \underline{A}_j = \underline{A}$.
For every representation $K$ of the conformal net $\cA$, we then have an operator 
\begin{equation}\label{eq: def operator X}
X_{(D_j,\varphi_j)}
:=
\underline{A}_{n+1}\varphi_n^*(x_n) \cdots \underline{A}_3 \varphi_2^*(x_2) \underline{A}_2 \varphi_1^*(x_1) \underline{A}_1  : K \to K,
\end{equation}
as in \eqref{eq: A[x,...,x]_D - with parametrisations}.
We need to show that this operator is independent of the choices of ordering of the intervals,  intermediate discs, intermediate parametrizations $\varphi_j$, and lifts $\underline{A}_j$.

If $K=H_0$ is the vacuum sector, then by \eqref{eq: parametrized vs unparametrized annuli} and our assumption $(\ref{item 2 independence : annuli})$ we have
\begin{equation}\label{eq: visibly independent of the choices}
X_{(D_j,\varphi_j)} = U(\varphi_{out})^*A[x_1 \ldots x_n] U(\varphi_{in}),
\end{equation}
which is visibly independent of the choices.

Returning to the case of an arbitrary sector $K$, we start by showing that $X_{(D_j,\varphi_j)}$ does not depend on the choice of intermediate parametrizations $\varphi_j$.
Given new such parametrizations $\varphi_j':S^1\to D_j$, let $\alpha_j:= \varphi^{-1}_j \circ\varphi'_j\in \Diff(S^1)$. 
Pick $z$-lifts $\underline{A}_j^\prime$, $\underline{\alpha}_j\in\tAnn_c$ of $(A_j,\varphi'_{j-1},\varphi'_j)$ and $\alpha_j$
subject to the condition
$
\underline{A}_j' = \underline{\alpha}_j^{-1} \underline{A}_j\, \underline{\alpha}_{j-1}$
(by convention, $\underline{\alpha}_0$ and $\underline{\alpha}_{n+1}$ are taken to be the identity).
We then have $\prod \underline{A}'_j = \underline{A}$, and we show by direct computation
that \eqref{eq: def operator X} does not depend on the intermediate parametrizations or the individual lifts $\underline{A}_j$:
\begin{align*}
X_{(D_j,\varphi_j')} &= 
 \underline{A}_{n+1} \underline{\alpha}_{n} \varphi_n'^*(x_n) \underline{\alpha}_n^{-1} \underline{A}_{n}\, \underline{\alpha}_{n-1} \cdots  \underline{\alpha}_1\varphi_1'^*(x_1)\underline{\alpha}_1^{-1}\underline{A}_1\\
&= 
\underline{A}_{n+1} \varphi_n^*(x_n) \underline{A}_{n}  \cdots \varphi_1^*(x_1)\underline{A}_1\\
&= 
X_{(D_j,\varphi_j)}.
\end{align*}
In light of the above, from now on, we denote this operator simply by $X_{(D_j)}$.

We now show that $X_{(D_j)}$ does not depend on the choices of discs $D_j$.
We first work under the additional assumption that $\underline A$ is localised in
$(\varphi_{in}^{-1}J)'$, i.e., that
$\underline A \in \Ann_c((\varphi_{in}^{-1}J)')\subset \tAnn_c$
(which in particular implies $\varphi_{in}|_{\varphi_{in}^{-1}J}=\varphi_{out}|_{\varphi_{in}^{-1}J}$).
Without loss of generality, we may insist that the intermediate parametrizations $\varphi_j$ are chosen so that $\varphi_{in}=\varphi_0,\varphi_1, \ldots, \varphi_{n+1}=\varphi_{out}$ all agree on $\varphi_{in}^{-1}J$, and that the $z$-lifts  
$\underline A_j$ are localised in
$(\varphi_{in}^{-1}J)'$.
Write $\pi:\cA((\varphi_{in}^{-1}J)') \to B(K)$ for the action of the conformal net, and let $X^0$ be the operator on $H_0$ defined by the same formula \eqref{eq: A[x,...,x]_D - with parametrisations}.
We showed in \eqref{eq: visibly independent of the choices} that $X^0$ does not depend on the choices of discs. 
By Proposition~\ref{lem: local annuli lie in local algebra} and the fact that $\varphi_j^*(x_j) \in \cA((\varphi_{in}^{-1}J)')$, we then have
\[
\pi(X^0) = X_{(D_j)},
\]
from which we can infer that $X_{(D_j)}$ also doesn't depend on the choice of discs $D_j$.

Finally, we drop the assumption that $\underline A$ is localised in some interval.
Since $A$ is thin, we may pick a lifted diffeomorphism $\underline{\beta} \in \tDiff_c(S^1)$ such that $\underline{\beta} \underline{A}$ is localised in
$(\varphi_{in}^{-1}J)'$.
Let $X_{\underline \beta\underline{A}}$ be the operator \eqref{eq: def operator X} corresponding to $\underline{\beta} \underline{A}$.
By the results of the previous paragraph,
$X_{\underline \beta \underline{A}}$
does not depend on choice of intermediate discs.
It follows that
$X_{(D_j)} = \underline{\beta}^{-1} X_{\underline{\beta} \underline{A}}$ also
doesn't depend on choice of intermediate discs.

\noindent\emph{Step 4.}
In this step, we show that if $(\ref{item 2 independence : annuli})$ and $(\ref{item 3 independence : annuli})$ hold for thin annuli with $n$ insertions, then they hold for all annuli with that same number $n$ of insertions.
We treat in parallel the following two cases which correspond, respectively, to $(\ref{item 2 independence : annuli})$ and $(\ref{item 3 independence : annuli})$:\\
\emph{Case (I):} annuli with unparametrised boundary (acting on vacuum sectors only), and\\
\emph{Case (II):} $z$-lifted annuli with parametrised boundary, acting an arbitrary sectors.

Let $A=D_{out}\setminus D_{in}$ be an annulus with $n$ insertions along intervals $I_1,\ldots, I_n$. [In case (II), we further equip $A$ with boundary parametrisations, and a lift to the central extension $\tAnn_c$.]
Our first task is to produce an operator $A[x_1 x_2 \ldots x_n]^\sim$, not necessarily equal to $A[x_1 x_2 \ldots x_n]_{(D_j)}$, but independent of any choices. Later we will show that these two operators are in fact equal to each other.

Consider the quotient space $A/{\sim}$ obtained by crushing each of the intervals $I_i$ to a single point $p_i\in A/{\sim}$. The annulus $A/{\sim}$ is a priori not a smooth manifold (it's certainly not a complex manifold), but we may pick a smooth structure extending the one given on $A\setminus\bigcup I_i$. Given a subspace $X\subset A/{\sim}$, we write $\tilde X$ for its preimage in $A$.

Let $\cD$ denote the space of possible decompositions $A/{\sim} = B \cup C$ into two annuli with thin parts.
[In case (II), both $\tilde B$ and $\tilde C$ are equipped with boundary parametrisations, and lifts to $\tAnn_c$, composing to the given lift of $A$.]
The space $\cD$ %is connected. It 
contains a subspace $\mathring\cD\subset \cD$ of decompositions such that the insertion points are not on the boundary of $B$ or $C$.
For $(B,C)\in \mathring\cD$, the annuli $\tilde B$ and $\tilde C$ have smooth boundary, and we set 
\begin{equation}\label{eq: A[xxx]=B[...] circ C[...]}
A[x_1 x_2 \ldots x_n]^\sim := \tilde B[\ldots] \tilde C[\ldots]
\end{equation}
where the unspecified dots refer to the insertions that lie inside $B$ and inside $C$, respectively.

Let us define $(B,C)$ and $(B',C')$ in $\mathring\cD$ to be \emph{close enough}, if there exists $(B'',C'')\in \mathring\cD$ such that $B''\subset B\cap B'$ and $C''\supset C\cup C'$.
In that case, we have
\begin{align}
\tilde B[\ldots]  \tilde C[\ldots]
&=\tilde B''[\ldots]  (\tilde B\setminus \tilde B'')[\ldots]  \tilde C[\ldots]
\notag
\\&=\tilde B''[\ldots]  \tilde C''[\ldots]
\label{eq: lots of B and C}
\\
\notag
&=\tilde B''[\ldots]  (\tilde B'\setminus \tilde B'')[\ldots]  \tilde C'[\ldots]
=\tilde B'[\ldots]  \tilde C'[\ldots]
\end{align}
As a consequence of the existence of framings (\cite[Corollary 4.17]{HenriquesTener24ax}),
the space $\cD$ is path connected. So any two decompositions $A/{\sim} = B_0 \cup C_0$ and $A/{\sim} = B_1 \cup C_1$ can be connected by a smooth path
\[
A/{\sim} = B_t \cup C_t\qquad t\in[0,1],
\]
which we may furthermore assume to be transverse to the insertion points $p_i$.
Pick $0=t_0<t_1<\ldots<t_m=1$ such that $(B_{t_i},C_{t_i})$ and $(B_{t_{i+1}},C_{t_{i+1}})$ are close enough,
and such that $(B_{t_i},C_{t_i})\in\mathring\cD$.
It follows from \eqref{eq: lots of B and C} that 
\[
\tilde B_0[\ldots]  \tilde C_0[\ldots]=\tilde B_1[\ldots]  \tilde C_1[\ldots].
\]
This finishes the proof that
\eqref{eq: A[xxx]=B[...] circ C[...]} is independent of the choice of point in $\mathring \cD$.

It remains to show that \eqref{eq: A[xxx]=B[...] circ C[...]} is equal to $A[x_1 x_2 \ldots x_n]_{(D_j)}$. This is obvious in the case $n=1$ of a single insertion.
So it's enough to prove that for every decomposition 
$A = A_1 \cup A_2$ into two annuli (not necessarily thin), compatible with the insertions, we have
\[
A[x_1 x_2 \ldots x_n]^\sim = A_1[\ldots]^\sim  A_2[\ldots]^\sim.
\]
Pick thin annuli $B_1, B_2, B'_2, C_1, C'_1, C_2 \subset A$ satisfying
$A_1 = B_1 C_1$,
$A_2 = B_2 \cup C_2$,
$\mathring B_2 = \mathring B'_2$,
$\mathring C_1 = \mathring C'_1$,
and such that $C_1\cup B_2=B'_2\cup C'_1$ is also thin.
[In case (I), we also require $B_2 = B'_2$ and $C_1 = C'_1$ as elements of $\tAnn_c$.]
Here's a typical example of what these annuli might look like:
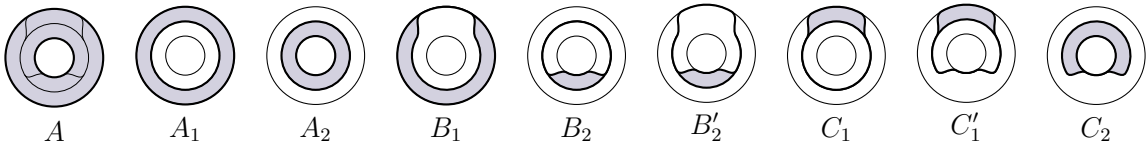
\[
\begin{tikzpicture}[scale=.65]
\filldraw[thick, fill=red!10!blue!20!gray!30!white] (0,0) circle (1);
\filldraw[thick, fill=white] (0,0) circle (.4);
%\draw (0,0) circle (1);
\draw[very thin] (0,0) circle (.7);
%\draw (0,0) circle (.4);
\draw[very thin] (-60:.4) to[out=30,in=-120] (-30:.7);
\draw[very thin] (-120:.4) to[out=150,in=-60] (-150:.7);
\draw[very thin] (60:1) to[out=-30,in=120] (30:.7);
\draw[very thin] (120:1) to[out=-150,in=60] (150:.7);
\node[scale=.9] at (0,-1.5) {$A$};
\end{tikzpicture}
\quad
\begin{tikzpicture}[scale=.65]
\filldraw[thick, fill=red!10!blue!20!gray!30!white] (0,0) circle (1);
\filldraw[thick, fill=white] (0,0) circle (.7);
\draw (0,0) circle (1);
\draw (0,0) circle (.7);
\draw (0,0) circle (.4);
\node[scale=.9] at (0,-1.5) {$A_1$};
\end{tikzpicture}
\quad
\begin{tikzpicture}[scale=.65]
\filldraw[thick, fill=red!10!blue!20!gray!30!white] (0,0) circle (.7);
\filldraw[thick, fill=white] (0,0) circle (.4);
\draw (0,0) circle (1);
\draw (0,0) circle (.7);
\draw (0,0) circle (.4);
\node[scale=.9] at (0,-1.5) {$A_2$};
\end{tikzpicture}
\quad
\begin{tikzpicture}[scale=.65]
\filldraw[thick, fill=red!10!blue!20!gray!30!white] (0,0) circle (1);
\filldraw[thick, fill=white]
(60:1) to[out=-30,in=120] (30:.7)
arc(360+30:150:.7)
(150:.7) to[out=60,in=-150] (120:1)
arc(120:60:1);
\draw (0,0) circle (1);
%\draw (0,0) circle (.7);
\draw (0,0) circle (.4);
%\draw (-60:.4) to[out=30,in=-120] (-30:.7);
%\draw (-120:.4) to[out=150,in=-60] (-150:.7);
\node[scale=.9] at (0,-1.5) {$B_1$};
\end{tikzpicture}
\quad
\begin{tikzpicture}[scale=.65]
\filldraw[thick, fill=red!10!blue!20!gray!30!white] (0,0) circle (.7);
\filldraw[thick, fill=white]
(-60:.4) to[out=30,in=-120] (-30:.7)
arc(-30:360-150:.7)
(-150:.7) to[out=-60,in=150] (-120:.4)
arc(-120:-60:.4);
\draw (0,0) circle (1);
\draw (0,0) circle (.7);
\draw (0,0) circle (.4);
%\draw (60:1) to[out=-30,in=120] (30:.7);
%\draw (120:1) to[out=-150,in=60] (150:.7);
\node[scale=.9] at (0,-1.5) {$B_2$};
\end{tikzpicture}
\quad
\begin{tikzpicture}[scale=.65]
\filldraw[thick, fill=red!10!blue!20!gray!30!white] (0,0) circle (.7);
\filldraw[thick, fill=white]
(-60:.4) to[out=30,in=-120] (-30:.7)
arc(-30:30:.7)
to[out=120,in=-30] (60:1)
arc(60:120:1)
to[out=-150,in=60] (150:.7)
arc(150:360-150:.7)
(-150:.7) to[out=-60,in=150] (-120:.4)
arc(-120:-60:.4);
\draw (0,0) circle (1);
%\draw (0,0) circle (.7);
\draw (0,0) circle (.4);
%\draw (-60:.4) to[out=30,in=-120] (-30:.7);
%\draw (-120:.4) to[out=150,in=-60] (-150:.7);
\node[scale=.9] at (0,-1.5) {$B'_2$};
\end{tikzpicture}
\quad
\begin{tikzpicture}[scale=.65]
\filldraw[thick, fill=red!10!blue!20!gray!30!white] 
(30:.7)
to[out=120,in=-30] (60:1)
arc(60:120:1)
to[out=-150,in=60] (150:.7)
arc(150:360+30:.7);
\filldraw[thick, fill=white] (0,0) circle (.7);
\draw (0,0) circle (1);
\draw (0,0) circle (.7);
\draw (0,0) circle (.4);
%\draw (-60:.4) to[out=30,in=-120] (-30:.7);
%\draw (-120:.4) to[out=150,in=-60] (-150:.7);
%\draw (60:1) to[out=-30,in=120] (30:.7);
%\draw (120:1) to[out=-150,in=60] (150:.7);
\node[scale=.9] at (0,-1.5) {$C_1$};
\end{tikzpicture}
\quad
\begin{tikzpicture}[scale=.65]
\filldraw[thick, fill=red!10!blue!20!gray!30!white] 
(30:.7)
to[out=120,in=-30] (60:1)
arc(60:120:1)
to[out=-150,in=60] (150:.7)
arc(150:360-150:.7) to[out=-60,in=150] (360-120:.4) arc(360-120:360-60:.4)
(360-60:.4) to[out=30,in=-120] (360-30:.7)
arc(360-30:360+30:.7);
\filldraw[thick, fill=white] 
(-30:.7) arc(-30:360-150:.7) to[out=-60,in=150] (360-120:.4) arc(360-120:360-60:.4)
(360-60:.4) to[out=30,in=-120] (360-30:.7);
%\filldraw[thick, fill=white] (0,0) circle (.7);
\draw (0,0) circle (1);
%\draw (0,0) circle (.7);
\draw (0,0) circle (.4);
%\draw (60:1) to[out=-30,in=120] (30:.7);
%\draw (120:1) to[out=-150,in=60] (150:.7);
\node[scale=.9] at (0,-1.5) {$C'_1$};
\end{tikzpicture}
\quad
\begin{tikzpicture}[scale=.65]
\filldraw[thick, fill=red!10!blue!20!gray!30!white] 
(-30:.7) arc(-30:360-150:.7) to[out=-60,in=150] (360-120:.4) arc(360-120:360-60:.4)
(360-60:.4) to[out=30,in=-120] (360-30:.7);
\filldraw[thick, fill=white] (0,0) circle (.4);
\draw (0,0) circle (1);
%\draw (0,0) circle (.7);
\draw (0,0) circle (.4);
%\draw (-60:.4) to[out=30,in=-120] (-30:.7);
%\draw (-120:.4) to[out=150,in=-60] (-150:.7);
%\draw (60:1) to[out=-30,in=120] (30:.7);
%\draw (120:1) to[out=-150,in=60] (150:.7);
\node[scale=.9] at (0,-1.5) {$C_2$};
\end{tikzpicture}
\]
Letting $B := B_1 B'_2$ and
$C := C'_1 C_2$,
we then have $A=B\cup C$, and therefore:
\begin{align*}
A_1[\ldots]^\sim A_2[\ldots]^\sim
&= B_1[\ldots] C_1[\ldots] B_2[\ldots] \circ C_2[\ldots]
\\&= B_1[\ldots] B'_2[\ldots] C'_1[\ldots] C_2[\ldots]
= B[\ldots] C[\ldots]
= A[x_1 x_2 \ldots x_n]^\sim.
\end{align*}
This finishes the proof of Step 4.

To finish the proof of Proposition~\ref{proposition independence}, we must still treat the case when the intervals $I_i$ touch the boundary in more than one connected component:

$(\ref{item 1 independence : discs})$ Let
$| x_1 x_2 \ldots x_n\rangle_D^{(1)}$
and
$| x_1 x_2 \ldots x_n\rangle_D^{(2)}$
be two instances of \eqref{eq: <x,...,x>_D},
defined with different choices of intermediate discs and orderings of the insertions.
If $A$ is a thick annulus with $\partial_{in}A = \partial D$,
then
\[
 Y_A |x_1 \ldots x_n \rangle_D^{(1)} = |x_1 \cdots x_n \rangle_{A \cup D}^{(1)} = 
 |x_1 \cdots x_n \rangle_{A \cup D}^{(2)} =
  Y_A |x_1 \ldots x_n \rangle_D^{(2)},
\]
where the middle equality holds by the work we have done (as the intervals $I_i$ do not touch the boundary of $A \cup D$).
By Lemmas~\ref{lem: annuli have dense image} and \ref{lem: Aunderline in terms of A}, the operator $Y_A$ is injective.
It follows that  $|x_1 \ldots x_n \rangle_D^{(1)} =
|x_1 \ldots x_n \rangle_D^{(2)}$.

$(\ref{item 2 independence : annuli})$
Let $A = D_{out}\setminus \mathring D_{in}$, and 
let
$A[x_1 x_2 \ldots x_n]^{(1)}$
and 
$A[x_1 x_2 \ldots x_n]^{(2)}$
be two instances of \eqref{eq: A[x,...,x]_D},
defined with different choices of intermediate discs and orderings of the insertions.
Pick thick annuli $A_1$ with $\partial_{in}A_1 = \partial_{out} A$, and $A_2\subset D_{in}$ with 
$\partial_{out}A_2 = \partial_{in} A$.
We then have:
\[
 A_1 A[x_1 \ldots x_n]^{(1)}A_2 = (A_1 A A_2)[x_1 \cdots x_n ]^{(1)} = 
 (A_1 A A_2)[x_1 \cdots x_n ]^{(2)} =
  A_1 A[x_1 \ldots x_n]^{(2)} A_2,
\]
where the middle equality holds 
as the intervals $I_i$ do not touch the boundary of $A_1 \cup A \cup A_2$.
By Lemmas~\ref{lem: annuli have dense image} and \ref{lem: Aunderline in terms of A}, the operator $Y_{A_1}$ and $Y_{A_2}$ (denoted above simply $A_1$ and $A_2$, following our convention \eqref{eq: A:H_0(D_in) --> H_0(D_out)}) are injective with dense image.
It follows that  $
A[x_1 \ldots x_n]^{(1)} =
A[x_1 \ldots x_n]^{(2)}
$.

The last case, $(\ref{item 3 independence : annuli})$, is established like $(\ref{item 2 independence : annuli})$, by an analogous application of Lemma~\ref{lem: annuli have dense image}.
\end{proof}

We finish this section by collecting a couple of useful lemmas.

\begin{lem}\label{lem: parametrized vs unparametrized annuli -- with worms}
Let $A = D_{out} \setminus \mathring{D}_{in}$ be an annulus equipped 
with boundaries parametrizations $\varphi_{in/out}:S^1\to\partial D_{in/out}$, and let $\underline{A}\in\tAnn_c$ be a $z$-lift of $A$.
Pick unitaries $U(\varphi_{in/out}):H_0  \to H_0(D_{in/out})$ implementing $\varphi_{in/out}$, and $z \in \bbC^\times$, satisfying
\[
\underline{A} = z\cdot U(\varphi_{out})^*A U(\varphi_{in}) %: H_0(\bbD) \to H_0(\bbD)
% A = U(?out)*A U(?in) : H0(D2) ? H0(D2)
\]
(such unitaries and scalar exist by Lemma~\ref{lem: Aunderline in terms of A}).
Then the same relation also holds in the presence of worm insertions:
\[
\underline{A}[x_1 x_2 \ldots x_n] = z\cdot U(\varphi_{out})^* A[x_1 x_2 \ldots x_n] U(\varphi_{in}).
% A[x1 x2 ... xn] =  U(?out)*A[x1 x2 ... xn] U(?in) 
\]
\end{lem}

\begin{proof}
This relation was established in \eqref{eq: parametrized vs unparametrized annuli}, on out way towards proving that the operators $A[x_1 x_2 \ldots x_n]$ and $\underline A[x_1 x_2 \ldots x_n]$ are well-defined.
\end{proof}

The following is a variant of 
\eqref{eq: compatibility of discs and annuli}:

\begin{lem}\label{lem: univalent acting on worm insertions}
Let $f:\bbD \to \bbD$ be a univalent map, let $\underline{A}_f \in \Univ \subset \Ann_c$ be the corresponding lifted annulus, and
let $y_1, \ldots, y_m$ be worm insertions with support intervals in $A_f$.
%and let $\pi_0:\Univ \subset \Ann_c \to B(H_0)$ be the representation introduced in Section~\ref{sec: reps of semigroup of annuli}.
Then
\[
\underline{A}_f[y_1 \ldots y_m]\big(|x_1 \ldots x_n\rangle_\bbD\big) = | y_1 \ldots y_m\,  \cA(f)(x_1) \ldots \cA(f)(x_n) \rangle_\bbD.
\]
\end{lem}
\begin{proof}
Let us write
$
A_f[y_1 \ldots y_m]:H_0(f(\bbD))\to H_0(\bbD)
$ and
%\qquad\text{and}\qquad
$\underline A_f[y_1 \ldots y_m]:H_0(\bbD)\to H_0(\bbD)
$
for the operators defined in 
\eqref{eq: A[x,...,x]_D} and
\eqref{eq: A[x,...,x]_D - with parametrisations}, respectively (using the embedding $\Univ \hookrightarrow \Ann_c$ for the latter).
%(see Section~\ref{sec: Central extensions of the semigroup of annuli}).
If $U_f:H_0(\bbD)\to H_0(f(\bbD))$ is the unitary induced by $f$ then,
by Lemma~\ref{lem: Aunderline in terms of A}, the operators $A_f:H_0(f(\bbD))\to H_0(\bbD)$
and $\underline{A}_f:H_0(\bbD)\to H_0(\bbD)$
satisfy
\begin{equation}\label{A_f = Y_{A_f}U_f}
\underline{A}_f = A_fU_f
\end{equation}
up to scalar. And as both sides fix the vacuum vector $\Omega \in H_0(\bbD)$, the identity \eqref{A_f = Y_{A_f}U_f} in fact holds on the nose.
It follows by Lemma~\ref{lem: parametrized vs unparametrized annuli -- with worms} that
\[
\underline{A}_f[y_1 \ldots y_m] = A_f[y_1 \ldots y_m]U_f.
\]
Since \eqref{eq: <x,...,x>_D} is functorial with respect to isomorphisms, we have
\[
U_f| x_1 \ldots x_n \rangle_\bbD = | \cA(f)(x_1) \ldots \cA(f)(x_n) \rangle_{f(\bbD)}.
\]
Hence, by \eqref{eq: compatibility of discs and annuli}, we have
\begin{align*}
\underline{A}_f[y_1 \ldots y_m] | x_1 \ldots x_n \rangle_\bbD
&=
A_f[y_1 \ldots y_m]U_f | x_1 \ldots x_n \rangle_\bbD\\
&=
A_f[y_1 \ldots y_m] \, | \cA(f)(x_1) \ldots \cA(f)(x_n) \rangle_{f(\bbD)}\\
&=
|y_1 \ldots y_m \, \cA(f)(x_1) \ldots \cA(f)(x_n) \rangle_{\bbD}. \qedhere
\end{align*}
\end{proof}

Our next result is an immediate consequence of Lemma~\ref{lem: thin annuli intertwine action of boundary intervals}:

\begin{lem}\label{lem: thin part A(I) module map}
Let $A=D_{out} \setminus \mathring D_{in}$ be a thin annulus, and let $I\subset \partial_{in}A \cap \partial_{out}A$ be an interval in the thin part of $A$.
Then if the support intervals of the worms are disjoint from $I$, the operator
\[
A[x_1 x_2 \ldots x_n]: H_0(D_{in}) \to H_0(D_{out})
\]
is equivariant for the actions of $\cA(I)$ on $H_0(D_{in})$ and on $H_0(D_{out})$.
\hfill $\square$
\end{lem}

The reader might have noticed that when $A$ is a thin annulus, we have used the notation $\underline{A}[x_1 x_2 \ldots x_n]$ to refer to two potentially different things.
The following lemma shows that there is in fact no problem:

\begin{lem}\label{lem: A[worms] in B(K)}
Let $\underline{A}\in\tAnn_c$ be localised in an interval $I$, and suppose that the support intervals of the worms are disjoint from the image of $I'$ in $A$. Then the operator
\begin{equation}\label{eq: A[worms] in A(I)}
\underline{A}[x_1 x_2 \ldots x_n]:H_0\to H_0
\end{equation}
lies in $\cA(I)\subset B(H_0)$.

Moreover, if $(K,\rho)$ is a representation of the conformal net, then the operator
\begin{equation}\label{eq: A[worms] in A(I) rep}
\underline{A}[x_1 x_2 \ldots x_n]:K\to K
\end{equation}
(defined in \eqref{eq: A[x,...,x]_D - with parametrisations}) is the image of \eqref{eq: A[worms] in A(I)} under the homomorphism $\rho_I:\cA(I)\to B(K)$.
\end{lem}

\begin{proof}
The first statement follows from \eqref{eqn: annulus with worm insertions is local}.
The second statement is immediate from Proposition~\ref{lem: local annuli lie in local algebra}, and the respective definitions of \eqref{eq: A[worms] in A(I)} and \eqref{eq: A[worms] in A(I) rep}.
\end{proof}

We now describe the adjoint of the operator $\underline{A}[x_1 x_2 \cdots x_n]:K\to K$.

\begin{lem}\label{lem: adjoint of the operator underline A}
Let $\underline{A} \in \tAnn_c$ and let $x_1, \ldots, x_n$ be a collection of worm insertions.
Then
\[
\underline{A}[x_1 x_2 \cdots x_n]^*
= \underline{A}^\dagger[x_1^* x_2^* \cdots x_n^*].
% A[x1 x2 ... xn]* = A?[x1* x2* ... xn*].
\]
\end{lem}
\begin{proof}
Decompose $\underline{A} = \underline{A}_0 \underline{A}_1 \cdots \underline{A}_n$ with the interval labelled by $x_j$ on the outgoing boundary of $A_j$.
We then have
\[
\underline{A}[x_1 x_2 \cdots x_n]^*
= \big(\underline{A}_0 x_1 \underline{A}_1 x_2 \cdots \underline{A}_n \big)^*
= \underline{A}_n^\dagger \cdots x_2^* \underline{A}_1^\dagger  x_1^* \underline{A}_0^\dagger 
= \underline{A}^\dagger[x_1^* x_2^* \cdots x_n^*].
\]
\end{proof}

\subsection{Holomorphicity of worm insertions}

Let $I_1,\ldots,I_n$ be abstract intervals (not embedded in any disc or annulus), and let $x_i\in\cA(I_i)$ be algebra elements.

\begin{lem}\label{lem: holomorphicity of worm insertions in annuli}
Let $A:=D_{out} \setminus \mathring D_{in}$ be an annulus,
let $M$ be a finite dimensional complex manifold, and let $\alpha_i:M\times I_i\to \mathring A$ be smooth maps which are pointwise holomorphic in the sense that $\forall q\in I_i$
\[
\alpha_i|_{M \times \{q\}}: M \to \mathring A
\]
is holomorphic.
Let us further assume that for every point $m\in M$, the the maps $\alpha_i|_{\{m\}\times I_i}$ are embeddings with disjoint images.
For $m\in M$, let $I_i(m):=\alpha_i(\{m\}\times I_i)$, and let
\[
x_i(m):=(\alpha_i|_{\{m\}\times I_i})_*(x_i)\in \cA(I_i(m)).
\]
Then the map 
\begin{equation}\label{ref: hol of worms -- UNparametrised}
M \to \cB\big(H_0(D_{in}), H_0(D_{out})\big): m \mapsto A[x_1(m) \ldots x_n(m)]
% M ?  B(H0(Din), H0(Dout)) : m ? A[a1(m) ... an(m)]
\end{equation}
is holomorphic.

Similarly, if we have a $z$-lift $\underline{A} \in \tAnn_c$ of $A$, and a representation $K$ of the conformal net $\cA$, then the map
\begin{equation}\label{ref: hol of worms -- parametrised}
M \to \cB(K): m \mapsto \underline{A}[a_1(m) \ldots a_n(m)]
\end{equation}
is holomorphic.
\end{lem}
\begin{proof}
By Lemma~\ref{lem: parametrized vs unparametrized annuli -- with worms}, we have $\underline{A}[x_1(m) \ldots x_n(m)]=z \, U(\varphi_{out})^*A[x_1(m) \ldots x_n(m)]U(\varphi_{in})$ for suitable unitaries $U(\varphi_{in})$ and $U(\varphi_{out})$, and scalar $z \in \bbC^\times$ which does not depend on $m$.
So \eqref{ref: hol of worms -- UNparametrised} follows from \eqref{ref: hol of worms -- parametrised}, and
we only need to consider the parametrised case.

The result being local in $M$, we may assume without loss of generality that $I_i(m)\cap I_j(m')=\emptyset$ for every $m,m'\in M$ and every $i\not = j$.
Choose a sequence of discs 
\[
D_{in} = D_0 \subset D_1 \subset \cdots \subset D_n = D_{out}
\]
such that $I_i(m) \subset \mathring D_i \setminus D_{i-1}\,\, \forall m\in M$.
Let $A_i = D_i \setminus \mathring D_{i-1}$, 
and let $\underline{A}_i \in \tAnn_c$ be compatible $z$-lifts satisfying $\underline{A}=\underline{A}_1\ldots \underline{A}_n$. Then we have:
\[
\underline{A}[x_1(m) \ldots x_n(m)] = \underline{A}_1[x_1(m)] \ldots \underline{A}_n[x_n(m)].
\]
We are thus reduced to proving the result for a single worm in an annulus.
Write $x$ for $x_1$, and $I$ for $I_1$.
Let us also fix an embedding $A\hookrightarrow \bbC$.

We wish to establish the holomorphicity of \eqref{ref: hol of worms -- parametrised} around some point $m_0\in M$.
We assume without loss of generality that $M\subset \bbC^N$, and that $I \subset S^1$.
As a first step, we wish to extend $\alpha : M \times I \to A$ to a pointwise-holomorphic map $\hat\alpha : M\times S^1 \to A$ with the property that each restriction $\hat\alpha|_{\{m\}\times S^1}: S^1 \to A$ is an embedding with winding number $1$ (it represents the standard generator of $\pi_1(A)$).
For that, first pick $\hat\alpha|_{\{m_0\}\times S^1}$ extending $\alpha|_{\{m_0\}\times S^1}$.
For every point $q\in I$, the map $\hat\alpha|_{M\times \{q\}}$ is determined by the coefficients of its Taylor series (which are smooth functions on $I$). We extend these coefficients arbitrarily to smooth functions on $S^1$, and declare them to be the Taylor coefficient of some map $M\times S^1 \to A$. Provided these smooth extensions have their sup-norm controlled by their values at $\partial I$ (e.g. they are bounded by twice those values), and that they vanish identically on the subset of $S^1$ which maps to $\partial A$, the corresponding Taylor series will have positive radius of convergence and will determine, locally, a family of embeddings $S^1 \to A$ parametrised by the points of $M$.

This family of embeddings $\hat\alpha|_{\{m\}\times S^1} : S^1 \to A$ provides a family  of decompositions $A = B(m) \cup C(m)$ in $\Ann$.
By \cite[Prop. 5.11]{HenriquesTener24ax}, we may (locally) choose $z$-lifts $\underline{B}(m)$, $\underline{C}(m)\in \tAnn_c$, which depend holomorphically on $m$.
These $z$-lifts may furthermore be chosen so that $\underline{A} = \underline{B}(m) \underline{C}(m)$.
The corresponding operators $\underline{B}(m)$ and $\underline{C}(m)$ in $\cB(K)$ depend holomorphically on $m$ by \cite[Thm. 6.4]{HenriquesTenerIntegratingax}.
It follows that $\underline{A}[x(m)] = \underline{B}(m) x \underline{C}(m)$ also depends holomorphically on $m$.
\end{proof}

\begin{cor}\label{cor: holomorphicity of worm insertions}
Let $D$ be a disc.
Let $M$ be a finite dimensional complex manifold, and let $\alpha_i:M\times I_i\to \mathring{D}$ be smooth maps which are pointwise holomorphic in the sense that $\forall q\in I_i$
\[
\alpha_i|_{M\times\{q\}}: M \to \mathring D
\]
is holomorphic. Let us further assume that for every point $m\in M$, the maps $\alpha_i|_{\{m\}\times I_i}$ are embeddings with disjoint images.
For $m\in M$, let $I_i(m)$ and $x_i(m)$ be as in Lemma~\ref{lem: holomorphicity of worm insertions in annuli}.
Then the map
\[
M \to H_0(D): m\mapsto | x_1(m) x_2(m) \ldots x_n(m) \rangle_D
\]
is holomorphic.
\end{cor}

\begin{proof}
The result being local in $M$, we may assume without loss of generality that there exists a disc $D_{in} \subset D$ which is disjoint from the images of all the $\alpha_i$.
Let $A:= D\setminus \mathring D_{in}$.
The result then follows from Lemma \ref{lem: holomorphicity of worm insertions in annuli}, as
\[
| x_1(m) x_2(m) \ldots x_n(m) \rangle_D = A[x_1(m) x_2(m) \ldots x_n(m)] \Omega_{D_{in}}.\qedhere
\]
\end{proof}

The same arguments used to establish holomorphicity of worm insertions can also be used to show continuity all the way up to the boundary:

\begin{lem}\label{lem: continuity of worm insertions}
Let $M$ be a topological space.
\begin{enumerate}[i)]
\item Let $D$ be a disc, and let $\alpha_i:M\times I_i\to D$ be a continuous family of embeddings of $I_i$ into $D$ (the corresponding map $M \to C^\infty(I_i,D)$ is continuous). Let us further assume that for every point $m\in M$, the maps $\alpha_i|_{\{m\}\times I_i}$ have disjoint images.
For $m\in M$, let $I_i(m):=\alpha_i(\{m\}\times I_i)$ and
$x_i(m):=(\alpha_i|_{\{m\}\times I_i})_*(x_i)\in \cA(I_i(m))$.
Then the map
\begin{equation}\label{eq: cont fam of worms}
M \to H_0(D): m\mapsto | x_1(m) x_2(m) \ldots x_n(m) \rangle_D
\end{equation}
is continuous.

\item Let $A:=D_{out} \setminus \mathring D_{in}$ be an annulus, and let $\alpha_i:M\times I_i\to A$ be a continuous family of smooth embeddings of $I_i$ into $A$.
Assume that for every point $m\in M$, the the maps $\alpha_i|_{\{m\}\times I_i}$ are embeddings with disjoint images, and
let $I_i(m)$ and $x_i(m)$ be as above.
Then the map 
\begin{equation*}
M \to \cB\big(H_0(D_{in}), H_0(D_{out})\big): m \mapsto A[x_1(m) \ldots x_n(m)]
\end{equation*}
is continuous for the strong operator topology on $\cB(H_0(D_{in}), H_0(D_{out}))$.

\item Similarly, given a $z$-lift $\underline{A} \in \tAnn_c$ of $A$, and a representation $K$ of the conformal net
\begin{equation*}
M \to \cB(K): m \mapsto \underline{A}[a_1(m) \ldots a_n(m)]
\end{equation*}
is continuous for the strong operator topology on $\cB(K)$.
\end{enumerate}
\end{lem}

\begin{proof}
The proof is identical to those of Lemma~\ref{lem: holomorphicity of worm insertions in annuli} and Corollary~\ref{cor: holomorphicity of worm insertions},
with 
\cite[Prop. 5.5]{HenriquesTener24ax}
in place of \cite[Prop. 5.11]{HenriquesTener24ax},
and the continuity statement in \cite[Thm. 6.4]{HenriquesTenerIntegratingax}
in place of the holomorphicity statement in \cite[Thm. 6.4]{HenriquesTenerIntegratingax}.
\end{proof}

\subsection{Unoriented worms}

Recall that $\bbD$ denotes the standard unit disc. Let $I_-\subset \partial \bbD$ be the lower semi-circle, let $I_+:=I_-'\subset \partial \bbD$ be the upper semi-circle (so that $\partial I_-=\partial I_+=\{\pm1\}$),
and
let $\delta:i\mathbb R\to \Aut(\bbD)$ be the unique action which fixes $\{\pm 1\}$, normalised so that the derivative at $1$ of $\delta_{it}$ is $e^t$.
Recall from Section \ref{sec: unparametrised annuli} that the Bisognano-Wichmann theorem for conformal nets provides a canonical identification  of $H_0 = H_0(\bbD)$ with the $L^2$-space of the von Neumann algebra $\cA(I_-)$. Under that identification, the rescaled modular flow $\Delta^{it/2\pi} : L^2\cA(I_-) \to L^2\cA(I_-)$ agrees with the induced action of $\delta_{it}:\bbD \to \bbD$ on $H_0(\bbD)$.

The homomorphism $i\bbR\to \Mob:it\mapsto \delta_{it}$ can be analytically continued to a homomorphism $\bbC\to \mathrm{Aut}(\bbC\bbP^1):z\mapsto \delta_z$,
which again fixes the points $\pm1$.
For example, the map $\delta_{\pi}$ exchanges $I_-$ and $I_+$:
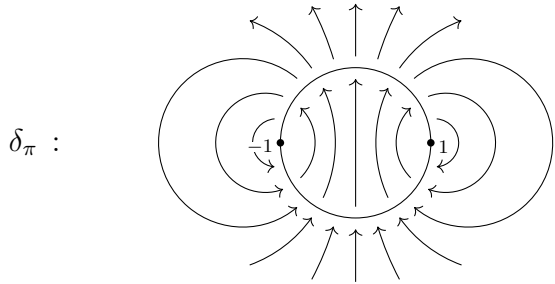
\[
\delta_{\pi}\,:\hspace{-.2cm}\tikz[baseline=-3]{
\draw circle (1);
\draw[->] (0,-.85) -- (0,.85);
\draw[->] (0,1.15) -- (0,1.85);
\draw[<-] (0,-1.15) -- (0,-1.85);
\fill (1,0) circle (.05) (-1,0) circle (.05);
\draw[->] (4,0) + (163.3:3.87) arc (163.3:152:3.87);
\draw[->] (4,0) + (-152:3.87) arc (-152:-163.3:3.87);
\draw[->] (2,0) + (145:1.73) arc (145:110:1.73);
\draw[->] (2,0) + (-110:1.73) arc (-110:-145:1.73);
\draw[->] (1.5,0) + (130.2:1.12) arc (130.2:-130.2:1.12);
\draw[->] (1.2,0) + (111:.66) arc (111:-111:.66);
\draw[->] (1.05,0) + (85:.32) arc (85:-85:.32);
\draw[<-] (2,0) + (155:1.73) arc (155:360-155:1.73);
\draw[<-] (1.2,0) + (135:.66) arc (135:360-135:.66);
\draw[->] (-4,0) + (180-163.3:3.87) arc (180-163.3:180-152:3.87);
\draw[->] (-4,0) + (180+152:3.87) arc (180+152:180+163.3:3.87);
\draw[->] (-2,0) + (180-145:1.73) arc (180-145:180-110:1.73);
\draw[->] (-2,0) + (180+110:1.73) arc (180+110:180+145:1.73);
\draw[->] (-1.5,0) + (180-130.2:1.12) arc (180-130.2:180+130.2:1.12);
\draw[->] (-1.2,0) + (180-111:.66) arc (180-111:180+111:.66);
\draw[->] (-1.05,0) + (180-85:.32) arc (180-85:180+85:.32);
\draw[<-] (-2,0) + (180-155:1.73) arc (180-155:180-360+155:1.73);
\draw[<-] (-1.2,0) + (180-135:.66) arc (180-135:180-360+135:.66);
\node[scale=.9] at (1.17,-.05) {$\scriptstyle1$};
\node[scale=.9, fill=white, inner sep=-1] at (-1.27,-.05) {$\scriptstyle-1$};
}
\]
As explained in Section \ref{sec: unparametrised annuli}, the modular conjugation $J:L^2\cA(I_-) \to L^2\cA(I_-)$ corresponds to an operator $\Theta: H_0 \to H_0$ which satisfies $\Theta(\cA(I))\Theta=\cA(\vartheta I)$, where $\vartheta(z)=\bar z$. Let us define
\[
\theta: \cA(I) \to \cA(\vartheta(I)), \qquad \theta(x) := \Theta x \Theta.
\]
As explained in \cite[\S2A]{BartelsDouglasHenriques15},
the map $\theta$ is, up to a $\ast$, a special case of the map $\cA(f):\cA(I)\to \cA(J)$
associated to an orientation reversing diffeomorphism $f:I\to J$:
\[
\theta: \cA(I) \to \cA(\vartheta I)
\qquad
\theta(x) = \cA(\vartheta)(x^*).
% ?(a) = A(?)(a*)
\]

If $I$ is an interval, let us write $\overline I$ for that same interval equipped with the opposite orientation, and let $j_I:I\to \overline I$ denote the identity map viewed as an orientation reversing map from $I$ to $\overline I$.
For $y\in \cA(I)$, let us then write $j(y):=\cA(j_I)(y)\in \cA(\overline I)$ for its image under the $\bbC$-linear anti-homomorphism $\cA(j_I):\cA(I)\to \cA(\overline I)$.

\begin{lem} \label{lem: <b> = < j(b*) >}
Let $y\in \cA(I)$ and $j(y)\in \cA(\overline I)$ be as above.
\begin{enumerate}[i)]
\item
 Let $D$ be a disc, let $x_1, \ldots, x_n, y$ be worms in $D$, and let us assume that the support interval $I$ of $y$ is contained in the interior of $D$. Then:
\begin{equation}\label{eq: b=j(b*)}
| x_1\ldots x_n y \rangle_{D} = | x_1 \ldots x_n j(y)\rangle_{D}.
\end{equation}
\item Let $A=D_{out} \setminus \mathring{D}_{in}$ be an annulus, let $x_1, \ldots, x_n$ be worms in $A$, and let $y$ be a worm with support interval $I\subset \mathring{A}$.
Then:
\begin{equation}\label{eq: unparametrised annuli: b=j(b*)}
A[x_1 \ldots x_n y] = A[x_1 \ldots x_n j(y)].
\end{equation}
\item
If $\underline{A} \in \tAnn_c$ is a $z$-lift of $A$, and $y$ is as above, then, in any representation $K$ of the conformal net, we have:
\begin{equation}\label{eq: annuli: b=j(b*)}
\underline{A}[x_1 \ldots x_n y] = \underline{A}[x_1 \ldots x_n j(y)].
\end{equation}
\end{enumerate}
\end{lem}

\begin{proof}
{\it i)}
Since both sides of \eqref{eq: b=j(b*)} depend continuously on the support intervals of the worms (Lemma~\ref{lem: continuity of worm insertions}), we may assume without loss of generality that each of the worms $x_i$ and $y$ are supported in the interior of our disc $D$, and that their support intervals are analytically embedded.
In this case, we may write $D = A \cup D'$ where $A$ is an annulus containing the worms $x_j$, and $D'$ is a disc containing the support interval of $y$ in its interior.
We then have:
\[
| x_1\ldots x_n y \rangle_{D} = A[x_1 \ldots x_n]|y\rangle_{D'} \quad \text{and} \quad | x_1 \ldots x_n j(y)\rangle_{D} = B[x_1 \ldots x_n] |j(y)\rangle_{D'}.
\]
Thus, by replacing $D$ by $D'$, it is enough to treat the case of a single insertion (i.e. $n=0$).
By choosing $D'$ small enough, we may furthermore assume that there exists a holomorphic embedding $f:D'\to \bbD$ mapping the support interval $I$ of $y$ to some interval on the real axis.
Applying $Y_{\bbD\setminus f(D')}$ to the vectors $|y\rangle_{D'}$ and $|j(y)\rangle_{D'}$ and using the injectivity of $Y_{\bbD\setminus f(D')}$ (Lemmas~\ref{lem: annuli have dense image} and \ref{lem: Aunderline in terms of A}),
we are further reduced to the case $D=\bbD$, with a single worm insertion whose support interval $I$ lies on the real axis.

Let $I_-\subset \partial \bbD$ be the lower semicircle.
By the Bisognano-Wichman theorem,
for any element $x \in \cA(I_-)$, the map $i\bbR\to H_0$ given by $it\mapsto \Delta^{it/2\pi} x\Omega$
admits an extension  
\[
F:\{z\in\bbC\,|\,0\le \Re\mathrm{e}(z)\le \pi\} \to H_0
\]
which is continuous on the closed strip and holomorphic in the open strip.
It satisfies
\[
F(it) = \delta_{it}(x)\Omega
\,\,\,\quad \text{and} \,\,\,\quad   F(\pi+it) =  \Theta \delta_{it}(x^*)\Omega = \theta(\delta_{it}(x^*)) \Omega = \delta_{it}(\theta(x^*)) \Omega
\]
for $t \in \bbR$
(where $\delta_{it}(x)$ denotes the image of $x$ under the map $\cA(\delta_{it}):\cA(I_-)\to \cA(I_-)$).

%$\Delta^{it} = \delta_{2\pi t}$
Let us now set $x := \delta_{-\pi/2}(y)\in \cA(\delta_{-\pi/2}(I))\subset\cA(I_-)$, and
%$x\in A(I)$ for some interval $I$ contained in the interior of $I_-$.
consider the map
\[
F_1:\{z\in\bbC\,|\,0\le \Re\mathrm{e}(z) < \pi\} \to H_0
\]
given by
$F_1(z) = | \delta_{z}(x)\rangle_{\bbD}$.
Here, $\delta_{z}(x)$ denotes the image of the element $x$ under the map $\cA(\delta_{z}):\cA(\delta_{-\pi/2}(I))\to \cA(\delta_{-\pi/2+z}(I))$.
Importantly, note that when $0 < \Re\mathrm{e}(z) < \pi$ we have $\delta_{-\pi/2+z}(I)\subset \mathring{\bbD}$, hence $| \delta_{z}(x)\rangle_{\bbD}$ is well defined.
The map $F_1$ is continuous (Lemma~\ref{lem: continuity of worm insertions}), and holomorphic in the interior of the strip (Corollary~\ref{cor: holomorphicity of worm insertions}).
By analytic continuation, since $F$ and $F_1$ agree on $i\bbR$, they also agree in the interior of the strip.

Next, consider the map
\[
F_2:\{z\in\bbC\,|\,0< \Re\mathrm{e}(z) \le \pi\} \to H_0
\]
given by $
F_2(z) =
| \delta_{z}(j(x))\rangle_{\bbD}=
| \delta_{z-\pi}(\theta(x^*))\rangle_{\bbD}.
$
%where $j$ is the identity map on the support interval of $x$, viewed as an orientation reversing map.
The map $F_2$ is again continuous and holomorphic in the interior.
The functions $F_2$ and $F$ agree on $\pi+i\bbR$,
%Since $F_2(\pi + it) = F(\pi + it)$ for $t\in \bbR$, the functions 
%$F$ and $F_2$ 
therefore agree on the interior of the strip.
It follows that $F_1(z) = F_2(z)$ for $0 < \Re\mathrm{e}(z) < \pi$.
Specialising to $z=\pi/2$, we get
$|y\rangle_{\bbD} = F_1(\pi/2) = F_2(\pi/2) = | j(y) \rangle_{\bbD}$, as desired.

{\it ii)} Once again, both sides of \eqref{eq: annuli: b=j(b*)} depend continuously on the support intervals of the worm, and we may assume without loss of generality that the worms are supported in the interior of $A$.
Write $A = B \cup  C$, where $B$ is a thin annulus containing the support intervals of all the worms.
We then have
\[
B[x_1 \ldots x_n y]\Omega_{D_{in}} = |x_1 \ldots x_n y\rangle_{B\cup D_{in}} = |x_1 \ldots x_n j(y)\rangle_{B\cup D_{in}} = B[x_1 \ldots x_n j(y)]\Omega_{D_{in}},
\]
where the middle equality follows from part {\it i)}.
As $B$ is thin, it follows from the Reeh-Schlieder theorem that
$B[x_1 \ldots x_n y] = B[x_1 \ldots x_n j(y)]$.
The desired statement follows.

{\it iii)} We may again assume that the worms are located in the interior of $A$.
Write $\underline A$ as a product $\underline A=\underline B\hspace{.5mm} \underline C$ of $z$-lifted annuli, with the property that $\underline{B} \in \tAnn_c(I_0)$ for some interval $I_0$, and such that $B$ contains the support intervals of all of the worms.
We then have
\[
\underline{A}[x_1 \ldots x_n y] = \underline{B}[x_1 \ldots x_n y] \underline{C}
\quad
\text{and}
\quad
\underline{A}[x_1 \ldots x_n j(y)] = \underline{B}[x_1 \ldots x_n j(y)] \underline{C},
\]
and we are thus reduced to establishing the desired identity with $\underline B$ in place of $\underline A$.
%Since $\underline B$ is localised,
By Lemma~\ref{lem: A[worms] in B(K)},
it suffices to consider when $K=H_0$ is the vacuum sector, in which case the desired identity follows from part {\it ii)} by  Lemma~\ref{lem: parametrized vs unparametrized annuli -- with worms}.
\end{proof}

Recall that $\vartheta: \bbD\to \bbD : z \mapsto \overline{z}$ denotes complex conjugation, and that $\Theta:H_0(\bbD) \to H_0(\bbD)$ is the corresponding antiunitary operator. For an interval $I\subset \bbD$, let us write $\theta:\cA(I)\to \cA(\vartheta(I))$
for the map $\theta(x) %=\Theta x \Theta 
=\cA(\vartheta)(x^*)$,
where $\vartheta(I)$ is oriented in such a way that $\vartheta: I \to \vartheta(I)$ is orientation reversing.

\begin{lem}\label{lem: Theta applied to worm insertions}
Let $I_1, \ldots, I_n\subset \bbD$ be disjoint intervals, and let $x_j\in \cA(I_j)$.
Then
\begin{equation}\label{eq: Theta(aaa) = (th(a),th(a),th(a))}
\Theta (| x_1 \ldots x_n \rangle_{\bbD}) =
| \theta (x_1) \ldots \theta (x_n) \rangle_{\bbD} .
\end{equation}
\end{lem} 

\begin{proof}
By continuity (Lemma~\ref{lem: continuity of worm insertions}), we may assume without loss of generality that the support intervals $I_j$ are all in $\mathring \bbD$.
Pick a disc $D\subset \bbD$ such that the  intervals $I_j$ are contained in $\partial D$, and let $A=\bbD\setminus\mathring D$. By choosing a single interval $I\subset \partial D$ containing every $I_j$ and letting $x=x_1 \ldots x_n \in \cA(I)$, we are reduced to proving equation \eqref{eq: Theta(aaa) = (th(a),th(a),th(a))} in the case $n=1$:
\[
\Theta(| x\rangle_{\bbD}) = | \theta (x)\rangle_{\bbD}.
% ?( ? a ?\bbD ) =  ? ?(a) ?\bbD .
\]

Choose a Riemann map $\varphi:D \to \bbD$.
Parametrize the incoming boundary of $A$ by $\varphi^{-1}$, and equip the outgoing boundary $S^1=\partial \bbD$ with the identity parametrization, so that $A\in\Univ$.
Letting $\overline\varphi:=\vartheta\circ\varphi\circ\vartheta$,
we then have:
\begin{align*}
\Theta | x \rangle_{\bbD}
= \Theta A | \varphi(x) \rangle_{\bbD}
= \Theta A \varphi(x) \Omega
&= \overline{A} \Theta\varphi(x) \Omega
\\&= \overline{A} \theta(\varphi(x)) \Omega
= \overline{A} \overline{\varphi}(\theta(x)) \Omega
= \overline{A} | \overline{\varphi} (\theta(x)) \rangle_\bbD
= | \theta (x)\rangle_{\bbD},
\end{align*}
where the third equality holds by \eqref{eq: pi:Ann_c rtimes Z/2 -> B(H_0)} and \eqref{eq: the invol} (and we've written $\overline A$ for the action \eqref{eq: the invol} of $\bbZ/2\bbZ$ on $A\in\Univ$).
\end{proof}

Recall that if
$f:D_1 \to D_2$ is an embedding which is either holomorphic or anti-holomorphic, then the pushforward $f_*:H_0(D_1) \to H_0(D_2)$ was defined in \eqref{eqn: pushforward by (anti)holomorphic map}.

\begin{cor}\label{cor: anti-holomorphic map applied to worm insertion}
Let $D_1$ and $D_2$ be discs, let $x_i$ be worms in $D_1$, and let $f:D_1\to D_2$ be an embedding that is either holomorphic or anti-holomorphic.
Then
\begin{equation}\label{eq: f_* no 1}
f_* | x_1 \ldots x_n \rangle_{D_1} = | \cA(f)(x_1) \ldots \cA(f)(x_n) \rangle_{D_2}
\end{equation}
when $f$ is holomorphic, and
\begin{equation}\label{eq: f_* no 2}
f_* | x_1 \ldots x_n \rangle_{D_1} = | \cA(f)(x_1^*) \ldots \cA(f)(x_n^*) \rangle_{D_2}
\end{equation}
when $f$ is anti-holomorphic (in the latter case, if $I_i$ is the support interval of $x_i$, then the support interval $f(I_i)$ of $\cA(f)(x_i^*)$ is oriented in such a way that the map $f:I_i\to f(I_i)$ is orientation reversing).
\end{cor}

\begin{proof}
When $f$ is holomorphic, equation \eqref{eq: f_* no 1} is a special case of \eqref{eq: compatibility of discs and annuli}.\vspace{-1mm}
For $f:D_1\to D_2$ anti-holomorphic,
write it as a composite 
$D_1\to \bbD\stackrel{\vartheta}\to \bbD \to D_2$
where the first and third maps are holomorphic, and apply 
Lemma~\ref{lem: Theta applied to worm insertions} (which is the special case of \eqref{eq: f_* no 2} for $f=\vartheta$) to the middle map.
\end{proof}

\begin{cor}\label{cor: anti-holomorphic identity acting on worms in a disc}
Let $D$ be a disc, let $D^\dagger$ be its complex conjugate, and the $i:D\to D^\dagger$ be the anti-holomorphic identity map. If the support intervals of the worms are contained in $\mathring D$, then we have
\[
i_* | x_1 \ldots x_n \rangle_{D}
= | x_1^* \ldots x_n^* \rangle_{D^\dagger}
= | j(x_1^*) \ldots j(x_n^*) \rangle_{D^\dagger},
\]
where $j$ is as in Lemma~\ref{lem: <b> = < j(b*) >}.

If the support intervals of the worms meet $\partial D$, then we still have 
\[
i_* | x_1 \ldots x_n \rangle_{D}
%= | x_1^* \ldots x_n^* \rangle_{D^\dagger}
= | j(x_1^*) \ldots j(x_n^*) \rangle_{D^\dagger},
\]
but the expression $| x_1^* \ldots x_n^* \rangle_{D^\dagger}$ is no longer defined,
as the support intervals of the worms $x_i^*$ fail the orientation requirements from Definition~\ref{def: extendable}.
\end{cor}

\section{Point insertions}\label{sec:Point insertions}

\subsection{Construction of point insertions}\label{sec: construction of point insertions}

In this section, we establish a state-field correspondence between finite energy vectors $v\in H_0^{f.e.}$ in the vacuum sector, and point-localised fields (\emph{local operators} in the physics terminology).
Our main technical lemma establishes that every point insertion can be written as a linear combination of two worms:

\begin{lem}\label{lem: finite energy vectors via worms}
Let $v\in H_0 = H_0(\bbD)$ be a finite energy vector, and let $I_1,I_2\subset S^1$ be intervals whose interiors cover $S^1$. Then there exist $x_1\in\cA(I_1)$ and $x_2\in\cA(I_2)$ such that 
\begin{equation}\label{eq: two worms for a pt}
v=x_1\Omega+x_2\Omega.
\end{equation}
\end{lem}

\begin{proof}
We may assume without loss of generality that $v$ is homogeneous.
Let $n\in \bbN$ be the conformal dimension of $v$, and let $p_n:H_0\to H_0(n)$ be the orthogonal projection onto the subspace of vectors of conformal dimension $n$.
The projection $p_n$ is given by the formula
\[
p_n = \frac{1}{2 \pi} \int_{S^1} e^{-in\theta} R_\theta \, d\theta,
\]
where $R_\theta = e^{i \theta L_0}$ is the operator of anticlockwise rotation by $\theta$.
Let $\delta$ be the length of the smallest connected component of $I_1\cap I_2$, and
let $I_j^-$ be the result of shortening $I_j$ by $\delta/3$ on either side.

Let $J$ be an interval of length $\delta/3$ centered around $1$.
The map $\cA(J)\to H_0:x\mapsto x\Omega$ has dense image by the Reeh-Schlieder theorem.
Therefore $\cA(J)\to H_0(n):x\mapsto p_n x\Omega$ also has dense image.
As $H_0(n)$ is finite-dimensional, that last map is therefore surjective.
Pick $x\in \cA(J)$ such that $p_n(x\Omega)=v$, and pick a partition of unity $\{\varphi_1,\varphi_2\}$ subordinate to the cover $\{I_1^-, I_2^-\}$.
Then
\begin{align*}
v=p_n(x\Omega)
&= \left(\frac{1}{2 \pi} \int_{S^1} e^{-in\theta} R_\theta \, x \,d\theta \right) \Omega\\
&= \left(\frac{1}{2 \pi} \int_{S^1} e^{-in\theta} r_\theta(x) \, d\theta \right) \Omega\\
&= \left(\frac{1}{2 \pi} \int_{I_1^-} \varphi_1(e^{i \theta}) e^{-in\theta} r_\theta(x) \, d\theta\right) \Omega
%\\&\phantom{=} 
+ \left(\frac{1}{2 \pi} \int_{I_2^-}\varphi_2(e^{i \theta}) e^{-in\theta} r_\theta(x) \, d\theta\right) \Omega,
\end{align*}
where $r_\theta(x)=R_\theta x R_{-\theta}$.
We finish by noting that since $x\in\cA(J)$, we have $r_\theta(x) \in \cA(I_j)$ for every $\theta\in I_j^-$. 
Thus setting
$
x_j := \frac{1}{2 \pi} \int_{I_j^-} \varphi_j(e^{i \theta}) e^{-in\theta} r_\theta(x) \, d\theta \in \cA(I_j)
$
gives the desired decomposition \eqref{eq: two worms for a pt}.
\end{proof}

Given a point $z$ in some Riemann surface $\Sigma$ (in our applications, $\Sigma$ will either be a disc or $\bbC P^1$), one can form the inverse limit (i.e. intersection)
\[
H_0(z) := \underleftarrow{\lim}_{\mathring D\ni z} H_0(D)
\]
of all the vacuum sectors $H_0(D)$, were the inverse limit runs over all discs $D\subset \Sigma$ containing $z$ in their interior.
Here, given two discs $D_1 \subset D_2$ containing $z$, the structure map in the above inverse limit is given by
\[
Y_A:H_0(D_1)\to H_0(D_2),
\]
where $A$ is the annulus $D_2 \setminus \mathring D_1$.

The space $H_0(z)$ admits an action by the group $G_z$ of germs (defined on a neighbourhood of $z$) of invertible maps $\Sigma\to \Sigma$ fixing the point $z$.

\begin{defn}\label{def: V = space of finite energy vectors}
Let $z$ be a point in some Riemann surface.
The space $V(z)$ of finite energy vectors at $z$ is the subspace\footnote{We will show later, in Lemma \ref{lem: V(z)=V(z)}, that the space $V(z)$ is the same as the one considered in \cite[Ch. 6]{FrenkelBenZvi04}, in the context of vertex algebras.}
\begin{equation}\label{eq: def V(z) number 1}
V(z) := \big\{ v \in H_0(z)\,\big|\, \spann\{gv : g \in G_z\} \text{ is finite dimensional}\big\}.
\end{equation}
If the Riemann surface is $\bbC$ (or $\bbD$) and the point $z$ is the origin $0\in\bbC$, then we denote $V(z)$ simply by $V$.
\end{defn}

\begin{lem}
The space $V \subset H_0$ is the algebraic direct sum of all the eigenspaces of $L_0$.
\end{lem}

\begin{proof}
Let $G$ be the group of germs of invertible maps $\bbC\to \bbC$ fixing the origin.
Consider the $1$-parameter subgroup $S^1\subset G$ given by rotations.
If a vector $v\in H_0(\bbD)$ is not in the algebraic direct sum of the $L_0$-eigenspaces,
then $\spann\{gv : g \in S^1\}$ is already infinite dimensional, so $v\not \in V$.

Conversely, if $L_0v=Nv$ for some $N\in\bbN$, we will show that 
for any $g\in G$, we have $gv\in H_0^{\le N}:= \bigoplus_{n\le N} H_0(n)$,
and hence the span of all such vectors $gv$ is finite dimensional.
Let $g\in G$ be a group element.
By writing $g$ as $g'\exp(qL_0)$ for some scalar $q$ and some $g'\in\Univ_0$ fixing the origin, we are reduced to the case $g\in \Univ_0$.

If $g\in \Univ_0$, then
the action of $g$ on $V$ extends to a map $H_0 \to H_0$.
By \cite[Prop. 4.22]{HenriquesTener24ax} and the fact that $g$ fixes $0$, we may write $g$ as a time-ordered exponential
\[
g = \prod_{1\ge \tau \ge 0} \Exp( X_\tau d\tau)
\]
where each
$
X_\tau = \sum_{n\ge 0} a_m(\tau) L_m
$
is a holomorphic vector field on $\bbD$ that vanishes at the origin, and whose restriction to $\partial\bbD$ is an inward pointing vector field.
The operators $L_m$ for $m\ge 0$ preserve  $H_0^{\le N}$. Therefore so does $g$.
\end{proof}

By the arguments in the above proof, the action of $G$ (defined above) on $V$ extends to the group $\Aut(\bbC[[t]])$. Indeed, on each filtered piece $H_0^{\le N}$, the action of $G$ descends to its finite dimensional quotient $\Aut(\bbC[t]/t^{N+1})$, and the group $\Aut(\bbC[[t]])$ is the inverse limit of those finite dimensional quotients.

Starting from the action of $\Aut(\bbC[[t]])$ on $V$,
an alternative definition of $V(z)$ goes as follows. Let $\Sigma$ be a Riemann surface,  let $z\in \Sigma$ be a point, and
let $\cO_z(\Sigma)$ be the formal completion of the ring of holomorphic functions on $\Sigma$ at the point $z$. We can then define
\begin{equation}\label{eq: def V(z) number 2}
V(z) = V \times_{\Aut(\bbC[[t]])} \mathrm{Isom}(\cO_z(\Sigma),\bbC[[t]]),
\end{equation}
where we have written $\mathrm{Isom}(\cO_z(\Sigma),\bbC[[t]])$ for the set of $\bbC$-algebra isomorphisms between $\cO_z(\Sigma)$ and $\bbC[[t]]$. 
This construction is functorial in $z$ in the sense that whenever a holomorphic map $f:\Sigma_1\to \Sigma_2$ sends $z_1$ to $z_2$ and the map is a local isomorphism around these points, then there is an induced map $f_*:V(z_1)\to V(z_2)$. When $z_1=z_2=0\in\bbC$, this recovers with the original action of $\Aut(\bbC[[t]])$ on $V$.
In particular, given a local coordinate at a point $z\in \Sigma$ (i.e. a holomorphic map $w \mapsto f(w)$ defined for $w$ in some neighbourhood of $0\in\bbC$, and satisfying $f(0)=z$), we get an isomorphism $V \to V(z)$.
We denote this isomorphism by $v \mapsto v(z; f)$.
When the choice of local coordinate is clear from context, we write this isomorphism simply as $v \mapsto v(z)$.

\begin{lem}\label{lem: V(z)=V(z)}
The two constructions \eqref{eq: def V(z) number 1} and \eqref{eq: def V(z) number 2} of $V(z)$ are canonically isomorphic.
\end{lem}

\begin{proof}
Let $\cG$ be the connected groupoid whose objects are pairs $(\Sigma, z)$ with $z\in \Sigma$,\vspace{-1mm} and whose morphisms $(\Sigma_1, z_1)\to(\Sigma_2, z_2)$ are germs of local isomorphisms $\Sigma_1 \supset U \stackrel f \to \Sigma_2$ satisfying $f(z_1)=z_2$.

Both \eqref{eq: def V(z) number 1} and \eqref{eq: def V(z) number 2} are functorial in $z$ in the sense that a local isomorphisms $f:\Sigma_1\to \Sigma_2$ induces a map $f_*:V(z)\to V(f(z))$. In other words, both constructions are functors from $\cG$ to the category of vector spaces.
The two constructions visibly agree on the object $(\bbC,0)\in\cG$, and on $G=\Aut_\cG(\bbC,0)$.
We finish the proof by invoking the general fact that two functors $F,G:\cC\to\cD$ which are naturally equivalent when restricted to a skeleton of $\cC$ are in fact already naturally equivalent.
\end{proof}

\begin{defn}\label{semigroup of annuli with point insertions}
Given a vector space $V$ equipped with an action of $\Aut(\bbC[[t]])$,
we let $\Ann^V$ denote the semigroup whose elements consists of triples $(A,\{z_1,\ldots,z_n\},\{v_1,\ldots,v_n\})$ where $A$ is an annulus, $z_i\in \mathring A$ are an unordered collection of points in the interior of $A$, and $v_i\in V(z_i)$ (where $V(z_i)$ is defined as in \eqref{eq: def V(z) number 2}).
We will typically denote such a triple $A[v_1 \ldots v_n]$, or alternatively $A[v_1(z_1) \ldots v_n(z_n)]$ when we wish to make the insertion points explicit.

The central extensions $\tAnn_c^V$ and $\Ann_c^V$ are defined in the same way, using an element $\underline A$ of $\tAnn_c$ or $\Ann_c$ instead of $A\in \Ann$.
\end{defn}

When $v_1, \ldots, v_n \in V$, and the points $z_i$ are equipped with local coordinates $f_i$, we can then form $A[v_1 \ldots v_n]\in\Ann^V$ by implicitly replacing the $v_i$ by their images $v_i(z_i; f_i)$ under the isomorphism $V \cong V(z_i)$.
For example,
given an annulus $A$ equipped with an embedding $A\subset \bbC$, along with points $z_i\in \mathring{A}$ and vectors $v_i\in V$, we will typically identify the vector spaces $V(z_i)$ with $V$ using the standard local coordinates $w \mapsto w+z_i$.

If $D$ is a disc and $v\in V(z)$ for some point $z\in \mathring D$, let us write $|v\rangle_D$ for the image of $v$ in $H_0(D)$.
By definition, these satisfy
\begin{equation}\label{eq: Y|v> = |v>}
Y_{D_2\setminus \mathring D_1} (|v\rangle_{D_1}) = |v\rangle_{D_2}
\end{equation}
whenever $D_1\subset D_2$.

The following is an immediate corollary of Lemma~\ref{lem: finite energy vectors via worms}:

\begin{cor}\label{cor: finite energy vectors via worms in discs}
    Let $D$ be a disc, let $I_1,I_2 \subset \partial D$ be intervals whose interiors cover $\partial D$, and let $z \in \mathring{D}$ be a point in the interior of the disc.
    Then for every vector $v \in V(z)$ there exist $x_i \in \cA(I_i)$ such that $|v\rangle_D = x_1\Omega_D + x_2 \Omega_D$.
\end{cor}

For every configuration of $n$ distinct points $z_1,\ldots,z_n\in \mathring{D}$, and vectors $v_i \in V(z_i)$, we will now describe a vector
\begin{equation}\label{eq: n-point on disc}
| v_1 v_2 \cdots v_n \rangle_D \in H_0(D). 
\end{equation}
Pick disjoint discs $D_i\subset D$ containing the points $z_i$ in their respective interiors.
Using Corollary~\ref{cor: finite energy vectors via worms in discs}, pick intervals $I_{1,i},I_{2,i}\subset \partial D_i$ and algebra elements $x_{1,i}\in\cA(I_{1,i})$ and $x_{2,i}\in\cA(I_{2,i})$ such that
\[
|v_i\rangle_{D_i} = x_{1,i}\Omega_{D_i} + x_{2,i}\Omega_{D_i}
\]
in $H_0(D_i)$.
We then define $| v_1 v_2 \cdots v_n \rangle_D \in H_0(D)$ to be the sum of the following $2^n$ terms:
\begin{equation}\label{eqn: sum of worms}
\sum_{(\epsilon_1,\ldots,\epsilon_n)\in \{1,2\}^n} | x_{\epsilon_1,1} x_{\epsilon_2,2} \cdots x_{\epsilon_n,n}\rangle_D.
\end{equation}
We will show below that this is independent of the choices of discs $D_i$, intervals $I_{1,i},I_{2,i}$, and algebra elements $x_{\epsilon,i}\in\cA(I_{\epsilon,i})$).

\begin{lem}\label{lem: points insert well defined}
The vector \eqref{eqn: sum of worms} is independent of the choices of discs $D_i$, intervals $I_{1,i},I_{2,i}\subset \partial D_i$, and algebra elements $x_{1,i}\in\cA(I_{1,i})$ and $x_{2,i}\in\cA(I_{2,i})$.
\end{lem}
\begin{proof}
It is enough to show that \eqref{eqn: sum of worms} remains unchanged when one of the discs $D_i$ gets replaced by a smaller disc $D'_i \subseteq D_i$, and the two worms $x_{1,i}$ and $x_{2,i}$ supported on $\partial D_i$ get replaced by worms $x'_{1,i}$ and $x'_{2,i}$ supported on $\partial D'_i$.
Without loss of generality we may take $i = 1$.
Let $A=D\setminus \mathring D_1$, $A'=D\setminus \mathring D'_1$, and $A'' = D_1 \setminus \mathring D'_1$.
By 
\eqref{eq: compatibility of discs and annuli}
and \eqref{eq: Y|v> = |v>},
we then have:
\begin{align*}
\sum_{(\epsilon_1,\ldots,\epsilon_n)\in \{1,2\}^n} | x_{\epsilon_1,1} x_{\epsilon_2,2} \cdots x_{\epsilon_n,n}\rangle_D
&=
\sum_{(\epsilon_2,\ldots,\epsilon_n)\in \{1,2\}^{n-1}}
A[x_{\epsilon_2,2} \cdots x_{\epsilon_n,n}] \left(| x_{1,1} \rangle_{D_1}+| x_{2,1} \rangle_{D_1}\right)
\\
&=
\sum_{(\epsilon_2,\ldots,\epsilon_n)\in \{1,2\}^{n-1}}
A[x_{\epsilon_2,2} \cdots x_{\epsilon_n,n}] (| v_1\rangle_{D_1})
\\&=
\sum_{(\epsilon_2,\ldots,\epsilon_n)\in \{1,2\}^{n-1}}
A[x_{\epsilon_2,2} \cdots x_{\epsilon_n,n}] A''(| v_1\rangle_{D'_1})
\\&=
\sum_{(\epsilon_2,\ldots,\epsilon_n)\in \{1,2\}^{n-1}}
A'[x_{\epsilon_2,2} \cdots x_{\epsilon_n,n}] \left(| x'_{1,1} \rangle_{D'_1}+| x'_{2,1} \rangle_{D'_1}\right)
\\&=
\sum_{(\epsilon_1,\ldots,\epsilon_n)\in \{1,2\}^n} | x'_{\epsilon_1,1} x_{\epsilon_2,2} \cdots x_{\epsilon_n,n}\rangle_D.
\end{align*}
\end{proof}

As in \eqref{eqn: sum of worms}, given an annulus $A=D_{out} \setminus \mathring D_{in}$, points $z_1,\ldots,z_n\in \mathring{A}$, and vectors $v_i \in V(z_i)$, we can define an operator
$A[v_1\cdots v_n]: H_0(D_{in}) \to H_0(D_{out})$.
Pick disjoint discs $D_i\subset \mathring A$ containing the points $z_i$ in their interiors,
intervals $I_{1,i},I_{2,i}\subset \partial D_i$ and algebra elements $x_{\epsilon,i}\in\cA(I_{\epsilon,i})$
%and $x_{2,i}\in\cA(I_{2,i})$ 
satisfying
$|v_i\rangle_{D_i} = x_{1,i}\Omega_{D_i} + x_{2,i}\Omega_{D_i}$, and set
\begin{equation}\label{eq: A[vvv] one}
A[v_1\cdots v_n]:=\!\!
\sum_{(\epsilon_1,\ldots,\epsilon_n)\in \{1,2\}^n} \!A[x_{\epsilon_1,1} x_{\epsilon_2,2} \cdots x_{\epsilon_n,n}]
: H_0(D_{in}) \to H_0(D_{out}).
\end{equation}
Finally, given a $z$-lifted annulus $\underline{A}\in\tAnn_c$ with underlying annulus $A$, points $z_i\in \mathring A$ and $v_i \in V(z_i)$ as above, and a representation $K$ of the conformal net, we can form 
\begin{equation}\label{eq: A[vvv] two}
\underline{A}[v_1\cdots v_n]:=
\sum_{(\epsilon_1,\ldots,\epsilon_n)\in \{1,2\}^n} \underline{A}[x_{\epsilon_1,1} x_{\epsilon_2,2} \cdots x_{\epsilon_n,n}] \,:\, K\,\to\, K.
\end{equation}
As before, we will need to show that \eqref{eq: A[vvv] one} and \eqref{eq: A[vvv] two} are independent of the choices of discs $D_i$, and algebra elements $x_{\epsilon,i}$.
This will be proven further down in Lemma~\ref{lem: point insertion in annuli well defined}.

We have the following immediate consequence of Lemma~\ref{lem: A[worms] in B(K)}.
\begin{lem}\label{lem: eqn annulus with points insertions is local}
Let $\underline{A} \in \tAnn_c$ be localised in some interval $I$.
Then the operator
\begin{equation}
\label{eqn: annulus with points insertions is local}
\underline{A}[v_1 \ldots v_n]:H_0\to H_0
\end{equation}
defined in \eqref{eq: A[vvv] two} lies in $\cA(I)\subset B(H_0)$.
Moreover, if $(K,\rho)$ is a representation of the conformal net then
\[
\underline{A}[v_1 v_2 \ldots v_n]:K\to K
\]
is the image of \eqref{eqn: annulus with points insertions is local}
under the homomorphism 
$\rho_I:\cA(I)\to B(K)$. \hfill $\square$
\end{lem}

\begin{lem}\label{lem: point insertion in annuli well defined}
The operators \eqref{eq: A[vvv] one} and \eqref{eq: A[vvv] two} are 
independent of the choices of discs $D_i\subset A$, intervals $I_{1,i},I_{2,i}\subset \partial D_i$, and algebra elements $x_{1,i}\in\cA(I_{1,i})$ and $x_{2,i}\in\cA(I_{2,i})$.
\end{lem}

\begin{proof}
We treat both cases in parallel.
Let $A$ be an annulus which is either of the form $D_{out} \setminus \mathring D_{in}$, or an element of $\tAnn_c$.
By writing $A$ as a composite $A=BC$, where $B$ is thin/localised and contains all the insertion points, we are reduced to proving the statements for thin/localised annuli.

If $\underline A$ in \eqref{eq: A[vvv] two} is localised, then by Lemma~\ref{lem: eqn annulus with points insertions is local} the choice of representation is immaterial. So it's enough to prove the statement for $K=H_0$; we henceforth assume that to be the case.
By Lemma~\ref{lem: parametrized vs unparametrized annuli -- with worms},
the independence of \eqref{eq: A[vvv] two} from all the choices is then a consequence of the corresponding independence statement for \eqref{eq: A[vvv] one}.
We are thus reduced to proving that \eqref{eq: A[vvv] one} is independent of the choices of discs $D_i$, intervals $I_{\epsilon,i}$, and algebra elements $x_{\epsilon,i}$, when $A=D_{out} \setminus \mathring D_{in}$ is a thin annulus.

Let $I$ be an interval in the thin part of $\partial A$.
By Lemma~\ref{lem: points insert well defined}, the expression
\[
A[v_1\cdots v_n]\Omega_{D_{in}}=
|v_1\cdots v_n\rangle_{D_{out}}.
\]
is independent of all choices.
As a consequence of the Reeh-Schlieder theorem, the vacuum vector $\Omega_{D_{in}}$ is separating among operators $H_0(D_{in})\to H_0(D_{out})$ that commute with $\cA(I)$.
The operator $A[v_1\cdots v_n]$ commutes with the actions of $\cA(I)$ by Lemma~\ref{lem: thin part A(I) module map}. It follows that $A[v_1\cdots v_n]$ is also independent of all the choices.
\end{proof}

By construction, if $A=D_{out}\setminus \mathring D_{in}$ is an annulus, we have
\begin{equation}\label{eq: compatibility of discs and annuli -- point insertions}
A[v_1 v_2 \ldots v_n] | u_1 u_2 \ldots u_m\rangle_{D_{in}}
= | v_1 v_2 \ldots v_n u_1 u_2 \ldots u_m\rangle_{D_{out}}.
\end{equation}
For $\underline{A}_1,\underline{A}_2\in\tAnn_c$, and $K$ a representation of the conformal net, we also have
\begin{equation}\label{eq: representation of annuli with point insertions}
\underline{A}_1[v_1 v_2 \ldots v_n] \underline{A}_2[u_1 u_2 \ldots u_m]
= \underline{A}_1 \underline{A}_2 [v_1 v_2 \ldots v_n u_1 \ldots u_m].
\end{equation}
This gives an action on $K$ of the semigroup $\tAnn_c^V$ (defined in Definition~\ref{semigroup of annuli with point insertions}). And when $K=H_0$, that action descends to $\Ann_c^V$.

Note that the relationship between the operators associated to parametrised and to unparametrised annuli established in Lemmas~\ref{lem: Aunderline in terms of A} and~\ref{lem: parametrized vs unparametrized annuli -- with worms} extends verbatim to the case of point insertions:

\begin{lem}\label{lem: parametrized vs unparametrized annuli -- with point fields}
Let $A = D_{out} \setminus \mathring{D}_{in}$ be an annulus with boundary parametrizations $\varphi_{in/out}:S^1\to\partial D_{in/out}$, and let $\underline{A}\in\tAnn_c$ be a $z$-lift.
Pick unitaries $U(\varphi_{in/out}):H_0  \to  H_0(D_{in/out})$ implementing $\varphi_{in/out}$, and $z \in \bbC^\times$ satisfying 
$\underline{A} = z \cdot U(\varphi_{out})^*A U(\varphi_{in})$.
Then the same relation holds in the presence of point insertions:
\[
\underline{A}[v_1 v_2 \ldots v_n] = z \cdot U(\varphi_{out})^* A[v_1 v_2 \ldots v_n] U(\varphi_{in}).
\]
\end{lem}
\begin{proof}
Immediate from Lemma~\ref{lem: parametrized vs unparametrized annuli -- with worms} and the definition of $\underline{A}[v_1 v_2 \ldots v_n]$.
\end{proof}

The following is immediate from Lemma~\ref{lem: thin part A(I) module map}:

\begin{lem}\label{lem: thin part A(I) module map -- points}
Let $A=D_{out} \setminus \mathring D_{in}$ be a thin annulus, and let $I$ be in the thin part of $A$.
Then 
\[
A[v_1 v_2 \ldots v_n]: H_0(D_{in}) \to H_0(D_{out})
\]
is an $\cA(I)$-module map.
\hfill $\square$
\end{lem}

We finish this section with an analog of Lemma~\ref{lem: univalent acting on worm insertions}:

\begin{lem}\label{lem: univalent acting on point insertions}
Let $f:\bbD \to \bbD$ be a univalent map, and let $\underline{A}_f \in \Univ \subset \Ann_c$ be the corresponding lifted annulus.
Let $z_1, \ldots, z_n \in \mathring{\bbD}$ and $w_1, \ldots, w_m \in \mathring{A}_f$ be points satisfying $z_i\neq z_j$ and $w_i\neq w_j$ when $i\neq j$, and let $v_1, \ldots, v_n, u_1, \ldots, u_m \in V$ be vectors.
Then
\begin{align*}
\underline{A}_f[u_1(w_1) &\ldots u_m(w_m)]\big(|v_1(z_1;f_1) \ldots v_n(z_n;f_n)\rangle_\bbD\big) \\
&= |u_1(w_1) \ldots u_m(w_m) v_1(f(z_1);f \circ f_1) \ldots v_n(f(z_n);f \circ f_n) \rangle_\bbD
\end{align*}
(where we use the standard local coordinates around the points $w_i$).
\end{lem}
\begin{proof}
    Pick discs $D_i \subset \bbD$ around $z_i$, intervals $I_{\epsilon,i} \subset \partial D_i$, and elements $x_{\epsilon,i} \in \cA(I_{\epsilon,i})$ satisfying
    \[
    |v_i(z;f_i)\rangle_{D_i} = x_{1,i} \Omega_{D_i} + x_{2,i} \Omega_{D_i}.
    \]
    Similarly, pick discs $D_i' \subset A_f$, intervals $J_{\epsilon,i} \subset \partial D_i'$, and elements $y_{\epsilon,i} \in \cA(J_{\epsilon,i})$ such that
    \[
    |u_i(w_i)\rangle_{D_i'} = y_{1,i} \Omega_{D_i'} + y_{2,i}\Omega_{D_i'}.
    \]
    Writing $U_i:H_0(D_i) \to H_0(f(D_i))$ for the unitary induced by $f:D_i\to f(D_i)$, we then have
    \[
    \hspace{-2mm}
    |v_i(f(z);f \circ f_i) \rangle_{f(D_i)} 
    = U_i|v_i(z;f_i)\rangle_{D_i} 
    = U_i(x_{1,i} \Omega_{D_i} + x_{2,i} \Omega_{D_i}) 
    = \cA(f)(x_{1,i})\Omega_{f(D_i)} + \cA(f)(x_{2,i})\Omega_{f(D_i)}.
    \]
    Hence, by Lemma~\ref{lem: univalent acting on worm insertions}, we have:
    \begin{align*}
    \underline{A}_f[&u_1(w_1) \ldots u_m(w_m)] \big(|v_1(z_1;f_1) \ldots v_n(z_n;f_n)\rangle_\bbD\big)\\ 
    &= \Big(\sum_{\substack{(\epsilon_1,\ldots,\epsilon_n)\in \{1,2\}^n\\(\epsilon'_1,\ldots,\epsilon_m')\in \{1,2\}^m}} \underline{A}_f[y_{\epsilon_1',1} \ldots y_{\epsilon_m',m}] | x_{\epsilon_1,1}  \cdots x_{\epsilon_n,n}\rangle_\bbD\Big)\\
    &= \sum_{\substack{(\epsilon_1,\ldots,\epsilon_n)\in \{1,2\}^n\\(\epsilon'_1,\ldots,\epsilon_m')\in \{1,2\}^m}} |y_{\epsilon_1',1} \ldots y_{\epsilon_m',m} \, \cA(f)(x_{\epsilon_1,1})  \cdots \cA(f)(x_{\epsilon_n,n})\rangle_\bbD\\
    &= |u_1(w_1) \ldots u_m(w_m) \, v_1(f(z_1);f \circ f_1) \ldots v_n(f(z_n);f \circ f_n) \rangle_\bbD. \qedhere
    \end{align*}
\end{proof}

\subsection{Holomorphicity of point insertions}

Given a Riemann surface $\Sigma$, the space
\begin{equation}\label{eq: def of V_Sigma}
V_\Sigma := \coprod_{z\in\Sigma} V(z) \to \Sigma
\end{equation}
is a holomorphic bundle (more precisely, an inductive limit of finite dimensional holomorphic bundles).
Indeed, the bundle 
$J \Sigma:= \coprod_{z\in\Sigma} \mathrm{Isom}(\cO_z(\Sigma),\bbC[[t]])\to\Sigma$ is a principal bundle for the group $\Aut(\bbC[[t]])$, and $V_\Sigma = V \times_{\Aut(\bbC[[t]])} J \Sigma$ is an associated bundle
 (see \cite[Ch.~6]{FrenkelBenZvi04}).
%from the trivial bundle with fiber $V$ over the space of holomorphic embeddings $\bbD \to \Sigma$, by quotienting out the diagonal action of $\Univ_0$, the semigroup of univalent maps $\bbD \to \bbD$ that fix the origin.

Given an embedding $\Sigma_1 \hookrightarrow \Sigma_2$, we have a natural isomorphism between $V_{\Sigma_1}$ and the restriction  of $V_{\Sigma_2}$ to $\Sigma_1$.
For example, if $\Sigma$ is embedded in $\bbC$, then $V_{\Sigma}$ is the trivial bundle $\Sigma\times V$,
and holomorphic section of $V_\Sigma$ are the same thing as holomorphic functions $\Sigma \to V$.
When that is case, then for any vector $v \in V$ we can associate the corresponding constant function on $\Sigma$; we denote the induced section of $V_\Sigma$ by $z \mapsto v(z)$.
Note that this $v(z)$ agrees with what we had previously denoted $v(z; w \mapsto w + z)$.  

\begin{lem}\label{lem: point insertions are holomorphic}
Let $M$ be a finite dimensional complex manifold.
\begin{enumerate}[i)]
\item  Let $D$ be a disc, $z_1,\ldots,z_n:M\to \mathring D$ be holomorphic functions,
and $v_1,\ldots,v_n:M\to V_D$ be holomorphic lifts of those functions to $V_D$ (defined in \eqref{eq: def of V_Sigma}).
Then
\[
m\mapsto | v_1(m)\ldots v_n(m)\rangle_D
\]
is a holomorphic function from $M$ to $H_0(D)$.

\item Let $A = D_{out} \setminus \mathring D_{in}$ be an annulus, $z_1,\ldots,z_n:M\to \mathring A$ be holomorphic functions,
and $v_1,\ldots,v_n:M\to V_A$ be holomorphic lifts. % of those functions.
Then
\begin{equation}\label{eq: non parametrised holomorphic family of insertions}
m\mapsto A[ v_1(m)\ldots v_n(m) ]
\end{equation}
is a holomorphic function from $M$ to the space of bounded linear maps $H_0(D_{in})\to H_0(D_{out})$.

\item If $\underline{A} \in \tAnn_c$ is a $z$-lift of $A$, $z_i:M\to \mathring A$ and $v_i:M\to V_A$ are as above, and $K$ is a representation of the conformal net, then
\begin{equation}\label{eq: parametrised holomorphic family of insertions}
m\mapsto \underline{A}[ v_1(m)\ldots v_n(m) ]
\end{equation}
is a holomorphic function from $M$ to the space of bounded linear maps $K \to K$.
\end{enumerate}
\end{lem}

\begin{proof}
It suffices to consider the case of annuli, as given a disc $D = D_{out}$ we may, locally in $M$, choose a disc $D_{in} \subset D_{out}$ which is disjoint from the images of the $z_i$ and write 
\[
| v_1(m)\ldots v_n(m)\rangle_D = A[ v_1(m)\ldots v_n(m) ] \Omega_{D_{in}}.
\]
It also suffices to consider the case of a single insertion, as we may decompose (again locally in $M$)
\[
\underline{A}[ v_1(m)\ldots v_n(m) ] = \underline{A}_1[v_1(m)] \ldots \underline{A}_n[v_n(m)],
\]
%as in the proof of Lemma \ref{lem: holomorphicity of worm insertions in annuli}, 
and similarly for the non-parametrised version.

We treat the two cases \eqref{eq: non parametrised holomorphic family of insertions} and \eqref{eq: parametrised holomorphic family of insertions} in parallel.
Let $A$ be an annulus.
Let $z(m)$ be the location of the insertion, and let $v(m)$ be the vector that is being inserted. Let us also assume without loss of generality that $A$ is embedded in $\bbC$, so that the section $v(m)$ is given by a holomorphic function $M\to V$.
The image of such a function necessarily lies in a finite dimensional subspace of $V$. 
After picking a basis $\{v_i\}$ of that subspace, we may write $v(m)=\sum_i f_i(m) v_i$, where $f_i$ are scalar-valued holomorphic functions on $M$.
We are therefore reduced to the case where the vector $v(m)$ is independent of $m\in M$.
Call that vector $v$.

At this stage, may assume without loss of generality that $M=\mathring A$. Our task is to show that $A[v(z)]$
depends holomorphically on the point $z\in \mathring A$ (where $v(z) \in V(z)$ is the image of $v$ under the isomorphism $V \cong V(z)$ induced by the standard local coordinate).
We do this locally around some point $z_0\in \mathring A$.

Let $\varepsilon>0$ be such that the closed ball $B(z_0,2\varepsilon)$ is contained in $\mathring A$, and let us
pick intervals $I_1,I_2 \subset \partial B(z_0,\varepsilon)$ and elements $x_i \in \cA(I_i)$ such that
\[
|v(z_0)\rangle_{B(z_0,\varepsilon)} = x_1\Omega_{B(z_0,\varepsilon)} + x_2 \Omega_{B(z_0,\varepsilon)}.
\]
For every $z\in B(z,\varepsilon)$, let $I_i(z) \subset \partial B(z,\varepsilon)$ be the image of $I_i$ under translation by $z-z_0$. 
Similarly, let $x_i(z) \in A(I_i(z))$ be the image of $x_i$ under the same translation.
We have $|v(z)\rangle_{B(z,\varepsilon)} = x_1(z)\Omega_{B(z,\varepsilon)} + x_2(z) \Omega_{B(z,\varepsilon)}$ so that, by definition,
\[
A[v(z)] = A[x_1(z)] + A[x_2(z)].
\]
The operators $A[x_1(z)]$ and $A[x_2(z)]$ depend holomorphically on $z$ by Lemma~\ref{lem: holomorphicity of worm insertions in annuli}, so $A[v(z)]$ also depends holomorphically on $z$.
\end{proof}

\begin{remark}
There is a slightly stronger version of the above lemma (and also of Lemma \ref{lem: holomorphicity of worm insertions in annuli}) where we also allows the annulus to vary holomorphically as a function of $m\in M$.
This stronger statement follows readily from the lemma by factoring the annuli, locally in $M$, as a product of an annulus with fixed geometry and varying point insertions, and two annuli with no point insertions and a geometry that depends on $m\in M$. 
\end{remark}

The following statement is essentially equivalent to Lemma~\ref{lem: point insertions are holomorphic}:

\begin{cor}\label{cor: embedded point insertions are holomorphic}
If $D$ is a disc embedded in $\bbC$, then $| v_1(z_1)\ldots v_n(z_n)\rangle_D \in H_0(D)$ depends holomorphically on the points $z_i\in D$, and the vectors $v_i\in V$.

Similarly, if $A=D_{out} \setminus \mathring D_{in}$ is an annulus embedded in $\bbC$, and $\underline{A} \in \tAnn_c$ is a $z$-lift of $A$, then the operators
$\underline{A}[v_1(z_1)\ldots v_n(z_n)] : K \to K$ and $A[v_1(z_1)\ldots v_n(z_n)] : H_0(D_{in}) \to H_0(D_{out})$ depend holomorphically on the points $z_i\in A$, and the vectors $v_i\in V$.
\end{cor}

\subsection{Covariance for orientation reversing maps}
\label{sec: Covariance for orientation reversing maps}

Recall that $\Theta:H_0 \to H_0$ is the antiunitary operator corresponding to the map $z \mapsto \overline{z}:\bbD\to \bbD$.
Since $\Theta$ commutes with $L_0$ (as it commutes with the real $1$-parameter semigroup generated by $L_0$), we have $\Theta(V) = V$.

We know from Lemma~\ref{lem: finite energy vectors via worms} that given a finite energy vector $v\in V\subset H_0$, we can find intervals $I_1,I_2 \subset S^1$ and algebra elements $x_i \in \cA(I_i)$ such that 
\[
v = x_1\Omega + x_2\Omega = | x_1 \rangle_\bbD + | x_2 \rangle_\bbD.
\]
The vector $\Theta(v)\in V$ can then be expressed as 
\begin{equation}\label{eq: Th(v) = <th(a)> + <th(b)>}
\Theta(v)= | \theta(x_1) \rangle_\bbD + | \theta(x_2) \rangle_\bbD
\end{equation}
because $\Theta(v)
=\Theta(x_1\Omega+x_2\Omega)
=\Theta(x_1\Theta\Omega+x_2\Theta\Omega)
=\theta(x_1)\Omega+\theta(x_2)\Omega$.

The following is a version of Lemma~\ref{lem: Theta applied to worm insertions} for point-insertions:

\begin{lem}\label{lem: Theta applied to point insertions}
Let $v_1, \ldots, v_n \in V$ and let $z_1, \ldots, z_n \in \mathring{\bbD}$ be distinct points.
Then
\[
\Theta \big(| v_1(z_1) \ldots v_n(z_n) \rangle_{\bbD}\big) =
| (\Theta v_1)(\overline{z_1}) \ldots (\Theta v_n)(\overline{z_n}) \rangle_{\bbD} .
\]
\end{lem}

\begin{proof}
Pick disjoint discs $D_i\subset D$ containing the points $z_i$ in their interiors, intervals $I_{1,i},I_{2,i}\subset \partial D_i$ and elements $x_{1,i}\in\cA(I_{1,i})$, $x_{2,i}\in\cA(I_{2,i})$ satisfying
$|v_i\rangle_{D_i} = x_{1,i}\Omega_{D_i} + x_{2,i}\Omega_{D_i}$ so that,
by definition, 
\[
| v_1 v_2 \cdots v_n \rangle_{\bbD} = \sum_{(\epsilon_1,\ldots,\epsilon_n)\in \{1,2\}^n} | x_{\epsilon_1,1} x_{\epsilon_2,2} \cdots x_{\epsilon_n,n}\rangle_{\bbD} .
% ? v1 v2 ... vn ?\bbD = ?( ?_1,...,?_n)?{1,2}^n ?a?_1,1 a?_2,2 ... a?_n,n?\bbD
\]
We then have:
\begin{align*}
\Theta(| v_1(z_1) \ldots v_n(z_n) \rangle_{\bbD}) &= 
\Theta\big(\sum_{(\epsilon_1,\ldots,\epsilon_n)\in \{1,2\}^n} | x_{\epsilon_1,1} x_{\epsilon_2,2} \cdots x_{\epsilon_n,n}\rangle_{\bbD}\big) \\
&= \sum_{(\epsilon_1,\ldots,\epsilon_n)\in \{1,2\}^n} | \theta(x_{\epsilon_1,1}) \theta(x_{\epsilon_2,2}) \cdots \theta(x_{\epsilon_n,n})\rangle_{\bbD} \\
&=| (\Theta v_1)(\overline{z_1}) \ldots (\Theta v_n)(\overline{z_n}) \rangle_{\bbD}
\end{align*}
where the second equality holds by Lemma~\ref{lem: Theta applied to worm insertions}, and the third one follows from~\eqref{eq: Th(v) = <th(a)> + <th(b)>}.
\end{proof}

Recall that the action of $\Vir_{\ge 0}:=\mathrm{Span}\{L_n\}_{n\ge 0}$ on $V$
%of the positive modes $L_n$, $n\ge 0$, of the Virasoro algebra 
exponentiates to an action of the pro-algebraic group $\Aut(\bbC[[t]])$. We wish to extend it to an action of $\Aut(\bbC[[t]])\rtimes \bbZ/2$, where the generator $\vartheta$ of $\bbZ/2$ acts by complex conjugation:
\[
\vartheta\cdot \left(\sum a_i t^i\right) = \sum \overline{a_i}t^i .
\]

\begin{lem}\label{lem: action of Aut(C[[t]]) rtimes Z/2}
Let $I_- \subset S^1$ be the lower semicircle, and let $\Theta:V \to V$ be the restriction to $V=H_0^{f.e.}$ of the operator with same name defined in \eqref{eq: first Theta}. Then the action of $\Aut(\bbC[[t]])$ on $V$ together with the involution $\Theta$ assemble to an action on $V$ of the semidirect product
\[
\Aut_\bbR(\bbC[[t]]) = \Aut(\bbC[[t]])\rtimes \bbZ/2.
\]
\end{lem}

\begin{proof}
We have a zigzag of inclusions
\[
\Diff_c(S^1) \hookrightarrow  \Ann_c \hookleftarrow \Univ_0 \hookrightarrow \Aut(\bbC[[t]]),
\]
where $\Univ_0$ is the sub-semigroup of univalent maps that fix the origin.
As in \eqref{eq: pi:Ann_c rtimes Z/2 -> B(H_0)}, the action of $\Ann_c(S^1)$ on $H_0$ extends to an action of the semidirect product $\Ann_c \rtimes \bbZ/2$.
The sub-semigroup of univalent maps $\Univ_0 \rtimes \bbZ/2$ therefore also acts on $H_0$.

The action of $\Univ_0 \rtimes \bbZ/2$ preserves $V$ (this is evident using Definition~\ref{def: V = space of finite energy vectors}). Since $\Univ_0$ generates a dense subgroup in $\Aut(\bbC[[t]])$, the commutation relations between $\Theta$ and $\Univ_0$ defining the semidirect product extend to $\Aut(\bbC[[t]])$.
So $\Aut(\bbC[[t]])\rtimes \bbZ/2$ acts on $V$.
\end{proof}

Given a point $z\in D$ in a disc, a holomorphic local coordinate $f$, and a vector $v\in V$, we have introduced the notation $v(z; f)$ in Section~\ref{sec: construction of point insertions}
for the insertion $v_1(z)$, where $v_1\in V(z)$ is the image of $v$ under the isomorphism $V\to V(z)$ induced by $f$. We can extend this notation to include anti-holomorphic local coordinates (i.e., anti-holomorphic maps $w \mapsto f(w)$ with $f(0)=z$, defined for $w$ in some neighbourhood of $0$). Once again, the map $f$ induces an isomorphism $V\to V(z)$ (which is now complex anti-linear), and we write $v(z; f)$ for the insertion $v_1(z)$, where $v_1\in V(z)$ is the image of $v$ under the above isomorphism.

We can use this freedom to use anti-holomorphic local coordinates to equip the semigroups and $\Ann^V$ and $\tAnn_c^V$ (Definition~\ref{semigroup of annuli with point insertions}) with the structures of dagger-semigroups.
Recall that $\Ann$ and $\tAnn_c$ have been equipped with the structures of dagger-semigroups in Section \ref{sec: semigroup of annuli} by declaring the dagger of an annulus $A$ to be the complex conjugate annulus $A^\dagger$ with incoming and outgoing boundaries exchanged.
If we write $\iota:A\to A^\dagger$ for the anti-holomorphic ``identity'' map, then
this map induces an anti-linear isomorphism $V(z) \cong V(\iota(z))$ which we shall denote $v\mapsto \bar v$.

\begin{defn}\label{semigroup of annuli with point insertions - dagger}
Let $V$ be a complex vector space equipped with a representation of $\Aut(\bbC[[t]])\rtimes \bbZ/2$ which is linear on the identity component and anti-linear on the non-identity component.
We may then equip $\Ann^V$ and $\tAnn_c^V$ with the structure of a dagger-semigroup by
setting 
\[
A[v_1 \ldots v_n]^\dagger \,:=\, A^\dagger[ \bar v_1 \ldots \bar v_n]
\]
where $v_i \in V(z_i)$, and $z_1, \ldots, z_n \in \mathring{A}$.
The dagger structures on the central extensions $\tAnn_c^V$ and $\Ann_c^V$ are defined in the same way, by using the dagger structures on $\tAnn_c$ and $\Ann_c$, respectively.
\end{defn}

Let $0\in D_{in}\subset D_{out}\subset \bbC$ be discs in the complex plane.
For $A=D_{out}\setminus D_{in}$ with boundary parametrisations $\varphi_{in/out}$, we may identify $A^\dagger$ with the image of $A$ under the inversion map $\zeta \mapsto \bar \zeta^{-1}$, equipped with the induced boundary parametrisations.
In that case, given $v_i\in V$ and $z_i\in\mathring A$, we have
\begin{equation}\label{eqn: AnnV dagger operation}
A[v_1(z_1) \ldots v_n(z_n)]^\dagger = A^\dagger[v_1(z_1)^\dagger \ldots v_n(z_n)^\dagger]
\end{equation}
where $v(z)^\dagger := \tilde v(\bar z^{-1})$ and $\tilde v=e^{\bar zL_1}(-\bar z^{-2})^{L_0}\Theta v$ is the image of $v\in V$ under the change of coordinate
\[
\big(\zeta \mapsto (\bar\zeta-\bar z)^{-1}+\bar z^{-1}\big)\in \Aut_\bbR(\bbC[[t]]) = \Aut(\bbC[[t]])\rtimes \bbZ/2.
\]
Note the similarity with the notation $Y(v,z)^\dagger$ used in \eqref{eq: not: Y(v,z)^dagger}.

We finish this section by discussing how the representations of $\tAnn_c^V$ constructed in \eqref{eq: A[vvv] two} and \eqref{eq: representation of annuli with point insertions} are compatible with the dagger structures:

\begin{prop} \label{prop: annuli adjoints with point insertions}
For any representation $K$ of the conformal net, the operators $\underline{A}[v_1, \ldots, v_n]$ (constructed in \eqref{eq: A[vvv] two}, and well defined by
Lemma~\ref{lem: point insertion in annuli well defined})
give a $*$-representation of $\tAnn^V_c$ on $K$.
\end{prop}

\begin{proof}
The map $\tAnn^V_c\to B(K)$ is a representation by \eqref{eq: representation of annuli with point insertions}, so all that remains is to do is check that it is a $*$-representation.
That is, for $\underline{A} \in \tAnn_c$ and $v_i\in V(z_i)$, we must show that the corresponding operators
\[
\underline{A}[v_1,\ldots,v_n]: K \to K
\quad \text{and}  \quad
\underline{A}^\dagger [\bar v_1,\ldots,\bar v_n]: K \to K
\]
are each other's adjoints.
Since $\tAnn^V_c$ is generated by $\tAnn_c$ along with the annuli with a single insertion, we may assume without loss of generality that $n\le 1$. The case of annuli with no insertions was treated in \cite{HenriquesTenerIntegratingax} (and restated in \eqref{eq: pi(Adag)=pi(A)*}), so we focus on the case $n=1$.
Namely, for $z\in \mathring A$ and $v\in V(z)$, we must show that
$\underline A[v]^*=\underline A^\dagger[\bar v]$.

Let $D\subset \mathring A$ be a disc 
%such that $z\in\mathring D$,
containing the point $z$ in its interior,
and let $I_{1},I_{2}\subset \partial D$ and $x\in\cA(I_{1})$ and $y\in\cA(I_{2})$ such that
$|v\rangle_{D} = |x\rangle_{D} + |y\rangle_{D}$, so that $\underline{A}[v]=\underline{A}[x]+\underline{A}[y]$.

Letting  $i:D\to D^\dagger$ be the anti-holomorphic ``identity'' map, we have
\begin{equation}\label{eqn: bar vi in terms of j}
| \bar v \rangle_{D^\dagger}
=i_*(| v \rangle_{D})
=i_*(| x \rangle_{D} + | y \rangle_{D})
=| j(x^*) \rangle_{D^\dagger} + | j(y^*) \rangle_{D^\dagger}
\end{equation}
(where $i_*$ is as in \eqref{eqn: pushforward by (anti)holomorphic map}),
where the first equality holds by the definition of $\bar v$, and the
last equality holds by Corollary~\ref{cor: anti-holomorphic identity acting on worms in a disc}.
It follows that
\[
\underline{A}[v]^*=\underline{A}[x]^*+\underline{A}[y]^*
= \underline{A}^\dagger[x^*]
+\underline{A}^\dagger[y^*]
= \underline{A}^\dagger[j(x^*)]
+ \underline{A}^\dagger[j(y^*)]= \underline{A}^\dagger [\bar v]
\]
where the second equality holds by Lemma~\ref{lem: adjoint of the operator underline A}, the third equality holds by Lemma~\ref{lem: <b> = < j(b*) >},
and the last one follows from \eqref{eqn: bar vi in terms of j}.
\end{proof}

\section{The vertex algebra associated to a conformal net}
\label{sec: The vertex algebra associated to a conformal net}

In this section, starting with a conformal net $\cA$, we construct a vertex algebra structure on the space $V=H_0^{f.e.}$ of finite energy vectors of the vacuum sector of the conformal net. Our only assumptions on the conformal net are that it is diffeomorphism covariant, and that the generator $L_0$ of rotations has finite-dimensional eigenspaces when acting on $H_0$ (see Remark~\ref{rem: VOA without finite dimensional eigenspaces} for what we believe can be achieved in the absence of that second condition).

\subsection{Completion of the state space}

Let $D\subset \bbC$ be a disc.
In the previous section, we have seen how to assign a vector
\[
| v_1(z_1) v_2(z_2) \cdots v_n(z_n) \rangle_{D} \in H_0(D). 
\]
to a configuration of points $z_1,\ldots,z_n\in \mathring D$, and a collection of vectors $v_i \in V$.
Moreover, if $D' \subset D$ is such that $v_1,\ldots,v_m\in\mathring D'$, and $v_{m+1},\ldots,v_n\in \mathring D\setminus D'$, then letting $A:=D\setminus\mathring D'$, we have
$| v_1(z_1) \ldots v_n(z_n) \rangle_{D} = 
A[v_{m+1}(z_{m+1}) \ldots v_n(z_n)] | v_1(z_1) \ldots v_m(z_m) \rangle_{D'}$.

We now wish to introduce a space $\widehat{V}$ which contains all the expressions $| v_1(z_1) \cdots v_n(z_n) \rangle$, without the constraint that the points $z_i$ lie in some specified disc.
If $V$ has the weight space decomposition $V=\bigoplus_{n\in\bbN}V(n)$, then we set 
$\widehat{V} := \prod_{n\in\bbN}V(n)$, and call it the \emph{algebraic completion} of $V$.

Consider the collection of all discs (not necessarily round) embedded in the complex plane.
This is a poset under inclusion, and for any conformal net $\cA$, there is a functor from that poset to the category of Hilbert spaces given by $D \mapsto H_0(D)$.
We denote the colimit by
\begin{equation}\label{eq: H_0 for union of all discs}
\widehat{H_0}(\bbC) := \underrightarrow{\lim}_{D} H_0(D).
\end{equation}
We claim that $\widehat{H_0}(\bbC)$ is naturally a subspace of $\widehat{V}$.
Indeed, the limit in \eqref{eq: H_0 for union of all discs} may be computed over the sub-poset of round discs centered at the origin. For $D$ such a round disc, the Hilbert spaces $H_0(D)$ admits a direct sum decomposition in terms of its weight spaces for rotation:
\[
H_0(D)=\bigoplus_{n\in\bbN}\!\!\!{\phantom{\big|}}^{\ell^2}H_0(D)(n).
\]
The inclusion $V\hookrightarrow H_0(D)$ induces isomorphisms of weight spaces $V(n)\cong H_0(D)(n)$, and the $\ell^2$-direct sum maps to the product in the obvious way:
\begin{equation}\label{eq: map to hat V}
\bigoplus_{n\in\bbN}\!\!\!{\phantom{\big|}}^{\ell^2}H_0(D)(n) \to 
\prod_{n\in\bbN} H_0(D)(n) \cong \prod_{n\in\bbN} V(n).
\end{equation}
Taken together, the maps \eqref{eq: map to hat V} assemble to a map 
$\widehat{H_0}(\bbC)\to \widehat V$.
In particular, for every disc $D\subset \bbC$ (not necessarily round) there is a canonical map $H_0(D)\to \widehat V$.

By construction, for any discs $0\in D'\subset D \subset \bbC$, we then have a commutative diagram
\[
\begin{tikzcd}
V \arrow[rd, bend right=10] \arrow[r] & H_0(D') \arrow[d,"Y_{D \setminus \mathring{D}'}"]\arrow[rd, bend left=10] &\\
& H_0(D) \arrow[r] & \widehat{V}
\end{tikzcd}
\]
Finally, if $z_1,\ldots,z_n \in \bbC$ are distinct points, and $v_1,\ldots,v_n$ are vectors in $V$, we write
\begin{equation}\label{eq: <v_1(z_1)...v_n(z_n)> in H_0(C)}
| v_1(z_1) v_2(z_2) \cdots v_n(z_n) \rangle \in \widehat V
\end{equation}
for the image of $| v_1(z_1) v_2(z_2) \cdots v_n(z_n) \rangle_{D} \in H_0(D)$ under the map $H_0(D)\to \widehat V$, for any large enough disc $D\subset \bbC$.

\subsection{Geometric vertex algebras}
\label{sec: geometric vertex algebras}

In this section we will show that the functions $| v_1(z_1) v_2(z_2) \cdots v_n(z_n) \rangle$ constructed in \eqref{eq: <v_1(z_1)...v_n(z_n)> in H_0(C)} equip the vector space $V$ with the structure of a \emph{geometric vertex algebra}.
We use a definition from \cite{Bruegmann2012}.

\begin{defn}[Geometric vertex algebra]
The data of a (bounded below, locally finite dimensional) geometric vertex algebra is:
\begin{itemize}
\item A $\bbN$-graded vector space $V = \bigoplus_{n \in \bbN} V(n)$ over $\bbC$ with finite-dimensional weight spaces.
\item Let $\Conf_n = \bbC^n \setminus \Delta$ denote the configuration space of $n$ points in $\bbC$.
For every $n \ge 0$, $z = (z_1, \ldots, z_n) \in \Conf_n$ and $v=(v_1, \ldots, v_n) \in V^n$, we have a vector $$| v(z) \rangle = | v_1(z_1) \ldots v_n(z_n) \rangle \in \widehat{V}:= \prod_{n\in\bbN}V(n)$$ which is multilinear in $v$ and depends holomorphically on $z$.
\end{itemize}

This data must satisfy the following axioms.
\begin{itemize}
\item (permutation invariance) For every permutation $\sigma \in S_n$ we have $| v^\sigma(z^\sigma) \rangle = | v(z) \rangle$.
\item (insertion at zero) For every $v \in V$ we have $| v(0) \rangle = v$.
\item (associativity) For $n \in \bbN$, let $p_n : \widehat{V} \to V(n)$ denote the projection. 
For all $v_1, \ldots, v_m$, $u_1, \ldots, u_n \in V$ and configurations $z \in \Conf_{m+1}$ and $w \in \Conf_n$ with $\max_i \abs{w_i} < \min_{1 \le j \le m} \abs{z_j - z_{m+1}}$ we require that the sum
\begin{equation}\label{eqn: gva associativity sum}
\sum_{k \in \bbN} \Big| v_1(z_1) \ldots v_m(z_m) \, p_k | u_1(w_1) \ldots u_n(w_n) \rangle(z_{m+1})\Big\rangle
\end{equation}
converges in the sense that, for each $\ell \in \bbN$, the $\ell$-th graded piece of the above expression is a locally uniformly absolutely convergent sum of functions of the $z_i$ and $w_j$ with values in $V(\ell)$.
Finally, we require that
\begin{align*}
\notag
\big| v_1(z_1) \ldots v_m(z_m)  u_1(w_1 + z_{m+1}) \ldots &u_n(w_n + z_{m+1}) \big\rangle
\\
&=
\Big| v_1(z_1) \ldots v_m(z_m) \, | u_1(w_1) \ldots u_n(w_n) \rangle(z_{m+1}) \Big\rangle 
\end{align*}
where the right-hand side is defined by the sum \eqref{eqn: gva associativity sum}.
\item ($\bbC^\times$-covariance) For all $z \in \bbC^\times$, $v_1, \ldots, v_n \in V$, and $w \in \Conf_n$,
\begin{equation}\label{eqn: Ctimes invariance}
z . | v_1(w_1) \ldots v_n(w_n) \rangle =  | z.v_1(zw_1) \ldots z.v_n(z w_n) \rangle
% z . ?a1(w1) ... an(wn)? = ?z.a1(zw1) ... z.an(zwn)?
\end{equation}
where $z$ acts on $\widehat{V}$ and $V$ via multiplication by $z^\ell$ on $V(\ell)$.
\end{itemize}
\end{defn}

\begin{remark}
The definition in \cite{Bruegmann2012} includes a meromorphicity axiom requiring that for some $N$ the function $z^N | v(z) u(0) \rangle$ has a removable singularity at $z=0$.
However, by Prop. 2.3 of loc.~cit., this axiom is redundant if the grading on $V$ is bounded below, which we have assumed.
\end{remark}

\begin{thm}\label{thm: geometric vertex algebra from conformal net}
Let $\cA$ be a conformal net, and let 
$V=H_0^{f.e.}$ be the space of finite energy vectors of its vacuum sector (Definition~\ref{def: V = space of finite energy vectors}).
Then the linear maps
\begin{align*}
V^{\otimes n} \,\,\,&\to\,\,\,\,\, \widehat{V} \\
v_1\otimes\ldots\otimes v_n \,&\mapsto\, | v_1(z_1) v_2(z_2) \cdots v_n(z_n) \rangle
\end{align*}
defined in \eqref{eq: <v_1(z_1)...v_n(z_n)> in H_0(C)}
equip $V$ with the structure of a geometric vertex algebra.
\end{thm}

\begin{proof}
The quantity $| v_1(z_1) v_2(z_2) \cdots v_n(z_n) \rangle$ depends holomorphically on the points $z_1,\ldots, z_n$ by Lemma~\ref{lem: point insertions are holomorphic}

The permutation axiom hold true by construction.
In Proposition~\ref{proposition independence}, we proved that worm-insertions are independent of the choice of ordering of the insertions. This independence of the choice of ordering propagates to $| v_1 v_2 \cdots v_n \rangle_D$, and hence to \eqref{eq: <v_1(z_1)...v_n(z_n)> in H_0(C)}.

To show the insertion at zero axiom, 
write $v = x_1\Omega + x_2\Omega$ as in \eqref{eq: two worms for a pt}.
Then using the embeddings
$V \hookrightarrow H_0(\bbD) \hookrightarrow \widehat{V}$,
we have
$| v(0) \rangle = | v(0) \rangle_{\bbD} = x_1\Omega + x_2\Omega = v$. 

For $\bbC^\times$-covariance \eqref{eqn: Ctimes invariance},
it suffices to show 
\begin{equation}\label{eq: p_k of stuff}
p_k \, z . | v_1(w_1) \ldots v_n(w_n) \rangle = 
p_k | z.v_1(zw_1) \ldots z.v_n(zw_n) \rangle .
\end{equation}
Both sides are holomorphic functions of $z$ and $w_j$,
so it's enough to verify the equality when $\abs{z} = 1$ and $\abs{w_j} < 1$.
In that case, by summing over $k$, \eqref{eq: p_k of stuff} is equivalent to the equation
\[
z . | v_1(w_1) \ldots v_n(w_n) \rangle_\bbD \,=\, 
| z.v_1(zw_1) \ldots z.v_n(zw_n) \rangle_\bbD
\]
in $H_0(\bbD)$.
Let $U_z:H_0(\bbD) \to H_0(\bbD)$ be the unitary induced by the automorphism $w \mapsto zw$ of $\bbD$, and let us denote by the same letter its restriction to $V$. Then, by Lemma~\ref{lem: univalent acting on point insertions}, we have
\begin{align*}
z . | v_1(w_1) \ldots v_n(w_n) \rangle_{\bbD}
&=
U_z | v_1(w_1) \ldots v_n(w_n) \rangle_{\bbD}\\
&=
| U_z v_1(zw_1) \ldots U_z v_n(zw_n) \rangle_{\bbD}
=
| z.v_1(zw_1) \ldots z.v_n(zw_n) \rangle_{\bbD}.
\end{align*}

For the associativity axiom,
let $\ell \in\bbN$, $v_1,\ldots,v_m,$ $u_1,\ldots,u_n\in V$,
and $z_1,\ldots,z_{m+1},$ $w_1,\ldots,w_n\in \bbC$ 
as in the statement of the axiom.
We must show that the sum
\[
\sum_{k\in\bbN} \,\,p_\ell\, \big| v_1(z_1)\ldots v_m(z_m) \, p_k | u_1(w_1)\ldots u_n(w_n)\rangle (z_{m+1}) \big\rangle
\]
is locally uniformly absolutely convergent, and that this sum is equal to
\[
p_\ell\, \big| v_1(z_1) \ldots v_m(z_m)  u_1(w_1 + z_{m+1}) \ldots u_n(w_n + z_{m+1}) \big\rangle.
\]
Let $D$ be a round disc centered around $z_{m+1}$, containing the points $z_{m+1}+w_i$ in its interior, and containing none of the points $z_j$ for $j\le m$, and let $D'$ be a round disc centered around $0$, containing $D$ and all the $z_j$ in its interior.
It suffices to show that, in $H_0(D')$, we have
\begin{align} \label{eqn: assoc convergence in Dprime}
\sum_{k\in\bbN} \,\;&p_\ell\,\big| v_1(z_1)\ldots v_m(z_m) \, p_k | u_1(w_1)\ldots u_n(w_n)\rangle (z_{m+1}) \big\rangle_{D'} \\
&=\,p_\ell\, \big| v_1(z_1) \ldots v_m(z_m)  u_1(w_1 + z_{m+1}) \ldots u_n(w_n + z_{m+1}) \big\rangle_{D'} \notag
\end{align}
with locally uniform absolute convergence.
Let $D_0:=D-z_{m+1}$, and let
$f:D_0 \to D$ be the isomorphism given by translation by $z_{m+1}$.
If we write $U:H_0(D_0) \to H_0(D)$ for the unitary corresponding to $f$,
then we have
\begin{align}
\notag
\big| p_k |u_1(w_1) \ldots u_n(w_n)\rangle(z_{m+1})\big\rangle_D
&=
U \big| p_k |u_1(w_1) \ldots u_n(w_n)\rangle(0)\big\rangle_{D_0}\\
\label{eq: stuff in proof of ass axiom}
&=
U p_k |u_1(w_1) \ldots u_n(w_n)\rangle_{D_0}\\
\notag
&=
(U p_k U^*) |u_1(w_1+z_{m+1}) \ldots u_n(w_n+z_{m+1}) \rangle_D.
\end{align}
Let $\tilde p_k = U p_k U^*$, and note that $\sum \tilde p_k = 1_{H_0(D)}$ in the strong operator topology.
Hence
\begin{align*}
\sum_{k \in \bbN} \;\big| p_k |u_1(w_1) \ldots u_n(w_n)\rangle(z_{m+1})\rangle_D
&=
\sum_{k \in \bbN} \;\tilde p_k |u_1(w_1+z_{m+1}) \ldots u_n(w_n+z_{m+1}) \rangle_D\\
&= |u_1(w_1+z_{m+1}) \ldots u_n(w_n+z_{m+1}) \rangle_D.
\end{align*}
Let $A:=D'\setminus\mathring D$. Applying the bounded operator $A[v_1(z_1) \ldots v_m(z_m)]$ to the above equation yields
\begin{align*}
\sum_{k\in\bbN} \;&\big| v_1(z_1)\ldots v_m(z_m) \, p_k | u_1(w_1)\ldots u_n(w_n)\rangle (z_{m+1}) \big\rangle_{D'}\\ 
&= \big| v_1(z_1) \ldots v_m(z_m)  u_1(w_1 + z_{m+1}) \ldots u_n(w_n + z_{m+1}) \big\rangle_{D'},
\end{align*}
with pointwise convergence in a neighbourhood of the $z_j$ and $w_i$.

It remains to establish the locally uniform absolute convergence of \eqref{eqn: assoc convergence in Dprime}.
Applying the operator $A[v_1(z_1) \ldots v_m(z_m)]$ to \eqref{eq: stuff in proof of ass axiom} gives us
\begin{align*}
\big|& v_1(z_1)\ldots v_m(z_m) \, p_k | u_1(w_1)\ldots u_n(w_n)\rangle (z_{m+1}) \big\rangle_{D'}\\
& =
A[v_1(z_1) \ldots v_m(z_m)] \tilde p_k |u_1(w_1+z_{m+1}) \ldots u_n(w_n+z_{m+1}) \rangle_D.
\end{align*}
By Lemma~\ref{lem: point insertions are holomorphic}, $A[v_1(z_1) \ldots v_m(z_m)]$ and $|u_1(w_1+z_{m+1}) \ldots u_n(w_n+z_{m+1}) \rangle_D$ depend holomorphically on the points $z_j$ and $w_i$. So, on sufficiently small neighbourhoods of these points, we have
\[
\norm{\,\big| v_1(z_1)\ldots v_m(z_m) \, p_k | u_1(w_1)\ldots u_n(w_n)\rangle (z_{m+1}) \big\rangle_{D'}} \le C
\]
for some constant $C$ independent of $k$ and of the points $z_j$ and $w_i$.
Pick a number $r<1$ which close enough to $1$ so that points $r z_j$ still lie in these neighbourhoods.
By $\bbC^\times$-covariance, 
\begin{align*}
p_\ell \,&\big| v_1(z_1)\ldots v_m(z_m) \, p_k | u_1(w_1)\ldots u_n(w_n)\rangle (z_{m+1}) \big\rangle_{D'}\\
&= r^{-\ell} \, p_\ell \, r.\big| v_1(z_1)\ldots v_m(z_m) \, p_k | u_1(w_1)\ldots u_n(w_n)\rangle (z_{m+1}) \big\rangle_{D'}\\
&= r^{-\ell+\sum n_j+k} \, p_\ell \, \big| v_1(rz_1)\ldots v_m(rz_m) \, p_k | u_1(w_1)\ldots u_n(w_n)\rangle (rz_{m+1}) \big\rangle_{D'},
\end{align*}
when $v_j \in V(n_j)$.
Hence, on sufficiently small neighbourhoods of the $z_j$ and $w_i$, we have
\[
\norm{p_\ell \big| v_1(z_1)\ldots v_m(z_m) \, p_k | u_1(w_1)\ldots u_n(w_n)\rangle (z_{m+1}) \big\rangle_{D'}} \le r^k C',
\]
where again $C'$ does not depend on $k$ or the points $z_j$ or $w_i$.
It follows that \eqref{eqn: assoc convergence in Dprime} converges locally uniformly absolutely, as required.
\end{proof}

Combined with \cite[Thm 1.11]{Bruegmann2012} (which states that the structures of a vertex algebra and that of a geometric vertex algebra are equivalent), the above theorem proves that $V$ is naturally equipped with the structure of a vertex algebra:

\begin{cor}\label{cor: n point functions from conformal nets}
Let $\cA$ be a conformal net with vacuum sector $H_0$, and let $V \subset H_0$ be the subspace of finite energy vectors, equipped with its standard $\bbN$-grading.
Then the maps
\[
Y(v,z)u := | v(z)u(0) \rangle, \qquad Tv := \frac{d}{dz} | v(z) \rangle\Big|_{z=0}
\]
equip $V$ with the structure of an $\bbN$-graded vertex algebra.
Moreover, whenever $\abs{z_1} > \cdots > \abs{z_m}$, we have
\begin{equation}\label{YYY = ...}
Y(v_1, z_1) \ldots Y(v_m,z_m)\Omega \,=\, | v_1(z_1) \ldots v_m(z_m) \rangle
\end{equation}
as vectors in $\widehat V$.
\end{cor}

\begin{remark}\label{rem: VOA without finite dimensional eigenspaces}
In our definition of conformal net (Definition~\ref{def: definition of conformal net}),
we have chosen to include the requirement that the operator $L_0$ has finite dimensional eigenspaces --- see the footnote on page~\pageref{foot: finite dim L0 eigenspaces} for a discussion of that condition. This ensures that the graded pieces of the associated vertex algebra are finite dimensional, and is used crucially in the proof of Lemma~\ref{lem: finite energy vectors via worms}.

If the $L_0$-eigenspaces of $H_0$ are not finite dimensional, we believe that it is still possible to get a vertex algebra structure on a suitable subspace $V=\bigoplus V(n)\subset \bigoplus H_0(n)=H_0^{f.e.}$ of the space of finite energy vectors (albeit not one whose graded pieces are finite-dimensional).
The graded pieces $V(n)$ consist of those vectors $\xi\in H(n)$ that can be written as
$\xi=\int_{M} | x_1(m) \ldots x_n(m) \rangle_\bbD\, d\mu(m)$
where $M$ is a topological space (or maybe just a manifold), $x_i(m)$ are as in \eqref{eq: cont fam of worms}, and $\mu$ is a measure on~$M$.
\end{remark}

Recall the standard vertex algebra notation $Y(u,z)v = \sum u_{(n)}v\,z^{-n-1}$ (see Section~\ref{sec: VOAs}), so that 
\begin{equation}\label{eq: u(n)v}
u_{(n)}v \,=\, 
\oint_{\abs{z}=\epsilon} Y(u,z)v\, z^n  \,dz
\end{equation}
in $\widehat V$.

\begin{lem}\label{lem: modes of fields in annuli}
Let $\cA$ be a conformal net,
let $V\subset H_0$ be as in Corollary~\ref{cor: n point functions from conformal nets}, and let $u,v \in V$.
Given an embedded annulus $A=D_{out}\setminus\mathring D_{in}\subset \bbC$, $z \in \mathring{A}$, $n \in \bbZ$, and $\epsilon>0$ small enough,
we have an equality
\begin{equation}\label{eqn: associativity in annuli -- unparametrised}
\oint_{w\,:\,\abs{w-z}=\epsilon} {A}[v(z)u(w)] (w-z)^n  \,dw
=
{A}[ \big(u_{(n)}v\big)(z) ]
\end{equation}
of operators $H_0(D_{in})\to H_0(D_{out})$.

If $\underline{A} \in \tAnn_c$ is a $z$-lift of $A$, and $K$ is a representation of $\cA$,
we also have
\begin{equation}\label{eqn: associativity in annuli}
\oint_{w\,:\,\abs{w-z}=\epsilon} \underline{A}[v(z)u(w)] (w-z)^n  \,dw
=
\underline{A}[ \big(u_{(n)}v\big)(z) ]
\end{equation}
as operators on $K$.
\end{lem}
\begin{proof}
By performing a translation, we may assume without loss of generality that $z=0$.
Writing $A=BC$ (respectively $A=\underline B\hspace{.5mm}\underline C$), where $B$ is thin (respectively $\underline B$ is localised in some interval) and contains all the insertion points, we are reduced to proving the statements for thin/localised annuli.
By Lemma~\ref{lem: A[worms] in B(K)},
we may then assume without loss of generality that $K=H_0$ in the statement of \eqref{eqn: associativity in annuli}.
Finally, \eqref{eqn: associativity in annuli} follows from \eqref{eqn: associativity in annuli -- unparametrised} by Lemma~\ref{lem: parametrized vs unparametrized annuli -- with point fields}. So we are reduced to proving \eqref{eqn: associativity in annuli -- unparametrised} for thin annuli.

Let $I\subset A$ be contained in the thin part of $A$. By the Reeh-Schlieder theorem, the vacuum vector $\Omega_{D_{in}}$ is separating among $\cA(I)$-module maps $H_0(D_{in})\to H_0(D_{out})$.
The operators in \eqref{eqn: associativity in annuli -- unparametrised} are $\cA(I)$-module maps by Lemma~\ref{lem: thin part A(I) module map -- points}. So
it's enough to check that equation  upon evaluating on $\Omega_{D_{in}}$, in which case it becomes
\[
\oint_{\abs{w}=\epsilon} |u(w)v(0)\rangle_{D_{out}} w^n  \,dw
=
| (u_{(n)}v)(0)\rangle_{D_{out}}.
\]
The latter is immediate from \eqref{eq: u(n)v} (using \eqref{YYY = ...}).
\end{proof}

We finish this section by identifying the translation operator of the vertex algebra constructed in Corollary~\ref{cor: n point functions from conformal nets}.
Recall that the \emph{translation operator} of $V$ is the operator
\begin{equation}\label{eq: transl op}
V\to V: v \,\mapsto\, \frac{d}{dz} |v(z) \rangle\Big|_{z=0} .
\end{equation}

\begin{lem}\label{lem: T is L-1}
The translation operator \eqref{eq: transl op} coincides with (the restriction to $V$ of) the operator $L_{-1}$ on $H_0$ constructed in Proposition~\ref{prop: exist: stress energy tensor}.
\end{lem}
\begin{proof}
For $v\in V$, we will show that $\frac{d}{dz} | v(z) \rangle_{\bbD}\big|_{z=0} = L_{-1}v$.
Fix $v \in V(n)$, and fix a positive number $r < 1$.
For every $z \in \bbD$ sufficiently small, we have a univalent map $f:\bbD \to \bbD$ given by $f(w) = rw + z$.
Writing just $A_f:H_0\to H_0$ for the operator associated to the annulus $A_f \in \Univ$ then, by Lemma~\ref{lem: univalent acting on point insertions}, we have:
\begin{equation}\label{eqn: v(z)D}
A_f v
=A_f |v(0)\rangle_{\bbD} = 
r^n|v(z)\rangle_{\bbD}.
\end{equation}
The annulus $A_f$ has a univalent framing $h:S^1 \times [0,1] \to A_f$ given by
\[
h(\theta,t) = r^{1-t}e^{i \theta} + \frac{1-r^{1-t}}{1-r} z,
\]
with associated path (Definition~\ref{def: Framings and paths}) given by:
\[
X(t) = -\frac{h_t}{h_\theta} \partial_\theta = -i \ln(r) \left(1 - \frac{e^{-i\theta} z}{1-r}\right) \partial_\theta = 
\ln(r)L_0 - \frac{z\ln(r)}{1-r} L_{-1}
\]
(and since $X(t)$ does not depend on $t$, we will just write $X$ below).
By definition, the operator $A_f$ on $H_0$ is then given by
\[
A_f = \prod_{1 \ge t \ge 0} \Exp(X \, dt) \;\!= \exp(X).
\]
We can then compute:
\begin{align*}
\frac{d}{dz} |v(z)\rangle_{\bbD}\Big|_{z=0} &=  r^{-n}\frac{d}{dz} A_f v\Big|_{z=0}\\
&=  r^{-n}\frac{d}{dz} \exp\Big(\ln(r)L_0 - \frac{z\ln(r)}{1-r} L_{-1}\Big) v\Big|_{z=0}\\
&=  r^{-n}\int_{0}^1\left.
\exp((1-s)X)
%\left(\prod_{1 \ge t \ge s} \Exp(X)\, dt \right)
\left(-\frac{\ln(r)}{1-r} L_{-1}\right) 
\exp(sX)
%\left(\prod_{s \ge t \ge 0} \Exp(X)\, dt \right)
v\, ds\right|_{z=0}  \\
&= -r^{-n}\left(\frac{\ln(r)}{1-r}\right)\int_0^1  r^{(1-s)L_0}\,  L_{-1}\,  r^{sL_0} \, v \, ds\\
&= -\left(\frac{\ln(r)}{1-r}\int_0^1 r^{1-s} \, ds \right)  L_{-1}v\\
&= L_{-1}v
\end{align*}
where the first equality holds by \eqref{eqn: v(z)D}, and 
the third one holds by
\cite[Lem. 3.7]{HenriquesTenerIntegratingax}. \end{proof}

\subsection{The conformal vector}\label{sec: conformal vector}

Recall that the vacuum sector $H_0$ of the conformal net $\cA$ comes equipped with a positive energy representation of $\Diff_c(S^1)$, which differentiates to a representation by unbounded operators of the Virasoro algebra $\Vir_c$ (see e.g. Proposition~\ref{prop: exist: stress energy tensor} or \cite[\S3.2]{CKLW18}). In this section, we will show that the associated stress-energy tensor
\[
T(z) := \sum_{n \in \bbZ} L_n z^{-n-2} : V \to \widehat V
\]
is one of the fields of the vertex algebra constructed in Corollary~\ref{cor: n point functions from conformal nets}.
This makes $V$ into a vertex operator algebra (as opposed to a mere vertex algebra), with conformal vector $L_{-2} \Omega$.

Let $\nu := L_{-2} \Omega$, let
\begin{align*}
Y(\nu,z):V&\,\to\, \widehat V\qquad\qquad \text{(for $z \in \bbC^\times$)}\hspace{-2cm}
\\
u\,\mapsto&\,\, | \nu(z)u(0) \rangle
\end{align*}
be the corresponding vertex algebra field,
and let $S_n:V\to V$ be the operators defined by
\[
Y(\nu,z) = \sum_{n \in \bbZ} S_n z^{-n-2}.
\]
The main goal of this section is to show that $S_n=L_n$ (where the $L_n$'s are as in Proposition~\ref{prop: exist: stress energy tensor}).
By construction, $S_{-2}\Omega = \nu = L_{-2}\Omega$.
It follows by induction that 
\begin{equation}\label{eq: SOm = LOm}
S_{-n}\Omega = L_{-n}\Omega\qquad\qquad \forall n \ge 2.\hspace{-2cm}
\end{equation}
Indeed, assuming 
\eqref{eq: SOm = LOm} for a given $n$,
we have
\[
L_{-(n+1)}\Omega = \tfrac{1}{n-1} [L_{-1}, L_{-n}]\Omega = \tfrac{1}{n-1} [L_{-1}, S_{-n}]\Omega = S_{-(n+1)}\Omega,
\]
where the first equality is one of the Virasoro relations, the second equality is our inductive hypothesis, and the third one is the translation axiom for the vertex algebra $V$ %constructed in Corollary~\ref{cor: n point functions from conformal nets} 
(along with Lemma~\ref{lem: T is L-1}).
We also have $L_n \Omega = S_n \Omega$ for all $n \ge -1$, as both expressions are zero (the first one is zero by the classification of unitary representations of the Virasoro algebra, and the second one is zero by the vacuum axiom of the vertex algebra $V$). It follows that
\begin{equation}\label{eqn: S and T agree on Omega}
Y(\nu,z)\Omega \,=\, T(z)\Omega
\end{equation}
in $\widehat V$, for all $z \in \bbC$.

The challenge in showing $Y(\nu,z) = T(z)$ is that these two objects have different technical characteristics, owing to their very different constructions.
We will geometrically modify them, so as to enable a direct comparison.

For $f \in C^\infty(S^1)$, let $Y(\nu,f):V \to \widehat V$ be the \emph{smeared field} defined by 
\begin{equation}\label{eq:  Y(nu,f)}
Y(\nu,f)\,v \,:=\, \sum_{n \in \bbZ} \hat f(n) S_n v = \int_{S^1} f(z) \cdot z^2 Y(\nu,z)v \, \frac{dz}{2\pi i z},
\end{equation}
where $\hat f(n)$ denotes the $n$-th Fourier coefficient of $f$.
We write $T(f)$ for the closure of the map $V \to H_0:v \mapsto \sum \hat f(n) L_n v$, and recall from Proposition~\ref{prop: exist: stress energy tensor} that $T(f)$ is affiliated with $\cA(I)$ whenever $\supp(f) \subset I$.
We wish to show that $Y(\nu,f)v=T(f)v$ for all $v\in V$, and $f \in C^\infty(S^1)$.

\begin{defn}\label{def; Ann rl0}
Let $r^{\ell_0}$ denote the round annulus $A_r:=\{z\in\bbC:r\le|z|\le 1\}$ with its standard boundary parametrizations, and
let
\[
\Ann^{\le r^{\ell_0}} := \{ A \in \Ann \mid \exists B \in \Ann : r^{\ell_0}=AB\}
\]
be the set of annuli $A$ such that $r^{\ell_0}=AB$ for some annulus $B$.
\end{defn}
Annuli $A\in \Ann^{\le r^{\ell_0}}$ admit a standard embedding in the unit disc because $A_r\subset \bbD$.

\begin{lem}\label{lem basic property of Ann^(le r^(ell_0))}
Let $K$ be a positive energy representation of the Virasoro algebra, and let $K^{f.e.}\subset K$ be its subspace of finite energy vectors (as in Definition~\ref{def: H^f.e.}).

If $\underline A$ and $\underline B \in \tAnn_c$ are such that $\underline{A}\hspace{.5mm}\underline{B}=r^{L_0}$,
then the range of $\underline{A}:K\to K$ contains $K^{f.e.}$, and 
\begin{equation}\label{eq: AK=BK}
\underline{A}^{-1}K^{f.e.} = \underline{B}K^{f.e.}.
\end{equation}
Moreover, $\underline{A}^{-1} K^{f.e.}$ is dense in $K$.
\end{lem}

\begin{proof}
Since $\underline{A}\hspace{.5mm}\underline{B}=r^{L_0}$, the range of $\underline{A}$ contains $K^{f.e.}$.
The inclusion $\underline{B}K^{f.e.} \subseteq \underline{A}^{-1}K^{f.e.}$ is immediate, so we focus on the reverse inclusion.
If $\eta \in K$ and $\underline{A}\hspace{.5mm}\eta \in K^{f.e.}$,
then 
\[
\underline{A}\hspace{.5mm} \eta = \underline{A}\hspace{.5mm}\underline{B}\hspace{.5mm} r^{-L_0}\underline{A}\hspace{.5mm}\eta.
\]
By Lemma~\ref{lem: annuli have dense image}, $\underline{A}$ is injective, hence $\eta = \underline{B}r^{-L_0}\underline{A}\eta$. It follows that $\eta \in \underline{B}K^{f.e.}$, as required.

Finally, since $K^{f.e.}$ is dense in $K$, and $\underline{B}:K\to K$ has dense image (Lemma \ref{lem: annuli have dense image}), the subspace \eqref{eq: AK=BK} is also dense in $K$.
\end{proof}

Given an annulus $A$, we denote by 
\[
\partial_{out}^{th} A:= \partial_{out} A\setminus (\partial_{out} A\cap \partial_{in} A)
\]
the ``thick part'' of the outgoing boundary of $A$.
If $A\in \Ann^{\le r^{\ell_0}}$, then
using the standard embedding of $A$ into $\bbD$, we may view $\partial_{out}^{th} A$ as a subset of $S^1$.
Let $\underline{A} \in \Ann_c$ be a lift
of $A$, and let $f \in C^\infty(S^1)$ be supported in some interval $I\subset \partial_{out}^{th} A$.
Then the operator
$\underline{A}^*\hspace{.5mm} T(f)\hspace{.5mm} \underline{A}$
on $H_0$ is densely defined by Lemma \ref{lem basic property of Ann^(le r^(ell_0))}, and can be used to define
a sesquilinear form 
\[
\xi,\eta\,\mapsto\, \langle \underline{A}^*T(f)\underline{A}\xi,\eta \rangle  =  \langle T(f) \underline{A}\xi, \underline{A}\eta \rangle
\]
with domain $\underline{A}^{-1}V$.

Now consider the sesquilinear form
$
\xi,\eta\,\mapsto\, 
\langle Y(\nu,f) \underline{A} \xi, \underline{A} \eta \rangle
$,
which is a smeared version of the sesquilinear form
\[
\xi,\eta\,\mapsto\, 
\langle Y(\nu,z) \underline{A} \xi, \underline{A} \eta \rangle
\]
with same domain $\underline{A}^{-1}V$.
The latter can be understood as coming from a point insertion in the doubled annulus $\hat A = A^\dagger \cup A$:

\begin{lem} \label{lem: vertex algebra fields in annuli}
Fix $A \in \Ann^{\le r^{\ell_0}}$, with its standard embedding in $\bbD$, along with a lift $\underline{A} \in \Ann_c$.
Let $\underline{A}^\dagger \in \Ann_c$ be the conjugate of $\underline{A}$, with underlying annulus $A^\dagger$, realized in $\bbC$ as the image of $A$ under the reflection $z \mapsto \overline{z}^{-1}$.
Finally let $\underline{\hat{A}} := \underline{A}^\dagger\underline{A}$, with underlying annulus $\hat A := A^\dagger\cup A\subset \bbC$.\vspace{-1mm}

Then, for any $\xi,\eta \in \underline{A}^{-1}V$, $v \in V$, and $z \in \mathring{\hat A}$, we have
\begin{equation}\label{eq: A hat expression}
\langle \underline{\hat{A}}[v(z)] \xi, \eta \rangle = \langle Y(v,z) \underline{A} \xi, \underline{A} \eta \rangle.
\end{equation}
\end{lem}

\begin{proof}
If $A = A_r := \{z \in \bbC : r \le |z| \le 1\}$, then \eqref{eq: A hat expression} is equivalent to
\begin{equation}\label{eq: worm in Ahat}
\underline{\hat{A}}[v(z)]u = r^{L_0}Y(v,z)r^{L_0}u
\vspace{-1mm}
\end{equation}
for $u \in V$ and $z\in \mathring{\hat A}$.
Both expressions being holomorphic, it is enough to show \eqref{eq: worm in Ahat} for $\abs{z} < 1$, 
in which case $\underline{\hat{A}}[v(z)] = r^{L_0} \underline A_r[v(z)]$.
So we are reduced to showing $\underline A_r[v(z)]u = Y(v,z)r^{L_0}u$ in $\widehat V$.
Unpacking the definitions, this is equivalent to showing the equality
\[
\underline A_r[v(z)]u = | v(z) (r^{L_0}u)(0) \rangle_{\bbD}
\]
in $H_0(\bbD)$.
The latter is an instance of Lemma~\ref{lem: univalent acting on point insertions}.

Back to the general case, choose $\underline{B} \in \tAnn_c$ such that $\underline A_r = \underline{A} \hspace{.5mm} \underline{B}$.
We then have $\hat A_r = \underline{B}^\dagger \hspace{.5mm} \underline{\hat{A}} \hspace{.5mm} \underline{B}$.
If $\underline{A}\xi, \underline{A}\eta \in V$ then,
by Lemma~\ref{lem basic property of Ann^(le r^(ell_0))}, we may write $\xi = \underline{B}\xi^\prime$ and $\eta = \underline{B}\eta^\prime$.
We can then compute:
\begin{align*}
\langle \underline{\hat{A}}[v(z)] \xi, \eta \rangle 
&= \langle \underline{\hat{A}}[v(z)] \underline{B}\xi^\prime, \underline{B}\eta^\prime \rangle\\
&=\langle \underline{B}^* \underline{\hat{A}}[v(z)] \underline{B}\xi^\prime, \eta^\prime \rangle\\
&=\langle \underline{\hat{A}}_{r}[v(z)] \xi^\prime, \eta^\prime \rangle\\
&=\langle Y(v,z) r^{L_0} \xi^\prime, r^{L_0} \eta^\prime \rangle\\
&=\langle Y(v,z) \underline{A}\xi, \underline{A} \eta \rangle.\qedhere
\end{align*}
\end{proof}

We can now prove the main result of this section.

\begin{thm}\label{thm: unitary VOA from conformal net}
Let $\cA$ be a conformal net, let $V$ be the associated vertex algebra (constructed in Corollary~\ref{cor: n point functions from conformal nets}),
and let $Y:V \to \End(V)[[z^{\pm 1}]]$ be the be state-field correspondence of the vertex algebra.

Let $T(z) = \sum_{n \in \bbZ} L_n z^{-n-2}$ be the stress-energy tensor of $\cA$ (constructed in Proposition~\ref{prop: exist: stress energy tensor}),
and let $\nu := L_{-2} \Omega$.
Then
\[
Y(\nu,z) = T(z), 
\]
and $V$ is a vertex operator algebra with conformal vector $\nu$.
\end{thm}

\begin{proof}
Let $f \in C^\infty(S^1)$ be supported in some interval $I$.
Fix $J$ containing $I$ in its interior, and let $\underline{A} \in \tAnn_c(J) \cap \Ann^{\le r^{\ell_0} }$ be localised in $J$, with $I \subset \partial_{out}^{th} A$.
Finally, let $\underline{\hat{A}} := \underline{A}^\dagger \underline{A}$, with underlying annulus $\hat A$ embedded in $\bbC^\times$ as in Lemma~\ref{lem: vertex algebra fields in annuli}.

By
\vspace{-1mm}
Lemmas~\ref{lem: point insertions are holomorphic} and \ref{lem: eqn annulus with points insertions is local}, the assignment $z \mapsto \underline A[\nu(z)]$ 
defines a holomorphic function $\mathring{\hat A}\to \cA(J)$.
Since $I \subset \mathring{\hat A}$,
the expression 
\begin{equation}\label{eqn: Anuf is in A(J)}
\underline{\hat{A}}[\nu[f]] := \int_{S^1} f(z) z^2 \underline{\hat{A}}[\nu(z)] \frac{dz}{2\pi iz}
\end{equation}
makes sense, and is again valued in $\cA(J)$.

By Lemma~\ref{lem: vertex algebra fields in annuli},
the sesquilinear form on $\underline{A}^{-1}V$ associated to
$\underline{\hat{A}}[\nu[f]]$ satisfies
\begin{equation}\label{eq: nuf versus Y(nu,f)}
\langle \underline{\hat{A}}[\nu[f]] \xi, \eta \rangle = \langle Y(\nu,f) \underline{A}\xi, \underline{A}\eta \rangle.
\end{equation}
Pick $\underline B\in \Ann_c$ so that $r^{L_0}=\underline A \hspace{.5mm}\underline B$
and set $\xi_0 := \underline{B} \Omega \in \underline{A}^{-1}V$, so as to have $\underline A\xi_0=\Omega$.
By \eqref{eqn: S and T agree on Omega}, we then have
\[
\langle \underline{\hat{A}}[\nu[f]] \xi_0, \eta \rangle 
= \langle Y(\nu,f) \Omega, \underline{A}\eta \rangle
= \langle T(f) \Omega, \underline{A}\eta \rangle
= \langle \underline{A}^*T(f)\underline{A} \xi_0, \eta \rangle
\]
for all $\eta \in \underline{A}^{-1}V$.
As $\underline{A}^{-1}V$ is dense in $H_0$, it follows that
\[
\underline{\hat{A}}[\nu[f]] \xi_0 = \underline{A}^*T(f)\underline{A} \xi_0.
\]
The product $\underline{A}^* T(f) \underline{A}$ is densely defined by Lemma~\ref{lem basic property of Ann^(le r^(ell_0))}.
Since $\underline{A}\in \cA(J)$, and since $T(f)$ is affiliated with $\cA(J)$ (Proposition~\ref{prop: exist: stress energy tensor}),
the closure of $\underline{A}^* T(f) \underline{A}$ is also affiliated with $A(J)$.
Recall that we also have $\hat A[\nu[f]] \in \cA(J)$ by \eqref{eqn: Anuf is in A(J)}.

The vector $\xi_0$ is separating for $\cA(J)$ by the lemma below. %~\ref{lem: annulus times vacuum is separating} below.
It follows that, upon taking closures, we have $\underline{A}^*T(f)\underline{A} = \underline{\hat{A}}[\nu[f]]$.
By \eqref{eq: nuf versus Y(nu,f)}, the expressions
$\underline{A}^*T(f)\underline{A}$ and $\underline{A}^*Y(\nu,f)\underline{A}$ thus agree as sesquilinear form on $\underline{A}^{-1}V$,
and therefore
\begin{equation}\label{eqn: T(f)u = Y(nu,f)u}
T(f)u = Y(\nu,f)u
\end{equation}
for every $u \in V$. (This is a priori an equality of vectors in $\widehat V$, but the energy bounds for $T(f)$ ensure that this vector in fact lies in $H_0(\bbD)$.\footnote{See \cite[\S6]{CKLW18} for a general discussion of polynomial energy bounds, and \cite[Prop. 2.1]{GoWa85}\cite[\S6]{ToledanoLaredo99} for the specific case of $T(f).$})
The function $f$ in \eqref{eqn: T(f)u = Y(nu,f)u} was initially required to be supported in some interval $I$; using a partition of unity we have the same identity for every $f \in C^\infty(S^1)$.
Setting $f(z) = z^n$, we see that the modes of $T(z)$ and $Y(\nu,z)$ coincide, as desired:
\[
\qquad\qquad S_n=L_n,\qquad \forall n\in\bbZ.
\]

By Proposition~\ref{prop: exist: stress energy tensor}, the modes of $T(z)$ satisfy the Virasoro relations.
The same therefore holds for the modes $S_n$ of $Y(\nu,z)$.
By definition, $L_0$ is the grading operator for $V$, the same therefore holds for the zero-mode $S_0$ of $Y(\nu,z)$.
Finally, $Y(S_{-1}v,z) = Y(L_{-1}v,z) = \frac{d}{dz} Y(v,z)$ by Lemma~\ref{lem: T is L-1}.
This finished the proof that $\nu = L_{-2} \Omega$ is a conformal vector for $V$.
\end{proof}

We conclude this section with a lemma that was used in the above proof.

\begin{lem}\label{lem: annulus times vacuum is separating}
Let $J \subset S^1$ be an interval,
let $\underline{A} \in \Ann_c$, and let $\xi_0 := \underline{A}\Omega \in H_0(\bbD)$.    
Then $\xi_0$ is a separating vector for the action of $\cA(J)$ on $H_0$. 
\end{lem}

\begin{proof}
Applying the Riemann uniformization theorem to the disc $A \cup \bbD$, we may write $\underline{A} = \underline{\gamma} \underline{B}$, for some $\underline{B} \in \Univ$, and $\underline{\gamma}$ a lift of some diffeomorphism $\gamma\in\Diff(S^1)$.
By \eqref{eq:  A Omega=Omega}, we then have $\xi_0 = \underline{A}\Omega = \underline{\gamma}\Omega$.

If $x\xi_0 = 0$ for some $x \in \cA(J)$, then
    \[
\big(\underline{\gamma}^{-1}x\underline{\gamma}\big)\Omega = \underline{\gamma}^{-1}x\xi_0 = 0.
    \]
    As $\underline{\gamma}^{-1}x\underline{\gamma} \in \cA(\gamma^{-1}(J))$, it follows from the Reeh-Schlieder theorem that
$\underline{\gamma}^{-1}x\underline{\gamma} = 0$,
    and hence that $x=0$ as well.
\end{proof}

\subsection{Unitarity of the vertex algebra associated to a conformal net}\label{sec: unitarity}

We have shown that the vector space $V:=H_0^{f.e.}\subset H_0$ of finite energy vectors inside the vacuum sector of a conformal net
is a vertex operator algebra.
Its state-field correspondence is given by $Y(a,z)b = | a(z)b(0) \rangle$,
and its conformal net is $\nu=L_{-2}\Omega$. Our next goal is to equip $V$ with the structure of a unitary vertex operator algebra.

Recall (Definition~\ref{def: unitary vertex operator algebra}) that an inner product $\langle\,\,,\,\rangle$ and an involution $\Theta$ on a vertex operator algebra $V$ equip it with the structure of a \emph{unitary vertex operator algebra} if:\\
1. $\Theta$ is an anti-linear vertex operator algebra automorphism. Namely, it fixes $\Omega$ and $\nu$, and satisfies
$
\Theta Y(a,z) \Theta = Y(\Theta a, \overline{z})$.
2. The following equation holds:
$
\langle a, Y(u, z)b \rangle = \langle Y(e^{\bar zL_1}(-\bar z^{-2})^{L_0}\Theta u,\bar z^{-1})a,b \rangle$.

In our notations, the first condition becomes:
\begin{equation}\label{eqn: our notation Theta automorphism}
\Theta | a(z)b(0) \rangle =
| (\Theta a)(\bar{z}) (\Theta b)(0)\rangle,
\end{equation}
and the second one is (see Remark~\ref{rem: mysterious formula for adjoint of field}):
\begin{align}\label{eqn: our notation inner product invariant}
\big\langle a, | u(z) b(0) \rangle \big\rangle &= \big\langle | (\Theta u)(\bar z^{-1} ; w \mapsto (\bar z+w)^{-1}) a(0)\rangle, b \big\rangle
\\\notag
&=\big\langle | u(\bar z^{-1} ; w \mapsto \overline{(z+w)}^{-1}) a(0)\rangle, b \big\rangle,
\end{align}
where the last expression uses the anti-holomorphic local coordinates explained in Section~\ref{sec: Covariance for orientation reversing maps}.
Note that the sesquilinear pairing $\langle\,\,,\,\rangle:V\times V\to\bbC$ is extended to a pairing 
$V\times\widehat V\to\bbC$ in the left hand side of \eqref{eqn: our notation inner product invariant}, and to a pairing
$\widehat V\times V\to\bbC$ in the right hand side of \eqref{eqn: our notation inner product invariant}.

\begin{thm}\label{thm: unitarity of VOA}
Let $\cA$ be a conformal net, and let $V$ be the associated vertex operator algebra.
Then the inner product on $V$ inherited from $H_0$ together with the involution $\Theta: V \to V$ (introduced in \eqref{eq: first Theta} and further discussed in Section \ref{sec: Covariance for orientation reversing maps}) equip $V$ with the structure of a unitary vertex operator algebra.
\end{thm}

\begin{proof}
We first show that $\Theta$ is a vertex operator algebra automorphism.
It fixes the vacuum by construction.
To see that it fixes the conformal vector $\nu$, recall from \eqref{eq: Th L Th= L} that $\Theta L_{-2}\Theta=L_{-2}$.
It follows that
\[
\Theta\nu = \Theta L_{-2}\Omega = \Theta L_{-2} \Theta\Omega = L_{-2}\Omega = \nu.
\]
Finally, $\Theta$ satisfies condition \eqref{eqn: our notation Theta automorphism} by Lemma \ref{lem: Theta applied to point insertions}, and is therefore a vertex operator algebra automorphism.

To complete the proof, we must verify \eqref{eqn: our notation inner product invariant}.
Pick $r$ such that $0<r<\min(|z|,|z^{-1}|)$.
Let $D$ be the disc of radius $r$,
$D'$ the disc of radius $r^{-1}$, and $A := D' \setminus \mathring D$, with standard lift $\underline A\in \Ann_c$.
The map $w\mapsto \bar{w}^{-1}$ induces an isomorphism $A^\dagger \stackrel{\scriptscriptstyle\simeq}\to A$.
By Proposition~\ref{prop: annuli adjoints with point insertions}, the adjoint of $\underline{A}[u(z)]$ is therefore given by
$\underline{A}[u(\bar z^{-1} ; w \mapsto \overline{(z+w)}^{-1}]$.
Letting $U:H_0(D) \to H_0(\bbD)$ and $U':H_0(D') \to H_0(\bbD)$ be the unitaries associated to the maps $w \mapsto r^{-1} w$ and $w \mapsto r w$, we compute:
\begin{align*}
\big\langle a, | u(z) b(0) \rangle_{\widehat V} \big\rangle
&=
\big\langle r^{2L_0}a, | u(z) b(0) \rangle_{D'} \big\rangle_{H_0(D')}\\
&=
\big\langle r^{2L_0}a, A[u(z)] |  b(0) \rangle_{D} \big\rangle_{H_0(D')}\\
&=
\big\langle r^{2L_0}a, U^{\prime *} \underline{A}[u(z)] U |  b(0) \rangle_{D} \big\rangle_{H_0(D')}\\
&=
\big\langle 
\underline{A}[u(\bar z^{-1} ; w \mapsto \overline{(z+w)}^{-1}]
U' | r^{2L_0} a(0) \rangle_{D'}, U |  b(0) \rangle_{D} \big\rangle_{H_0(\bbD)}\\
&=
\big\langle 
\underline{A}[u(\bar z^{-1} ; w \mapsto \overline{(z+w)}^{-1}]
U | a(0) \rangle_{D'}, U' | r^{2L_0} b(0) \rangle_{D} \big\rangle_{H_0(\bbD)}\\
&=
\big\langle 
U^{\prime *}\underline{A}[u(\bar z^{-1} ; w \mapsto \overline{(z+w)}^{-1}]U
 | a(0) \rangle_{D},  r^{2L_0} b \big\rangle_{H_0(D')}\\
&=
\big\langle 
{A}[u(\bar z^{-1} ; w \mapsto \overline{(z+w)}^{-1}]
 | a(0) \rangle_{D},  r^{2L_0} b \big\rangle_{H_0(D')}\\
&=
\big\langle | u(\bar z^{-1} ; w \mapsto \overline{(z+w)}^{-1}) a(0) \rangle_{D'}, r^{2L_0} b \big\rangle_{H_0(D')}\\
&=\big\langle | u(\bar z^{-1} ; w \mapsto \overline{(z+w)}^{-1}) a(0)\rangle_{\widehat V}, b \big\rangle,
\end{align*}
where the third and seventh equalities hold by Lemma~\ref{lem: parametrized vs unparametrized annuli -- with point fields}.
\end{proof}

\section{The conformal net to VOA correspondence}\label{sec: smeared fields}

In their landmark paper \cite{CKLW18}, Carpi-Kawahigashi-Longo-Weiner introduced the notion of a \emph{strongly local} unitary VOA, and explained how to associate a conformal net $\cA_V$ to a strongly local VOA $V$.
We shall work here with a slight weakening of strong locality called \emph{AQFT-locality} (Definition~\ref{def: AQFT-local voa}), which was shown in \cite[Prop. 4.8]{RaymondTanimotoTener22} to be sufficient for constructing a conformal net $\cA_V$
(the construction in \cite{RaymondTanimotoTener22} is explained in \eqref{eqn: cn from smeared fields} below, and is formally identical to the one in \cite{CKLW18}).

The main result of this section is
Theorem~\ref{thm: inverse constructions}. In it, we prove that for every conformal net $\cA$, the associated unitary VOA $V_\cA$ (constructed in Section~\ref{sec: The vertex algebra associated to a conformal net}) is AQFT-local, and that $$\cA=\cA_{V_\cA}$$
(where $\cA_{V_\cA}$ denotes the conformal net associated by \eqref{eqn: cn from smeared fields} to the unitary VOA $V_\cA$).
We also prove that for every AQFT-local unitary VOA $V$ we have $$V=V_{\cA_V}$$
(where $V_{\cA_V}$ denotes the unitary VOA associated by the construction in Section~\ref{sec: The vertex algebra associated to a conformal net} to the conformal net $\cA_V$).

We now describe what it means for a unitary VOA to be AQFT-local.
Let $V$ be a unitary VOA.
Given $v\in V$ and $f \in C^\infty(S^1)$, generalising \eqref{eq:  Y(nu,f)}, we have the \emph{smeared field}
\[
Y(v,f)=\int_{S^1} f(z) Y(z^{L_0} v,z) \, \frac{dz}{2 \pi i z}
=\int_{S^1} f(z) z^{L_0} Y(v,1) z^{-L_0} \, \frac{dz}{2 \pi i z}:V\to \widehat V.
\]
By \cite{RaymondTanimotoTener22}, 
$Y(v,f)u\in H_0$ for every $u\in V$, and $Y(v,f)$ is a closable operator on $H_0$.

\begin{remark}
It was shown in \cite{RaymondTanimotoTener22} that the map $Y(v,f):V\to H_0$ lands in the subspace $H_0^\infty\subset H_0$ of smooth vectors for the operator $L_0$.
Assuming $V$ satisfies polynomial energy bounds, it was further shown in \cite{CKLW18} that this map extends by continuity to a map $Y(v,f):H_0^\infty\to H_0^\infty$ (see loc.~cit.~for the definition of polynomial energy bounds).
For our purposes, it will be enough to just know that $Y(v,f):V\to H_0$ is a closable operator on $H_0$.
\end{remark}

\begin{defn}
Let $V$ be a unitary VOA.
For a vector $v\in V$ and a smooth function $f$ on $S^1$, we define $Y(v,f)$ to be the closure of $\int_{z\in S^1} f(z) Y(z^{L_0} v,z) \frac{dz}{2 \pi i z}$ on the domain of finite energy vectors.
\end{defn}

Given a collection $\{A_i\}_{i\in\cI}$ of unbounded operators on some common Hilbert space,
we write $\mathsf{vN}\{A_i :i\in\cI\}$ for the von Neumann algebra they generate.
By definition, it is the smallest von Neumann algebra such that each $A_i$ is affiliated with it. Two closed unbounded operators $A$ and $B$ on the same Hilbert space are said to \emph{commute strongly} if %the von Neumann algebras 
$\mathsf{vN}\{A\}$ and $\mathsf{vN}\{B\}$ commute.
Following \cite{CKLW18}, a unitary VOA is called \emph{strongly local} if it satisfies polynomial energy bounds, and the operators $Y(v_1,f_1)$ and $Y(v_2,f_2)$ commute strongly whenever $f_1$ and $f_2$ have disjoint supports. We will use a slightly weaker definition, which does not include the energy polynomial bounds condition:

\begin{defn}\label{def: AQFT-local voa}
A unitary VOA is called \emph{AQFT-local} if the operators $Y(v_1,f_1)$ and $Y(v_2,f_2)$ commute strongly whenever $f_1$ and $f_2$ have disjoint supports.
\end{defn}

If $V$ is AQFT-local, then by \cite[Prop. 4.8]{RaymondTanimotoTener22} (proven first in \cite{CKLW18} under the assumption of polynomial energy bounds) the assignment
\begin{equation}\label{eqn: cn from smeared fields}
I \,\,\mapsto\,\, \mathsf{vN}\big\{ Y(v,f) : \mathrm{supp}(f) \subseteq I\big\} 
\end{equation}
defines a conformal net $\cA_V$ on the Hilbert space completion of $V$.

The key fact that we will need, proven in Proposition~\ref{pro: smeared fields are affiliated} below, is that when a unitary VOA $V$ is constructed from a conformal net as in Section~\ref{sec: The vertex algebra associated to a conformal net}, the smeared fields are affiliated with the local algebras of the conformal net. 
In order to prove this, we will need to establish the boundedness of certain products of smeared fields and annuli.

Recall that $r^{\ell_0}=\bbD\setminus r\mathring \bbD$, and that $\Ann^{\le r^{\ell_0}}$ denotes the set of annuli $A$ such that $\exists B \in \Ann$ s.t.~$r^{\ell_0}=AB$ (Definition~\ref{def; Ann rl0}).
Let $\Ann^{\le ?^{\ell_0}}:=\bigcup_{r>0}\Ann^{\le r^{\ell_0}}$.
An annulus $A$ lies in $\Ann^{\le ?^{\ell_0}}$
iff its inverse outgoing boundary parametrization
$\varphi_{out}^{-1}:\partial_{out} A\to S^1$
extends to a (holomorphic) embedding $A\hookrightarrow\bbD\setminus \{0\}$.
We identify $\Ann^{\le ?^{\ell_0}}$ with the set of such embedded annuli.
By Lemma~\ref{lem basic property of Ann^(le r^(ell_0))}, for any lift $\underline{A} \in \Ann_c$ of $A\in \Ann^{\le ?^{\ell_0}}$, the subspace $\underline{A}^{-1}V$ is dense in $H_0$,
hence $Y(v,f)\underline{A}$ is densely defined.
The next result describes circumstances under which this operator is bounded:

\begin{prop}\label{prop: smeared field times annulus is bounded}
Let $\cA$ be a conformal net, and $V$ its associated unitary VOA.
Let $A\in \Ann^{\le ?^{\ell_0}}$ be embedded in $\bbD$ as above,
let $\underline{A} \in \Ann_c$ be a lift,
and let $f \in C^\infty(S^1)$ be a function whose support is disjoint from $\partial_{in}A$:
\begin{equation}\label{eq: fig of A with supp f}
\begin{tikzpicture}[baseline={([yshift=-.5ex]current bounding box.center)},rotate=-35]
	\coordinate (a) at (120:1cm);
	\coordinate (b) at (240:1cm);
	\coordinate (c) at (180:.25cm);
% BIG DISC
	\fill[fill=red!10!blue!20!gray!30!white] (0,0) circle (1cm);
	\draw (0,0) circle (1cm);
% CURVED BOUNDARY REGION
	\fill[fill=white] (a)  .. controls ++(210:.6cm) and ++(90:.4cm) .. (c) .. controls ++(270:.4cm) and ++(150:.6cm) .. (b) -- ([shift=(240:1cm)]0,0) arc (240:480:1cm);
	\draw ([shift=(240:1cm)]0,0) arc (240:480:1cm);
	\draw (a) .. controls ++(210:.6cm) and ++(90:.4cm) .. (c);
	\draw (b) .. controls ++(150:.6cm) and ++(270:.4cm) .. (c);
% INTERVAL LABEL
	\draw (150:1.05cm) -- (150:1.15cm);
	\draw (210:1.05cm) -- (210:1.15cm);
	\draw (150:1.1cm) arc (150:210:1.1cm);
	\node[left] at (180:1.1cm) {\scriptsize{$\mathrm{supp}(f)$}};
% COORDINATE LABELS
%	\node at (a) {(a)};
%	\node at (b) {(b)};
	\node[left, xshift=-2.5, yshift=7, scale=.9] at (c) {$A$};
\end{tikzpicture}
\end{equation}
Then $Y(v,f)\underline{A}$ is bounded.

If $\underline{A}$
is furthermore localised in some interval $I$, then $Y(v,f)\underline{A} \in \cA(I)$.
\end{prop}
\begin{proof}
Assume without loss of generality that $v$ is homogeneous of conformal dimension $n$.
Pick $r<1$ close enough to $1$ so that the map $z \mapsto r z$ maps the support of $f$ into the interior of $A$.
The operator 
\[
T_r := \int_{\abs{z}=1} f(z)z^{n}\underline{A}[v(r z)]  \, \frac{dz}{2 \pi iz}
\]
is bounded.
We will first show in Step 1 below that $\lim_{r \to 1^-} T_r$ converges pointwise to $Y(v,f)\underline{A}$ on a dense subspace of $H_0$.
We will separately show in Step 2 that $\lim_{r \to 1^-} T_r$ exists in the norm topology on $B(H_0)$.
It will follow from these two results that $Y(v,f)\underline{A}$ is bounded.

\emph{Step 1:} Let $\xi \in \underline{A}^{-1}V$, and let $u = \underline{A}\xi \in V$.
We will show that $\lim_{r \to 1^-} T_r\xi = Y(v,f)u$.
As the range of $\underline{A}$ contains $V$, we may assume without loss of generality that $u$ is homogeneous with conformal dimension $m$.
Choose $\underline{B} \in \Ann_c$ such that $\underline{A} \, \underline{B} = s^{L_0}$ for some $0 < s < 1$, and note that by Lemma~\ref{lem basic property of Ann^(le r^(ell_0))} we may write $\xi = \underline{B}u'$ with $u' \in V$.
We compute
\begin{equation}\label{eqn: Avz xi}
\underline{A}[v(z)]\xi = \underline{A}[v(z)]\underline{B}u' = \underline{AB}[v(z)]u' = | v(z) (\underline{A} \underline{B} u')(0) \rangle_{\bbD}  = Y(v,z)u
\end{equation}
where the third equality is Lemma~\ref{lem: univalent acting on point insertions} and the fourth equality is the definition of the vertex algebra $V$ and the fact that $\underline{A} \underline{B} u' = u$.
Write $v_k = v_{(k+n-1)}$, and denote the Fourier series coefficients of $f$ by $\hat f(k)$, so that $\hat f(k)=\int_{|z|=1}z^{-k}f(z)\frac{dz}{2\pi iz}$.
Note that since $u$ has conformal dimension $m$, we have $v_k u = 0$ for $k > m$.
Combining \eqref{eqn: Avz xi} with the definition of $T_r$ we get
\[
T_r \xi = \int_{\abs{z}=1}f(z)z^nY(v, r z)u   \, \frac{dz}{2\pi iz} = \sum_{k=-\infty}^m r^{-n-k} \hat f(k)\, v_k u 
= Y(v,g_r)u,
\]
where $g_r(z) := \sum_{k=-\infty}^m r^{-n-k} \hat f(k) z^k \in C^\infty(S^1)$.
Note that $\lim_{r \to 1^{-}} g_r = f$ in $C^\infty(S^1)$.
The expressions $Y(v,g)u$ are continuous in $g$ by \cite[Prop. 2.9]{RaymondTanimotoTener22}. Hence
\[
\lim_{r \to 1^{-}} T_r \xi = \lim_{r \to 1^{-}} Y(v,g_r)u = Y(v,f)u = Y(v,f)\underline{A}\xi
\]
as claimed.

\emph{Step 2:} We will show that $\lim_{r \to 1^-} T_r$ exists in the norm topology using standard results about distributional boundary values of holomorphic functions.
In order to do this, we must show that $\norm{\underline{A}[v(z)]}$ has at most polynomial growth as $z$ tends to the support of $f$.
Let $\underline{A}^\dagger \in \Ann_c$ be the conjugate annulus, with underlying annulus $A^\dagger$ embedded in $\bbC$ as the image of $A$ under the reflection $z \mapsto \overline{z}^{-1}$.
By Proposition~\ref{prop: annuli adjoints with point insertions} and \eqref{eqn: AnnV dagger operation}, the adjoint of the bounded operator $\underline{A}[v(z)]$ on $H_0$ is given by 
\[
\underline{A}[v(z)]^* = \underline{A}^\dagger[v(z)^\dagger] = \underline{A}^\dagger[\tilde v(\overline{z})^{-1}]
\]
where $\tilde v \in V$ is as in \eqref{eqn: AnnV dagger operation} ($\tilde v$ is the image of $v$ under a certain antiholomorphic change of coordinate, which itself depends antiholomorphically on $z$).
In particular, this means that the non-zero homogeneous components of $\tilde v$ lie in $V(k)$ with $k \le n$.
Let $\hat A = A^\dagger \cup A$ and let $\underline{\hat A} = \underline{A}^\dagger \, \underline{A}$.
Then for $\xi \in H_0$ and $z \in \mathring A$ we have
\begin{equation}\label{eqn: norm Avz}
\big\|\underline{A}[v(z)]\big\|^2 = \norm{\underline{A}^\dagger[\tilde v (\overline{z}^{-1})] \, \underline{A}[v(z)]} = \big\| \underline{\hat A}\big[\tilde v (\overline{z}^{-1}) v(z)\big]\big\|
\end{equation}
where the second equality is Proposition~\ref{prop: annuli adjoints with point insertions}.

We now analyze the function $\underline{\hat A}[v(z)  \tilde v(\overline{w}^{-1})]$. This is a holomorphic $B(H_0)$-valued function of $z$ and $\tilde w := \overline{w}^{-1}$, defined on the domain
$\{(z,\tilde w)\in(\hat A)^\circ \times (\hat A)^\circ\,|\, z\not =\tilde w\}$.
%which is holomorphic in $z$ and $\tilde w := \overline{w}^{-1}$ away from the diagonal $z=\tilde w$.
By Lemma~\ref{lem: modes of fields in annuli}, its Laurent expansion about $z=\tilde w$ is given by
\[
\underline{\hat A}[v(z)  \tilde v (\tilde w)] = \sum_{m \in \bbZ} \underline{\hat A}[(v_{(m)}\tilde v)(\tilde w)](z-\tilde w)^{-m-1}.
\]
Since the non-zero homogeneous components of $\tilde v$ have conformal dimension at most $n$, we have $v_{(k)} \tilde v = 0$ for $k \ge 2n$. 
Hence $(z-\tilde w)^{2n} \underline{\hat A}[v(z) \tilde v(\tilde w)]$ is holomorphic on $(\hat A)^\circ \times (\hat A)^\circ$.
It follows that for every compact $K\subset (\hat A)^\circ$, there is a constant $C=C(K)$ such that
\[
\big\| \underline{\hat A}[v(z) \tilde v(\tilde w)] \big\| \le C \abs{z- \tilde w}^{-2n} \le C \max \{ (1-\abs{z})^{-2n}, (\abs{\tilde w}-1)^{-2n}\}
\]
whenever $w,z\in K$ satisfy $\abs{z} < 1 $ and $\abs{\tilde w} > 1$.
Combining with \eqref{eqn: norm Avz}, we see that there is a neighborhood of $\supp(f)$ in $A$ on which we have
\[
\big\|\underline{A}[v(z)]\big\| \le C(1-\abs{z})^{-n}.
\]
It now follows by standard results (included below as Lemma~\ref{lem: growth at boundary} for the convenience of the reader) that $\lim_{r \to 1^-} T_r$ exists with respect to the norm topology on $B(H_0)$, completing Step~2.

Finally, when $\underline{A}$ is localized in an interval $I$ we have $\underline{A}[v(z)] \in \cA(I)$ by Lemma~\ref{lem: eqn annulus with points insertions is local}.
Hence $T_r \in \cA(I)$, and this passes to the limit $Y(v,f)\underline{A}$, as desired.
\end{proof}

In the proof of Proposition~\ref{prop: smeared field times annulus is bounded}, we used the following result:

\begin{lem}\label{lem: growth at boundary}
Let $I\subset S^1$ be an interval, and let $U\subset \bbC$ be an open set containing $I$.
Suppose that $U^-:=\{z\in U:|z|<1\}$ is simply connected, and
that $F:U^- \to X$ is a holomorphic function with values in some Banach space $X$.
If $F$ satisfies the growth condition
\[
\|F(z)\| \le C \cdot (1-\abs{z})^{-n}
\]
for some $n\in\bbN$, then the limit
\[
\lim_{r\to 1^-} \int_{z \in I} f(z) F(r z) \frac{dz}{2 \pi i z} 
\]
exists in $X$ for every $f \in C^\infty(S^1)$ with support in $I$.
\end{lem}

\begin{proof}
Let $F^{(-k)}$ denote a $k$th antiderivative of $F$ on $U^{-}$.
The growth of the norm of $F^{(-k)}(z)$, as $|z|\to 1^-$, is bounded by the growth of the $k$th antiderivative of $|1-r|^{-n}$.
Taking $k=n+1$, we learn that $F^{(-n-1)}$ admits a continuous extension to $U^-\cup I$.
Let $\tilde f(z)=f(z)/(2 \pi i z)$, and write $\tilde f^{(k)} = \partial_z^k \tilde f$, where $\partial_z = (2 \pi i z)^{-1} \partial_\theta$.
We have
\begin{align*}
\int_{z \in I} f(z)\, F(r z) \,\frac{dz}{2\pi i z}
&=
\int_{z \in I} \tilde f(z) \, F(rz) \,  dz
=
r^{-n-1}  \int_{z \in I} \tilde f^{(n+1)}(z) \, F^{(-n-1)}(rz)  \,dz.
\end{align*}
The latter visibly admits a limit as $r\to 1^-$.
\end{proof}

\begin{rem}\label{rem: dom Yvf contains range A}
    If $X$ is a closed unbounded operator and $B$ is a bounded operator such that $XB$ is densely defined, then $XB$ is  closed on its natural domain $B^{-1} \mathrm{Dom}(X)$.
    Under the hypotheses of Proposition~\ref{prop: smeared field times annulus is bounded}, the operator $Y(v,f)\underline{A}$ is  closed and bounded, and hence defined on all of $H_0$.
    It follows that the domain of $Y(v,f)$ contains the image of $\underline{A}$.
\end{rem}

\begin{rem}
    By \cite[Prop. 2.9]{RaymondTanimotoTener22}, the domain of a product $Y(v_1,f_1) \cdots Y(v_k,f_k)$ of smeared fields contains $V$.
    Thus for the class of annuli $A$ considered in Proposition~\ref{prop: smeared field times annulus is bounded}, the operator $Y(v_1,f_1) \cdots Y(v_k,f_k)\underline{A}$ is densely defined.
    One may show, using the same argument as Proposition~\ref{prop: smeared field times annulus is bounded}, that if each test function $f_j$ 
    has its support disjoint from $\partial_{in}A$,
    %is supported in the boundary of the thick part of $A$ 
    then the operator $Y(v_1,f_1) \cdots Y(v_k,f_k)\underline{A}$ is bounded.
\end{rem}

We can now show that when $V$ is the unitary VOA constructed from a conformal net, the smeared fields of $V$ are affiliated with the local algebras of the conformal net.
\begin{prop}\label{pro: smeared fields are affiliated}
Let $\cA$ be a conformal net, and $v\in V$ a vector in the associated VOA.
Then for $f\in C^\infty(S^1)$ with support in $I\subset S^1$, the operator $Y(v,f)$ is affiliated with $\cA(I)$.
\end{prop}
\begin{proof}
By the additivity property of the conformal net (see \cite[\S3.1]{KawahigashiLectureNotes} and references therein), it suffices to show that $Y(v,f)$ is affiliated with $\cA(J)$, for every interval $J$ containing $I$ in its interior.
Fix such an interval $J$.

Choose $\underline{A} \in \Ann^{\le r^{L_0} }_c$ localized in $J$, with $I$ contained in the boundary of the thick part of $A$ as in \eqref{eq: fig of A with supp f}.
As $\underline{A} \in \cA(J)$, the closed operator $\underline{A}^{-1}$ is affiliated with $\cA(J)$.
By Proposition~\ref{prop: smeared field times annulus is bounded}, $Y(v,f)\underline{A} \in \cA(J)$ as well.
It follows that the closure of $(Y(v,f)\underline{A})\underline{A}^{-1}$ is also affiliated with $\cA(J)$.
%as the closure of the product of affiliated operators is again affiliated, provided that it is densely defined.
(This operator is densely defined by Lemma~\ref{lem: annuli have dense image} and Remark~\ref{rem: dom Yvf contains range A}.)

By Remark~\ref{rem: dom Yvf contains range A}, the domain of $Y(v,f)$ contains the image of $\underline{A}$,
so $(Y(v,f)\underline{A})\underline{A}^{-1}$ is precisely the restriction of $Y(v,f)$ to the image of $\underline{A}$.
By Lemma~\ref{lem basic property of Ann^(le r^(ell_0))}, the image of $\underline{A}$ contains $V$.
It follows that $Y(v,f)$ is the closure of $(Y(v,f)\underline{A})\underline{A}^{-1}$, hence affiliated with $\cA(J)$.
\end{proof}

We now establish the main result of this section, which says that the construction of conformal nets from AQFT-local unitary VOAs is inverse to our construction  of unitary VOAs from conformal nets:

\begin{thm}\label{thm: inverse constructions}
\begin{enumerate}[i)]
\item Let $\cA$ be a conformal net, and let $V$ be the associated unitary VOA.
Then $V$ is AQFT-local and the conformal net $\cA_V$ associated to $V$ coincides with $\cA$.

\item Let $V$ be a unitary VOA which is AQFT-local, and let $\cA_V$ be the associated conformal net on the Hilbert space completion of $V$.
Then the unitary VOA associated to $\cA_V$ is again $V$.
\end{enumerate}
\end{thm}
\begin{proof}
We first prove (i).
Let $\cA$ be a conformal net, and let $V$ be the associated unitary VOA.
Given $v_1,v_2 \in V$ and $f_1,f_2 \in C^\infty(S^1)$ supported in disjoint intervals $I_1$ and $I_2$, then $v_j[f_j]$ is affiliated with $\cA(I_j)$ by Proposition~\ref{pro: smeared fields are affiliated}. It follows that $v_1[f_1]$ and $v_2[f_2]$ strongly commute, and hence that $V$ is AQFT-local.
By \cite[Prop. 4.8]{RaymondTanimotoTener22}, %the construction 
\eqref{eqn: cn from smeared fields} yields a conformal net $\cA_V$ with the same vacuum Hilbert space $H_0$ as $\cA$, and which is diffeomorphism covariant with respect to the same representation of $\Diff_c(S^1)$ on $H_0$.
Since the local algebras $\cA_V(I)$ are generated by smeared fields supported in the interval $I$, and these smeared fields are affiliated with $\cA(I)$ (by Proposition~\ref{pro: smeared fields are affiliated}), it follows that $\cA_V(I) \subset \cA(I)$ for all intervals $I$.
Thus, by \cite[Lem. 2.2]{KaLo04}, we in fact have $\cA_V(I) = \cA(I)$, as required.

We now prove (ii).
Let $V$ be a unitary VOA which is AQFT-local, let $\cA:=\cA_V$ be the associated conformal net, and let $V_\cA$ be the unitary VOA associated to $\cA$.
Note that, by construction, the underlying vector spaces and inner products of $V$ and $V_\cA$ coincide; we must show that the VOA structures coincide as well.

By \cite[Prop. 4.8]{RaymondTanimotoTener22}, the conformal net $\cA$ is diffeomorphism covariant for the representation of $\Diff_c(S^1)$ obtained by integrating the stress-energy tensor of $V$.
The stress-energy tensor of $V_\cA$ is then obtained by differentiating that same representation of $\Diff_c(S^1)$ (Theorem~\ref{thm: unitary VOA from conformal net}). The stress-energy tensors and conformal vectors of $V$ and $V_\cA$ therefore agree.

The smeared fields associated to $V$ are affiliated with the local algebras of $\cA$ by construction.
The smeared fields associated to $V_\cA$ are affiliated with those same local algebras by Proposition~\ref{pro: smeared fields are affiliated}.
The VOAs $V$ and $V_\cA$ satisfy the hypothesis of \cite[Prop. 4.6]{RaymondTanimotoTener22}, and consequently $V = V_\cA$ as VOAs.
\end{proof}

\section{Modules}\label{sec: modules}

Let $\cA$ be conformal net, and let $V$ be the associated VOA. 
Let $K$ be a representation of $\cA$, and let $L_0$ be the generator of the rotation group $\widetilde{\mathrm{Rot}}(S^1)$, which is a positive operator on $K$ \cite[Prop. 3.8]{Weiner06}. 
Throughout this section, we assume that $L_0$ acts on $K$ with discrete spectrum and finite-dimensional eigenspaces, and we let $M:=K^{f.e.}$ be the algebraic direct sum of the eigenspaces of $L_0$.
In this section, we show that $M$ naturally admits the structure of a $V$-module (Definition~\ref{def: VOA module}).

Recall that by Proposition \ref{prop: annuli adjoints with point insertions}, the semigroup $\tAnn^{V}_c$ acts on $K$, and note that for any $r>0$ the operator $r^{L_0}$ acts invertibly on $M$.
For $0 < r < R$, let $A_{R,r} \in \tAnn_c$ be the annulus $\{r \le \abs{z} \le R\}$, equipped with its standard boundary parametrization and standard $z$-lift.

\begin{defn}\label{def: module mode}
For any vector $v\in V$ and integer $n\in\bbZ$, we define the $n$\textsuperscript{th} mode $v^M_{(n)}$ to be the operator which acts on $a \in M$ by
\begin{equation}\label{eqn: module mode definition}
v_{(n)}^M a = \oint_{\abs{z}=\rho} z^n A_{1,r}[v(z)]r^{-L_0} a \,  \, \tfrac{dz}{2\pi i},
\end{equation}
where $0<r<\rho<1$.
Since $z \mapsto A_{1,r}[v(z)]$ is holomorphic (Corollary~\ref{cor: embedded point insertions are holomorphic}), this operator is independent of the choices of $r$ and $\rho$.
\end{defn}

The operator $v_{(n)}^M$ leaves $M$ invariant:

\begin{lem}\label{lem: basic property of modes on M}
    Let $K$ be a representation of $\cA$ with the proprety that $L_0$ has discrete spectrum and finite-dimensional eigenspaces, and let $M=K^{f.e.}$.
    Then we have
    \begin{enumerate}[i)]
    \item The modes \eqref{eqn: module mode definition} define operators $v_{(n)}^M:M \to M$ which satisfy
    \begin{equation}\label{eqn: comm rel L0 vKm}
        [L_0, v^{M}_{(n)}] = (m-n-1)v^M_{(n)}.
    \end{equation}

    \item For $R > \rho > r > 0$, and $a \in M$, we have
    \[
    v_{(n)}^M a =  \oint_{\abs{z}=\rho} z^n\,R^{L_0} A_{R,r}[v(z)] r^{-L_0}a \,  \tfrac{dz}{2\pi i}.
    \]
    \end{enumerate}
\end{lem}
\begin{proof} (i). Suppose $v\in V$ has conformal dimension $m$. For $1 > s > 0$ and $a \in M$ we have 
\[
s^{L_0} A_{1,r}[v(z)]r^{-L_0}a = A_{s^{-1},r}[v(z)]r^{-L_0}a =  s^m A_{1,rs}[v(sz)]r^{-L_0} a =  s^m A_{1,rs}[v(sz)](rs)^{-L_0} (s^{L_0}a).
\]
Integrating this identity around the contour $\abs{z}=\rho$ with $1 > \rho > r$ yields
\begin{align*}
s^{L_0} v^M_{(n)}a &= s^m \oint_{\abs{z}=\rho} z^n A_{1,rs}[v(sz)](rs)^{-L_0} (s^{L_0}a)  \, \tfrac{dz}{2\pi i}\\ 
& = s^{m-n-1} \oint_{\abs{z}=s\rho} z^n A_{1,rs}[v(z)](rs)^{-L_0} (s^{L_0}a)  \, \tfrac{dz}{2\pi i}\\
& = s^{m-n-1} v^M_{(n)}s^{L_0}a.
\end{align*}
Thus, if $a$ is an eigenvector of $L_0$ with eigenvalue $k$, then $v^M_{(n)}a$ is an eigenvector with eigenvalue $k+m-n-1$.
The identity \eqref{eqn: comm rel L0 vKm}
immediately follows, as does the fact that $v_{(n)}^M$ leaves $M$ invariant.

(ii). Suppose $a\in M$ has conformal dimension $k$, and $v\in V$ has conformal dimension $m$.
Then, by part (i), we have:
\begin{align*}
R^{-L_0} v_{(n)}^M a &= R^{-k-m+n+1} \oint_{\abs{z}=\rho/R} z^n A_{1,r/R}[v(z)] \big(\tfrac{r}{R}\big)^{-L_0}a \,  \tfrac{dz}{2\pi i}\\
&= R^{n+1} \oint_{\abs{z}=\rho/R} z^n A_{R,r}[v(R z)] r^{-L_0}a \, \tfrac{dz}{2\pi i}\\
&= \oint_{\abs{z}=\rho} z^n A_{R,r}[v(z)] r^{-L_0}a \, \tfrac{dz}{2\pi i}
\end{align*}
as required.
\end{proof}

Let $Y^M:V \to \mathrm{End}(M)[[z^{\pm 1}]]$ be given by
\begin{equation}\label{eqn: def of YM}
    Y^M(v,z) := \sum_{n \in \bbZ} v_{(n)}^M z^{-n-1}.
\end{equation}
Letting $\widehat M:= \prod_{n \in \bbR} M(n)$ denote the algebraic completion of $M$ (where $M(n) = \mathrm{ker}(L_0-n)$), the expression $Y^M(v,z)a$ may be interpreted as a holomorphic function $\bbC^\times \to \widehat M$.

If $R > r > 0$ and $z \in \mathring A_{R,r}$, then 
\begin{equation}\label{eqn: YMvzu in terms of annuli}
Y^M(v,z)a = R^{L_0}A_{R,r}[v(z)]r^{-L_0}a
\end{equation}
since both sides define $\widehat M$-valued holomorphic functions on $\mathring A_{R,r}$ and their Laurent series coefficients coincide by Lemma~\ref{lem: basic property of modes on M}.
When $R=1$, the vector given by \eqref{eqn: YMvzu in terms of annuli} lies in $K$.
In the special case $K=H_0$ we have $M=V$ and it follows from \eqref{eqn: YMvzu in terms of annuli} and Lemma~\ref{lem: univalent acting on point insertions} that $Y^M(v,z)$ agrees with the vertex algebra action $Y(v,z)$.

\begin{prop}\label{prop: Kfe weak module}
    Let $\cA$ be a conformal net, let $V$ be the associated unitary VOA, and let $K$ be a representation of $\cA$.
    Assume that $L_0$ acts on $K$ with discrete spectrum and finite-dimensional eigenspaces.
    Then the action $Y^M$ defined in \eqref{eqn: def of YM} makes $M:=K^{f.e.}$ into a unitary weak\footnote{We will later show, in Corollary~\ref{cor: Kfe strong module}, that $M$ is in fact a strong module.} $V$-module.
\end{prop}
\begin{proof}
 The vacuum axiom of a weak module is clear, as is the lower truncation condition (by \eqref{eqn: comm rel L0 vKm}).
 By \cite[Prop 4.2.1 and 4.3.1]{LepowskyLi04}, the remaining two axioms follow from the Borcherds identities:
 \begin{align}
\nonumber \sum_{j = 0}^\infty \binom{m}{j} \big(u_{(n+j)}v\big)^M_{(m+k-j)}& = \sum_{j=0}^\infty (-1)^j \label{eqn: borcherds in proof} \binom{n}{j} u_{(m+n-j)}^M v_{(k+j)}^M\\
& -\sum_{j=0}^\infty(-1)^{j+n} \binom{n}{j} v_{(n+k-j)}^M u_{(m+j)}^M,
\end{align}
which we now establish.
Let $f(z,w)=z^k w^m (w-z)^n$, and let $u,v \in V$.
Fix a small $\epsilon > 0$, let $r=1-4\epsilon$, and consider $(z,w)\mapsto A_{1,r}[v(z) u(w)] f(z,w)$, which is a holomorphic $B(K)$-valued function of the variables $z,w \in \mathring A_{1,r}$ defined away from the diagonal $z=w$.
Let $a \in M$.
A standard contour integration argument (c.f. \cite[\S 3.3]{FrenkelBenZvi04}) yields
\begin{align}\label{eq: Borcherds identity (contours)}
\nonumber &\oint_{\abs{z}=1-2\epsilon}  \oint_{w\;\!:\;\!\abs{w-z}=\epsilon}
A_{1,r}[v(z)u(w)]  a \, f(z,w)  \,\tfrac{dw}{2\pi i}\,  \tfrac{dz}{2\pi i} \,\,= \\
\nonumber & \quad \oint_{\abs{z}=1-3\epsilon} \oint_{\abs{w}=1-\epsilon}
A_{1,1-2\epsilon}[u(w)] A_{1-2\epsilon,r}[v(z)]  a \, f(z,w) \, \tfrac{dw}{2\pi i} \, \tfrac{dz}{2\pi i}\\
 &\,\, -
\oint_{\abs{z}=1-\epsilon} \oint_{\abs{w}=1-3\epsilon}
 A_{1,1-2\epsilon}[v(z)]A_{1-2\epsilon,r}[u(w)]  a \, f(z,w)  \,\tfrac{dw}{2\pi i}  \,\tfrac{dz}{2\pi i}.
\end{align}

We begin by investigating the first term on the right-hand side of \eqref{eq: Borcherds identity (contours)}.
Let $a' = r^{L_0}a$.
Expanding $(w-z)^n=\sum_{j=0}^\infty (-1)^{j}\binom{n}{j}w^{n-j}z^j$ in the domain $\abs{w} > \abs{z}$ we have
\begin{align*}
    \oint_{\abs{z}=1-3\epsilon} &\oint_{\abs{w}=1-\epsilon}
A_{1,1-2\epsilon}[u(w)] A_{1-2\epsilon,r}[v(z)] a \, f(z,w) \, \tfrac{dw}{2\pi i} \, \tfrac{dz}{2\pi i}\\
& = 
 \oint_{\abs{w}=1-\epsilon}
A_{1,1-2\epsilon}[u(w)] \oint_{\abs{z}=1-3\epsilon} A_{1-2\epsilon,r}[v(z)] a \, f(z,w) \, \tfrac{dz}{2\pi i} \, \tfrac{dw}{2\pi i}\\
&=  \oint_{\abs{w}=1-\epsilon}
A_{1,1-2\epsilon}[u(w)]w^{m} \left[ \sum_{j=0}^\infty (-1)^{j}\binom{n}{j} w^{n-j}\oint_{\abs{z}=1-3\epsilon} A_{1-2\epsilon,r}[v(z)]a z^{k+j} \tfrac{dz}{2\pi i} \right] \tfrac{dw}{2\pi i}\\
&=  
\sum_{j=0}^\infty (-1)^j\binom{n}{j} \oint_{\abs{w}=1-\epsilon} A_{1,1-2\epsilon}[u(w)]w^{m+n-j} \,  (1-2\epsilon)^{-L_0} \,  v^M_{(k+j)}a' \, \tfrac{dw}{2\pi i}\\
&=  \sum_{j=0}^\infty (-1)^j\binom{n}{j} u^M_{(m+n-j)}v^M_{(k+j)}a',
\end{align*}
where we note that the sum in $j$ is finite, and the third equality uses Lemma~\ref{lem: basic property of modes on M}(ii).
The second term on the right-hand side of \eqref{eq: Borcherds identity (contours)} may be computed similarly, and we see that the right-hand side of \eqref{eq: Borcherds identity (contours)} is equal to the right-hand side of \eqref{eqn: borcherds in proof} applied to the vector $a'$.

We now turn to the left-hand sides of \eqref{eq: Borcherds identity (contours)} and \eqref{eqn: borcherds in proof}.
Expanding $w^m = (z + (w-z))^m=\sum_{j=0}^\infty\binom{m}jz^{m-j}(w-z)^j$ in the domain $\abs{z} > \abs{w-z}$ and applying Lemma~\ref{lem: modes of fields in annuli} yields 
\[
\oint_{w\;\!:\;\!\abs{w-z}=\epsilon} A_{1,r}[v(z)u(w)] z^k w^m (w-z)^n  \,\tfrac{dw}{2\pi i}
= A_{1,r}\big[ \sum_{j\ge 0} {\textstyle\binom{m}{j}} z^{m+k-j} \big(u_{(n+j)} v\big)(z) \big].
\]
Hence the left-hand side of \eqref{eq: Borcherds identity (contours)} also agrees with the left-hand side of \eqref{eqn: borcherds in proof} applied to the vector $a'$.
As $a' \in M$ was arbitrary, we have established the Borcherds identity, and conclude by \cite[Prop 4.2.1 and 4.3.1]{LepowskyLi04} that $M$ is a weak module.

We now show that $M$ is unitary.
Pick $v \in V$ and $a,b \in M$.
Fix $0<r<1$, and let $z \in \mathring A_{1,r}$.
By Proposition~\ref{prop: annuli adjoints with point insertions} and \eqref{eqn: AnnV dagger operation}, the adjoint of $A_{1,r}[v(z)]$ acting on $K$ is given by 
\[
A_{1,r}[v(z)]^* = A_{r^{-1},1}[v(z)^\dagger].
\]
%where $v(z)^\dagger = \tilde v(\bar z^{-1})$ for $\tilde v = e^{\bar zL_1}(-\bar z^{-2})^{L_0}\Theta v$.
We thus have:
\begin{align*}
\langle Y(v,z)a, b\rangle_K
&= \langle A_{1,r}[v(z)]r^{-L_0}a, b \rangle_K\\
&= \langle r^{-L_0}a, A_{r^{-1},1}[v(z)^\dagger])b \rangle_K\\
&= \langle a, r^{-L_0} A_{r^{-1},1}[v(z)^\dagger])b \rangle_{M,\widehat M}\\
&= \langle a, Y(\tilde v, \bar z^{-1})b \rangle_{M,\widehat M}
\end{align*}
where the first and final equalities follow from \eqref{eqn: YMvzu in terms of annuli}.
For this argument, we have assumed that $z \in \mathring A_{1,r}$, but the identity holds for all $z \in \bbC^\times$ by holomorphicity, and is also an identity of formal series.
We conclude that $M$ is a unitary weak module.
\end{proof}

At this point, we have two potentially different representations of the Virasoro algebra on $M$: one obtained by differentiating the positive-energy representation of $\tDiff(S^1)$ as in Proposition~\ref{prop: exist: stress energy tensor}, and one obtained from the modes $\nu^M_{(n)}$ of the conformal vector $\nu \in V$.
We will see in Corollary~\ref{cor: Virasoro representations on M agree} that these representations coincide, and it will follow (Corollary~\ref{cor: Kfe strong module}) that $M$ is a (strong) $V$-module.
In the following arguments, 
whenever $\tAnn_c$ acts on $K$, this always refers to the representation which arises from the action of $\cA$ on $K$, as opposed to one arising from the conformal vector.

Consider the smeared fields
\[
Y^M(v,f) := \int_{\abs{z}=1} f(z) Y^M(z^{L_0}v,z) \tfrac{dz}{2 \pi i z}.
\]
A priori, this formula defines a map $M \to \widehat M$, where $\widehat M$ is the algebraic completion of $M$. 
However, an argument identical to that in \cite[\S2.4]{RaymondTanimotoTener22} (specifically, Prop. 2.9 and Cor. 2.11 of loc.~cit.) shows that $Y^M(v,f)$ maps $M$ into $K$, and defines a closable operator on the domain $M$.

\begin{defn}\label{defn: module smeared field}
    For a vector $v \in V$ and a smooth function $f$ on $S^1$, we define $Y^M(v,f)$ to be the closure of $\int_{\abs{z}=1} f(z) Y^M(z^{L_0}v,z) \tfrac{dz}{2 \pi i z}$ on the domain $M$.
\end{defn}

Recall that when the test function $f$ is supported in an interval $I \subset S^1$, the smeared field $Y(v,f)$ is affiliated with $\cA(I)$ (Proposition~\ref{pro: smeared fields are affiliated}).
There is a natural way to apply the representation $\pi_I$ to $Y(v,f)$ and obtain an operator $\pi_I(Y(v,f))$, which is then affiliated with the von Neumann algebra $\pi_I(\cA(I))$.
Since $\cA(I)$ is a type III factor (\cite[Prop. 1.2]{GuidoLongo96}), the representation $\pi_I$ may be written $\pi_I = \operatorname{Ad} U$ for a unitary operator $U:H_0 \to K$.
The operator $\pi_I(Y(v,f))$ is then given by $UY(v,f)U^*$.
This operator does not depend on the choice of the implementing unitary $U$ because $Y(v,f)$ is affiliated with $\cA(I)$.

\begin{thm}\label{thm: CN rep applied to smeared field is smeared field}
Let $(K,\pi)$ be a representation of the conformal net $\cA$
on which $L_0$ acts with discrete spectrum and finite-dimensional eigenspaces.
And let $M$ be the associated unitary weak $V$-module (as in Proposition~\ref{prop: Kfe weak module}), where $V$ is the unitary VOA associated to $\cA$.
\begin{enumerate}[i)]
\item 
Let $A \in \Ann^{\le ?^{\ell_0}}$ (where $\Ann^{\le ?^{\ell_0}}:=\bigcup_{r>0}\Ann^{\le r^{\ell_0}}$ as in Proposition~\ref{prop: smeared field times annulus is bounded}), with its standard embedding $A\subset \bbC$ mapping $\partial_{out}A$ to $S^1$, and let 
$\underline{A} \in \tAnn_c$ be a $z$-lift of $A$.
Write $\underline{A}$ for the corresponding operator on $H_0$, and $\underline{A}^K$ for the action of $\underline{A}$ on $K$ induced by the representation of $\cA$.
Let $\xi \in K$ be such that $\underline{A}^K\xi \in M$.
Then, for $v \in V$ and $z \in \mathring{A}$, we have
\begin{equation}\label{eq: A^K[v(z)]=Y^M}
\underline{A}^K[v(z)]\xi = Y^M(v,z)\underline{A}^K\xi.
\end{equation}
If $\underline{A}$ is localised in $I\subset S^1$, we also have
\[
\pi_I(\underline{A}[v(z)])\xi = Y^M(v,z)\underline{A}^K\xi.
\]
\item If $f \in C^\infty(S^1)$ is supported in an interval $I\subset S^1$, then
\begin{equation}\label{eqn: piI Yvf = YMvf}
\pi_I(Y(v,f)) = Y^M(v,f).
\end{equation}
Moreover, the identity \eqref{eqn: piI Yvf = YMvf} uniquely characterises the $V$-module action $Y^M$.
\end{enumerate}
\end{thm}
\begin{proof}
(i)
We first establish \eqref{eq: A^K[v(z)]=Y^M} when $\underline{A}=A_{1,r}$ for some $0 < r < 1$, in which case $\xi \in M$.
By definition, $Y^M(v,z) = \sum_{n\in\bbZ} v^M_{(n)} z^{-n-1}$, where
\[
v^M_{(n)} r^{L_0} = \oint_{\abs{z}=\rho} A^K_{1,r}[v(z)] z^n \, \tfrac{dz}{2\pi i}.
\]
Equation \eqref{eq: A^K[v(z)]=Y^M} then follows from the standard formula for the Laurent coefficients of the holomorphic function $A^K_{1,r}[v(z)]\xi$.

For general $\underline{A} \in \Ann^{\le r^{\ell_0}}$, we have $r^{L_0}=\underline{A} \, \underline{B}$ for some $\underline{B} \in \tAnn_c$. 
By Lemma~\ref{lem basic property of Ann^(le r^(ell_0))} we may write $\xi = \underline{B}^K u$ for some $u \in M$.
Hence
\[
\underline{A}^K[v(z)]\xi = (\underline{A} \underline{B})^K[v(z)]u = Y^M(v,z)\underline{A}^K\underline{B}^K u = Y^M(v,z)\underline{A}^K \xi,
%\underline{A}^K[v(z)]\underline{B}^Ku = Y^M(v,z)\underline{A}\underline{B}^Ku,
\]
which establishes \eqref{eq: A^K[v(z)]=Y^M}.

We finish part (i) by noting that when $\underline{A}$ is localised in an interval $I$, we then have $\pi_I(\underline{A}[v(z)])=\underline{A}^K[v(z)]$ by Lemma~\ref{lem: eqn annulus with points insertions is local}.

(ii)
Given $f \in C^\infty(S^1)$ supported in $I$, choose an interval $J$ containing $I$ in its interior, and an annulus $\underline{A} \in \tAnn_c(J) \cap \Ann^{\le ?^{\ell_0}}$ such that $I$ is disjoint from $\partial_{in}A$.

Assume without loss of generality that $v$ is homogeneous of conformal dimension $n$.
As in the proof of Proposition~\ref{prop: smeared field times annulus is bounded}, let $T_r \in \cA(J)$ be the bounded operator
\[
T_r = \int_{\abs{z}=1} f(z) z^{n} \underline{A}[v(r z)] \, \tfrac{dz}{2 \pi i z} .
\]
From the proof of that proposition, $\lim_{r \to 1^{-}} T_r = Y(v,f)\underline{A}$ with convergence in the norm topology.
Similarly, set
\[
T^K_r := \pi_J(T_r) = \int_{\abs{z}=1} f(z) z^{n} \pi_J(\underline{A}[v(r z)]) \, \tfrac{dz}{2 \pi i z}.
\]
We then have
\begin{equation}\label{eqn: TKr converges SOT}
\lim_{r \to 1^{-}} T^K_r = \pi_J(\lim_{r \to 1^{-}} T_r) = \pi_J(Y(v,f)\underline{A}) = \pi_J(Y(v,f)) \pi_J( \underline{A})
\end{equation}
with convergence again in the norm topology.
In particular, this means that $\pi_J(Y(v,f))\pi_J(\underline{A})$ is bounded and that the range of $\pi_J(\underline{A})$ is contained in the domain of $\pi_J(Y(v,f))$ (as in Remark~\ref{rem: dom Yvf contains range A}).

When $\pi_J(\underline{A})\xi \in M$,
by part (i) and the fact that $\underline{A}^K = \pi_J(\underline{A})$ (Proposition~\ref{lem: local annuli lie in local algebra}), we also have 
\[
T_r^K \xi = \int_{\abs{z}=1}  f(z) z^{n} Y^M(v,rz)  \pi_J(\underline{A})\xi \, \tfrac{dz}{2 \pi i z}.
\]
Combining this with \eqref{eqn: TKr converges SOT} yields
\begin{equation}\label{eqn: other TKr formula}
\lim_{r \to 1^{-}} T^K_r \xi = \lim_{r \to 1^{-}} \int_{\abs{z}=1} f(z) z^{n} Y^M(v,rz) \pi_J(\underline{A})\xi \, \tfrac{dz}{2 \pi i z} = Y^M(v,f)\pi_J(\underline{A})\xi.
\end{equation}
The second equality is clear when the limit is interpreted in $\widehat M$, but also holds in the Hilbert space topology of $K$ as the limit converges a priori in this topology by \eqref{eqn: TKr converges SOT}.

Combining \eqref{eqn: TKr converges SOT} with \eqref{eqn: other TKr formula}, we see that $\pi_J(Y(v,f))$ and $Y^M(v,f)$ agree on $M$.
By construction, $M$ is a core for $Y^M(v,f)$. So it suffices to show that $M$ is also a core for $\pi_J(Y(v,f))$.
The range of $\underline{A}$ contains $V$
by Lemma~\ref{lem basic property of Ann^(le r^(ell_0))}, and is thus a core for $Y(v,f)$.
Hence the range of $\pi_J(\underline{A})$ is a core for $\pi_J(Y(v,f))$, and it suffices to show that the range of $\pi_J(\underline{A})$ is contained in the domain of the closure of $\pi_J(Y(v,f))|_M$.

Fix $\eta = \pi_J(\underline{A})\xi$ in the range of $\pi_J(\underline{A})$.
As $\pi_J(\underline{A})^{-1} M$ is dense, we may choose a sequence $\xi_n \in K$ such that $\pi_J(\underline{A})\xi_n \in M$ and $\xi_n \to \xi$.
Since $\pi_J(Y(v,f))\pi_J(\underline{A})$ is bounded, we have
\[
\pi_J(Y(v,f))\pi_J(\underline{A})\xi_n \to \pi_J(Y(v,f))\pi_J(\underline{A})\xi
\]
and thus $\eta = \lim \pi_J(\underline{A})\xi_n$ is in the domain of the closure of $\pi_J(Y(v,f))|_M$, as required.
Hence $\pi_J(Y(v,f))$ and $Y^M(v,f)$ agree on a common core, and thus are equal.
Finally, by the definition of a representation of a conformal net, we have $\pi_I(Y(v,f)) = \pi_J(Y(v,f))$, which completes the proof of \eqref{eqn: piI Yvf = YMvf}.

It remains to show that \eqref{eqn: piI Yvf = YMvf} uniquely characterises $Y^M$.
As argued above, when $\supp(f)\subset I$, %is supported in an interval $I$, 
the domain of $\pi_I(Y(v,f))$ contains 
the range of $\pi_J(\underline{A})$,
and in particular contains $M$.
Suppose $v\in V$ has conformal dimension $n$, and write $e_m$ for the function $e_m(z) = z^m$.
The mode $v^M_{(m)}$ is given $v^M_{(m)} = Y^M(v,e_{m-n+1})$.
Choose intervals $I_j$ whose interiors cover $S^1$, and choose a partition of unity $\chi_j$ subordinate to this cover.
Then by \eqref{eqn: piI Yvf = YMvf} we have
\[
v^M_{(m)} = \sum_j Y^M\big(v,e_{m-n+1} \chi_j\big)\big|_M= \sum_j \pi_{I_j}\big(Y(v,e_{m-n+1}\chi_j)\big)\big|_M,
\]
from which we conclude that \eqref{eqn: piI Yvf = YMvf} characterises the action $Y^M$, as claimed.
\end{proof}

\begin{cor}\label{cor: Virasoro representations on M agree}
Let $\cA$, $V$, $K$, $M$ be as in Theorem~\ref{thm: CN rep applied to smeared field is smeared field}.
The following two representations of the Virasoro algebra on $M$ coincide:
\begin{enumerate}
\item The representation of $\Vir_c$ obtained by differentiating the representation of $\tDiff_c(S^1)$ associated with the conformal net representation $K$.
\item The representation of $\Vir_c$ obtained from the modes of $Y^M(\nu,z)$.
\end{enumerate}
\end{cor}

\begin{proof} 
The Virasoro representation 
\[
T^K:C^\infty(S^1)\partial_z \to \{\text{operators on } K\}
\]
obtained by differentiating the representation of $\tDiff_c(S^1)$ is given by $T^K(f) = \pi_I(Y(\nu,f))$ when $f$ is supported in $I\subset S^1$ \cite[Prop. 2.1 and Eqn (31)]{Carpi04} (see also \cite[Thm. 6.4.6]{Weiner05}).
Hence, $T^K(f) = Y^M(\nu,f)$ by Theorem~\ref{thm: CN rep applied to smeared field is smeared field}.
It follows from a partition of unity that for any $f \in C^\infty(S^1)$ the operators $T^K(f)$ and $Y^M(\nu,f)$ agree on $M$.
In particular, the corresponding representations of the Virasoro algebra agree.
\end{proof}

\begin{cor}\label{cor: Kfe strong module}
Let $\cA$ be a conformal net and let $K$ be a representation of $\cA$ such that $L_0$ acts on $K$ with discrete spectrum and finite-dimensional eigenspaces.
Let $V$ be the VOA associated to $\cA$.
Then the construction of Proposition~\ref{prop: Kfe weak module} makes $M:=K^{f.e.}$ into a (strong) unitary $V$-module.
\end{cor}
\begin{proof}
  It was shown in Proposition~\ref{prop: Kfe weak module} that $M$ is a weak unitary module.
  Moreover, it was shown that the grading by the eigenspaces of the operator $L_0$ (obtained by differentiating the action of rotation) is compatible with the weak module structure.
  By Corollary~\ref{cor: Virasoro representations on M agree}, this operator $L_0$ agrees with the zero mode of the field $Y^M(\nu,z)$, and so $M$ is a strong module.
\end{proof}

\begin{rem}
    The results of this section go through unchanged if one only assumes that $L_0$ has discrete spectrum but potentially infinite-dimensional eigenspaces. One obtains an action of $V$ on $M$ which satisfies all of the requirements of a strong module, except for the finite-dimensionality of its $L_0$-eigenspaces.
    
    In fact, the results also go through essentially unchanged even if the operator $L_0$ is not assumed to have discrete spectrum, except that one then only obtains an admissible (or $\bbN$-gradable weak) module\footnote{An admissible $V$-module is a weak module $M$ such that there exists a grading $M=\bigoplus_{n=0}^\infty M[n]$ with the property that $v_{(n)}M[k] \subset M[k+m-n-1]$ when $v$ has conformal dimension $m$. This grading is not assumed to come from the conformal vector of $V$, and the weight spaces are not required to be finite-dimensional.}.
    The vector space $M$ which carries an action of $V$ is now defined to be $M:=\bigcup_{N > 0} p_{[0,N]}K$, where $p_S$ is the spectral projection of $L_0$ corresponding to the set $S \subset [0,\infty)$.
    The necessary $\bbN$-grading is given by $M[n]=p_{[n,n+1)}K$.
    The conclusions of Proposition~\ref{prop: Kfe weak module} and Theorem~\ref{thm: CN rep applied to smeared field is smeared field} remain valid, with only minor modification to the proofs.
\end{rem}

%\bibliographystyle{alphaabbr}
%\bibliography{ffbib} 
\def\lfhook#1{\setbox0=\hbox{#1}{\ooalign{\hidewidth
  \lower1.5ex\hbox{'}\hidewidth\crcr\unhbox0}}}

\end{document}